\documentclass{article}

\usepackage{amsthm,amsmath,amsfonts,amssymb}
\usepackage{natbib}
\usepackage[colorlinks,citecolor=blue,urlcolor=blue, linkcolor=blue]{hyperref}
\usepackage{graphicx}
\usepackage{booktabs}				  %
\usepackage{subcaption}			     %
\usepackage{bbm}					     %
\usepackage{dutchcal}				   %
\usepackage{enumitem}       	     %
\usepackage{fontawesome5}   	%
\usepackage{xcolor}         			 %
\usepackage{authblk}					%
\usepackage{orcidlink}					%
\usepackage{geometry}
\newtheorem{Theor}{Theorem}[section]
\newtheorem{Prop}{Proposition}[section]

\newtheorem{Lem}{Lemma}[section]
\newtheorem{Corol}{Corollary}[section]

\numberwithin{equation}{section}

\newcommand{\Dstat}{D}

\DeclareMathOperator*{\argmax}{\arg\max}

\newcommand{\N}{\mathbb N}
\newcommand{\R}{\mathbb R}

\newcommand{\E}{{\rm E}}

\newcommand{\n}{^{(n)}}

\defcitealias{Vaart2023}{VW23}

\title{Möbius-Invariant Goodness-of-Fit Tests for the Spherical Cauchy Model}
\author[1]{Diego Bolón}
\author[2]{Davy Paindaveine} 
\affil[1]{Department of Probability and Mathematical Statistics, Charles University, Prague (Czech Republic)} 
\affil[2]{ECARES and Department of Mathematics, Université libre de Bruxelles, Brussels (Belgium)} 
\date{}

\begin{document}
\maketitle

\begin{abstract}
	We introduce a class of goodness-of-fit tests for the spherical Cauchy model on the unit hypersphere. The proposed procedures exploit the invariance of the spherical Cauchy family under M\"{o}bius transformations: after estimating the parameter by the sample M\"{o}bius mean, the observations are transformed to approximate spherical uniformity, and a projection-based uniformity statistic is applied to the resulting sample. We show that, under mild conditions, the resulting tests are exactly distribution-free under the null hypothesis, so that exact critical values can be arbitrarily well approximated by Monte Carlo simulation. We study the M\"{o}bius mean as a population functional, establish its existence and uniqueness under mild conditions, and prove equivariance and asymptotic linearity of its empirical counterpart, which coincides with the spherical Cauchy maximum likelihood estimator. We derive the asymptotic null distribution of the test statistic and show that it coincides with that of the underlying uniformity statistic after removing the degree-one spherical-harmonic component, which corresponds to the tangent space of the spherical Cauchy model. We establish consistency against fixed alternatives and characterize local powers through the spherical-harmonic decomposition of contiguous alternatives. Monte Carlo experiments demonstrate the finite-sample accuracy of the asymptotic approximations and the empirical power of the proposed tests. A real data example is treated.
\end{abstract}
\begin{flushleft}
\end{flushleft}

\section{Introduction}
\label{sec:Intro}

Directional data arise whenever the magnitude of an observation is irrelevant or unavailable and only its direction carries information; see, e.g., \citet{MardiaJupp2000} and \citet{LeyVerdebout2017} for comprehensive accounts of directional statistics. Examples occur in meteorology, structural geology, palaeomagnetism, astronomy, biology, and bioinformatics, among many other fields. Such observations naturally take values on a unit sphere and require statistical procedures that respect the geometry of their support. 

Parametric directional models play an important role in this setting: they provide concise descriptions of location and concentration, enable likelihood-based inference and prediction, and often constitute convenient building blocks for more elaborate regression, clustering, or time-series models. As in Euclidean statistics, however, the usefulness of the conclusions drawn from a parametric analysis depends on whether the assumed distribution adequately describes the data. A formal goodness-of-fit assessment is therefore an essential complement to estimation and inference rather than a secondary modeling step. The need for a model-specific goodness-of-fit assessment is particularly clear for spherical data. Different rotationally symmetric distributions may have similar centers and mean resultant lengths while allocating substantially different amounts of probability away from their modal direction. Consequently, choosing one such family rather than another can affect estimated probabilities of spherical regions, prediction regions, and the interpretation of concentration. %

Classical goodness-of-fit work for directional data has concentrated mainly on testing uniformity; see \citet{GarciaPortuguesVerdebout2018} for an extensive review. Model-specific procedures have focused on a small number of prominent families, especially the Fisher--von Mises and Watson families; see \citet{FisherBest1984}, \citet{MardiaHolmesKent1984}, \citet{LockhartStephens1985}, and \citet{BestFisher1986}. By comparison, procedures applicable to other parametric families on spheres of arbitrary dimension remain much more limited; see, for example, \citet{Boulerice1997}, \citet{Boente2014}, and \citet{Ebner2024}. This scarcity is especially pronounced for composite hypotheses, in which the parameter indexing the null family must be estimated from the same observations used to assess fit.

In this paper, we consider the spherical Cauchy model studied by \citet{Kato2020}, which is indexed by a vector~$\phi$ in the open unit ball, whose direction determines the modal direction and whose norm controls concentration. It extends the well-known wrapped Cauchy distribution from the circle to spheres of arbitrary dimension and provides an alternative to the Fisher--von Mises model. Although the two models are unimodal and rotationally symmetric, their concentration profiles can differ markedly: at comparable mean resultant lengths, the spherical Cauchy distribution can allocate more probability near the mode and the antimode and less over intermediate regions of the sphere. Assessing whether this model is compatible with the data is therefore a genuine modeling question that cannot be reduced to testing rotational symmetry (see \citet{GarciaPortuguesPaindaveineVerdebout2020} and the references therein) or estimating a single concentration index.

One of the main appeals of the spherical Cauchy family of distributions is geometric. It is closed under a special type of transformations of the sphere that are called \emph{Möbius transformations}. More precisely, the spherical Cauchy family forms a transformation model under the action of these transformations.
The uniform distribution is the central member of this model: every spherical Cauchy distribution is obtained by applying an appropriate Möbius transformation to a uniform random direction, and conversely, a suitable inverse transformation maps a spherical Cauchy random direction to uniformity. As shown by \citet{Kato2020}, this structure yields an efficient random-generation algorithm, pivotal quantities, simple expressions for probabilities of spherical regions, and a connection with a multivariate $t$ distribution through stereographic projection. These properties make the spherical Cauchy family both mathematically tractable and potentially useful in applications. Yet the same transformation structure has not, to our knowledge, been exploited to develop a goodness-of-fit procedure for this model.

General-purpose goodness-of-fit methods for directional data are available, but none takes advantage of this transformation structure. The smooth tests of \citet{Boulerice1997} embed the null family into a larger exponential family and apply a score-test construction. This approach is flexible and can diagnose particular departures, but it requires the practitioner to select an embedding family and a finite collection of components; its power is consequently tied to that choice, and adapting the required orthonormal construction to the spherical Cauchy family is not immediate. The density-based tests of \citet{Boente2014} compare a directional kernel density estimator with a smoothed parametric fit. They apply to general parametric families, but require the selection of a bandwidth, have a null approximation whose convergence may be slow, and are typically implemented through a parametric bootstrap. More recently, \citet{Ebner2024} proposed a characteristic-function approach in which an artificial sample is generated from the fitted null distribution. This yields a broadly applicable test, but introduces an artificial-sample size and a kernel tuning parameter, and calibration under a composite hypothesis again relies on a parametric bootstrap. Thus, although these methods are valuable in their generality, a test tailored to the spherical Cauchy transformation model can offer substantial advantages in invariance and calibration.

Our starting point is the observation that every spherical Cauchy distribution can be transformed into spherical uniformity by an appropriate Möbius transformation. If~$\phi$ were known, goodness-of-fit could therefore be assessed by transforming the observations and applying a test of uniformity. Since~$\phi$ is unknown, we replace it with a Möbius-equivariant estimator and apply a rotation-invariant uniformity statistic to the transformed sample. For this step, we use the projection-based class of \citet{GarciaPortugues2023}, which aggregates, over all projection directions, weighted quadratic discrepancies between the projected empirical distributions and their distributions under spherical uniformity. This class includes Cramér--von Mises-type procedures, admits computable pairwise-kernel representations, and yields omnibus tests under mild conditions on the weighting measure.

Estimating~$\phi$ makes the transformed observations neither independent nor exactly uniform, even under the null hypothesis. The Möbius structure nevertheless yields a stronger result than asymptotic pivotality.
We use the Möbius transformations to introduce a new location/scale parameter for directional distributions that we call the \textit{Möbius mean}. This new parameter characterizes the distribution under the spherical Cauchy model, and its sample version coincides with the maximum likelihood estimator of $\phi$ considered by \citet{Kato2020}. We establish existence and uniqueness of the population and sample Möbius means under mild conditions and prove equivariance of the sample version. Combined with the rotational invariance of the uniformity statistic, this implies that the finite-sample null distribution of the resulting goodness-of-fit statistic does not depend on~$\phi$. Exact critical values can therefore be arbitrarily well approximated by Monte Carlo simulation for each sample size, dimension, and weighting measure, rather than obtained by parametric bootstrap for every data set.

Our theoretical contributions go well beyond exact distribution-freeness. We derive a Bahadur representation and the asymptotic distribution of the sample Möbius mean under general non-atomic distributions, as well as the asymptotic null distribution of the proposed test statistics. Remarkably, parameter estimation removes precisely the degree-one spherical-harmonic component of the underlying uniformity statistic, which corresponds to the tangent space of the spherical Cauchy family at uniformity. We also establish consistency against fixed alternatives, show that the Cramér--von Mises version is omnibus over a broad class of absolutely continuous distributions, and characterize the behavior of the proposed tests under contiguous alternatives. Their local power is determined by the spherical-harmonic components orthogonal to the nuisance tangent space, explaining why Möbius-invariant tests cannot detect, at the $n^{-1/2}$ scale, alternatives moving only along a degree-one direction.

The remainder of the paper is organized as follows. Section~\ref{sec:preliminaries} reviews the spherical Cauchy family and the Möbius transformations used throughout the paper. Section~\ref{secinvariance} introduces the proposed tests and establishes their invariance and exact distribution-freeness. Section~\ref{sec:estimation} develops the population and sample M\"{o}bius means and studies the asymptotic properties of the sample version. Sections~\ref{secNull} and~\ref{secNonNull} derive, respectively, the asymptotic null distribution and the behavior of the tests under fixed and local alternatives, and provide numerical illustrations supporting the asymptotic results. Section~\ref{secsimus} focuses on the Cram\'{e}r--von Mises version of the proposed tests: it investigates its finite-sample power against several alternatives and compares it with the characteristic-function procedure of \citet{Ebner2024}. Section~\ref{sec:realdata} illustrates the practical relevance of our proposal with a real-data example. Finally, Section~\ref{secwrapup} concludes. Proofs are collected in the appendix.

For the sake of convenience, we introduce some notation that will be used throughout the paper.
We write
$
\mathbb S^d
=
\{x\in\mathbb R^{d+1}:\|x\|^2=x^Tx=1\}
$
for the unit sphere in~$\mathbb R^{d+1}$ and
$
\mathbb B^{d+1}
=
\{x\in\mathbb R^{d+1}:\|x\|<1\}
$
for the corresponding open unit ball. We denote as~$\sigma_d$ the surface area measure on~$\mathbb S^d$, and write~$\nu_d$ for the corresponding uniform probability measure, that is,
$
\nu_d
=
\{\Gamma\big(\frac{d+1}{2}\big)/(2\pi^{(d+1)/2})\}
\sigma_d$.
We write $\mathrm{SO}(d+1)$ for the group of $(d+1)\times(d+1)$ rotation
matrices and $I=I_{d+1}$ for the $(d+1)$-dimensional identity matrix. Finally, $\mathbbm{1}[A]$ will stand for the indicator function of the set or condition~$A$.

\section{Preliminaries}
\label{sec:preliminaries}

This section collects the main facts about spherical Cauchy distributions and Möbius transformations that will be used throughout the paper. We first recall the definition and basic features of the spherical Cauchy family, and then describe the action of the Möbius group on both the sample space and the parameter space. The key consequence is that the spherical Cauchy family forms a transformation model with spherical uniformity as its central distribution. This property will provide the basis for the goodness-of-fit construction in Section~\ref{secinvariance}.

\subsection{Spherical Cauchy distributions}

The spherical Cauchy family, introduced and studied in detail by
\citet{Kato2020}, is indexed by a parameter
$\phi\in\mathbb B^{d+1}$. Its density with respect to the surface area
measure~$\sigma_d$ is
\begin{equation}
	\label{eq:densCauchy}
	f(x;\phi)
	:=
	\frac{\Gamma\big(\frac{d+1}{2}\big)}
	{2\pi^{(d+1)/2}}
	\bigg(
	\frac{1-\|\phi\|^2}{\|x-\phi\|^2}
	\bigg)^d,
	\qquad
	x\in\mathbb S^d.
\end{equation}
We write $X\sim\mathrm{SC}(\phi)$ when $X$ has density
\eqref{eq:densCauchy}. For~$\phi=0$, this yields the uniform distribution~$\nu_d$ on~$\mathbb S^d$. For~$\phi\neq 0$, the larger~$\|\phi\|$, the higher the concentration about the modal direction~$\phi/\|\phi\|$.

\subsection{Möbius transformations}

We next recall the Möbius transformations underlying the transformation
structure of the spherical Cauchy family. Let
$\bar{\mathbb R}^{d+1}:=\mathbb R^{d+1}\cup\{\infty\}$ denote the
one-point compactification of~$\mathbb R^{d+1}$. For
$R\in\mathrm{SO}(d+1)$ and
$\phi\in\mathbb R^{d+1}$ with $\|\phi\|\neq1$, define the Möbius transformation 
\begin{equation}
	\mathcal{M}_{R, \phi} (x)
	:=
	R
	\Bigg[
	\frac{1 - \| \phi \|^2 }{\| \frac{x}{\| x \|^2 } + \phi \|^2}
	\bigg(
	\frac{x}{\| x \|^2} + \phi \bigg)
	+
	\phi
	\Bigg]
	,
	\quad
	x \in \R^{d+1} \setminus \bigg\lbrace 0, - \frac{\phi}{\| \phi \|^2}  \bigg\rbrace.
\end{equation}
This transformation can be extended to a mapping from $\bar{\R}^{d+1} = \R^{d+1} \cup \{\infty \}$ onto itself by further letting 
$$
\mathcal{M}_{R, \phi} (0) 
:=
R \phi, 
\quad
\mathcal{M}_{R, \phi} (- \phi/\| \phi \|^2) 
:=
\infty,
\quad
\textrm{and}
\quad
\mathcal{M}_{R, \phi} (\infty) 
:=
R \phi/\| \phi \|^2 
.
$$
Then, \mbox{$\mathcal{M}_{R, \phi}: \bar{\R}^{d+1} \to \bar{\R}^{d+1}$} is a bijection (see Theorem~2.1 in \citealp{Kato2020}). The collection of all M\"{o}bius transformations $\mathcal{G} = \{ \mathcal{M}_{R, \phi} \colon R \in \mathrm{SO} (d + 1), \phi \in \R^{d+1} \setminus \mathbb{S}^d \}$ is a group under composition (see Theorem~2.2 in \citealp{Kato2020}). The identity element is $\mathcal{M}_{I, 0}$, and the inverse of $\mathcal{M}_{R, \phi}$ is $\mathcal{M}_{R^T, - R \phi}$.

One easily sees that any M\"{o}bius transformation maps the unit hypersphere $\mathbb{S}^d$ onto itself. In fact, the restriction of $\mathcal{M}_{R, \phi}$ to $\mathbb{S}^d$, say $\tilde{\mathcal{M}}_{R, \phi}$, has some interesting properties:
\begin{enumerate}
	\item This restriction has the form
	\begin{equation}
		\tilde{\mathcal{M}}_{R, \phi} (x)
		:=
		R
		\bigg[
		\frac{1 - \| \phi \|^2 }{\| x + \phi \|^2}
		(
		x + \phi )
		+
		\phi
		\bigg]
		,
		\quad
		x \in \mathbb{S}^{d}.
	\end{equation}
	\item $\tilde{\mathcal{M}}_{R, \phi}: \mathbb{S}^{d} \to \mathbb{S}^{d}$ is a bijection.
	\item Given $\phi \neq 0$,
	\begin{equation}
		\label{eq:phi>1}
		\tilde{\mathcal{M}}_{I, \phi/ \| \phi \|^2} (x) = T_{\phi} \big\{ \tilde{\mathcal{M}}_{I, \phi} (x) \big\},
	\end{equation}
	where $T_{\phi} = 2 \phi \phi^T / \| \phi \|^2 - I$ is the reflection with respect to $\phi$-axis. 
\end{enumerate}

The spherical Cauchy family is invariant with respect to M\"{o}bius transformations:
\begin{equation}
	X \sim \mathrm{SC} (\phi) \Rightarrow \mathcal{M}_{R, \varphi} (X) \sim \mathrm{SC} \big( \mathcal{M}_{R, \varphi} (\phi) \big)
\end{equation}
(see Theorem~2.3 in \citealp{Kato2020}). Since $\mathrm{SC} (0)=\mathrm{Unif} (\mathbb{S}^d)$, we have in particular that
\begin{equation}
	\label{eq:CauchyUnif}
	X \sim \mathrm{SC} (\phi) \Leftrightarrow \mathcal{M}_{I, -\phi} (X) \sim \mathrm{Unif} (\mathbb{S}^d)
\end{equation}
for any $\phi\in\mathbb{B}^{d+1}$.

\section{Test construction and invariance arguments}
\label{secinvariance}

Let $X$ be a random vector on~$\mathbb{S}^d$ with distribution ${\rm P}$. We aim to test the null hypothesis that~$F$ is a distribution from the spherical Cauchy family. That is to say, our objective is to construct a test for the null hypothesis
\begin{equation}
	\label{eq:GoFCauchy}
	H_0: {\rm P} \in \{ \mathrm{SC} (\phi) \colon \phi \in \mathbb{B}^{d + 1} \}.
\end{equation}
The property \eqref{eq:CauchyUnif} allows us to address this testing problem in a natural manner. We know that $X_1, \ldots, X_n$ is an i.i.d.~sample from $\mathrm{SC} (\phi)$ if and only if $\mathcal{M}_{I, -\phi} (X_1), \ldots, \mathcal{M}_{I, -\phi} (X_n)$ is an i.i.d.~sample from the uniform distribution on the hypersphere. Consequently, if $\hat{\phi}_n$ is a consistent estimator of $\phi$, the transformed sample $\mathcal{M}_{I, -\hat{\phi}_n} (X_1), \ldots, \mathcal{M}_{I, -\hat{\phi}_n} (X_n)$ should approximately behave like a uniform sample and, therefore, we can test \eqref{eq:GoFCauchy} by applying a uniformity test to the transformed sample. If both $\hat{\phi}_n$ and the uniformity test fulfill mild equivariance/invariance properties, then the resulting goodness-of-fit test will be distribution-free under the null hypothesis. More precisely, we have the following results.

\begin{Theor}
	\label{th:invariance}
	Let $\hat{\phi}_n \colon \mathbb{S}^{d} \times \ldots \times \mathbb{S}^{d} \to \mathbb{B}^{d+1}$ be an estimator of~$\phi$ that is equivariant  with respect to the M\"{o}bius transformations, in the sense that
	\begin{equation}
		\hat{\phi}_n \big( \mathcal{M}_{R,\varphi} (x_1), \ldots, \mathcal{M}_{R, \varphi} (x_n) \big)
		=
		\mathcal{M}_{R, \varphi} \big( \hat{\phi}_n (x_1, \ldots, x_n) \big)
	\end{equation}
	for all $x_1, \ldots x_n \in \mathbb{S}^d$, $\varphi \in \mathbb{B}^{d+1}$ and $R \in \mathrm{SO} (d + 1)$. 
	Let $\mathcal{U}_n: \mathbb{S}^d \times \ldots \times \mathbb{S}^d \to \R$ be a function that is invariant under rotations, that is,
	$$\mathcal{U}_n (R x_1, \ldots, R x_n) = \mathcal{U}_n ( x_1, \ldots,  x_n)$$
	for all $x_1, \ldots x_n \in \mathbb{S}^d$ and $R \in \mathrm{SO}(d+1)$.
	Consider the mapping $\Dstat_n \colon \mathbb{S}^d \times \ldots \times \mathbb{S}^d \to \R$ defined as
	$$
	\Dstat_n (x_1, \ldots, x_n)
	=
	\mathcal{U}_n \big( \mathcal{M}_{I, - \hat{\phi}_n (x_1, \ldots, x_n)} (x_1), \ldots, \mathcal{M}_{I, - \hat{\phi}_n (x_1, \ldots, x_n)} (x_n) \big).
	$$
	Then, $\Dstat_n$ is invariant under M\"{o}bius transformations, that is,
	\begin{equation}
		\Dstat_n (x_1, \ldots, x_n)
		=
		\Dstat_n \big( \mathcal{M}_{R, \varphi} (x_1), \ldots, \mathcal{M}_{R, \varphi} (x_n) \big)
	\end{equation}
	for all $x_1, \ldots x_n \in \mathbb{S}^d$, $\varphi \in \mathbb{B}^{d+1}$ and $R \in \mathrm{SO} (d + 1)$.
\end{Theor}

\begin{Corol}
	\label{CorDF}
	Let~$X_1, \ldots, X_n \sim \mathrm{SC} (\phi)$, with~$\phi \in \mathbb{B}^{d+1}$, and let the assumptions of Theorem~\ref{th:invariance} hold. Then, the distribution of
	$\Dstat_n (X_1, \ldots, X_n)$
	does not depend on~$\phi$.
\end{Corol}

In practice, we need to choose a test of uniformity on~$\mathbb{S}^d$. To fix ideas, we consider the projection-based tests from 
\cite{GarciaPortugues2023}. When based on a random sample~$Z_1,\ldots,Z_n$, these tests reject the null hypothesis of uniformity for large values of
$$
P^W_{n}
:=
P^W_{n}(Z_1,\ldots,Z_n)
:=
n
\int_{\mathbb{S}^d}
\int_{-1}^{1}
\{
F_{n,u}(s)-F_d(s)
\}^2
\,
dW(F_d(s))
\,
d\nu_d(u)
,
$$
where~$F_{n,u}$ is the empirical cumulative distribution function of~$u^TZ_1,\ldots,u^TZ_n$, 
\begin{equation}
	\label{FdExpr}
	s\mapsto 
	F_d(s)
	:=
	\frac{\Gamma(\frac{d+1}{2})}{\sqrt{\pi}\Gamma(\frac{d}{2})}
	\int_{-1}^s
	(1-t^2)^{(d-2)/2}
	\,
	dt,
	\qquad
	s\in[-1,1],
\end{equation}
is the cumulative distribution function of~$u^T V$ when~$V\sim \mathrm{Unif} (\mathbb{S}^d)$, and $W$ is a nonnegative $\sigma$-finite Borel measure on~$[0,1]$. These test statistics can be computed without evaluating the involved integrals as
\begin{equation}
	\label{Edu}
	P^W_{n}
	=
	\frac{1}{n}
	\sum_{i,j=1}^n
	\psi(Z_i,Z_j)
	,
\end{equation}
with
\begin{equation}
	\label{DefPsi}
	\psi(x,y)
	=
	\int_{\mathbb{S}^d}
	\int_{-1}^1
	\{\mathbbm{1}[ u^T x \leq s ]-F_d(s)\}
	\{\mathbbm{1}[ u^T y \leq s ]-F_d(s)\}
	\,
	dW(F_d(s))
	d\nu_d(u)
	;
\end{equation}
see Equation~(14) in \cite{GarciaPortugues2023}. 

When based on a random sample~$X_1,\ldots,X_n$ from a distribution~$\rm P$ on~$\mathbb{S}^d$, the resulting goodness-of-fit tests thus reject the null hypothesis~$H_0$ in~(\ref{eq:GoFCauchy}) for large values of 
\begin{equation}
	\label{TestStatistDn}
	\Dstat_n
	:=
	\Dstat_n(X_1,\ldots,X_n)
	:=
	\frac{1}{n}
	\sum_{i,j=1}^n
	\psi\big(\mathcal{M}_{I, -\hat{\phi}_n} (X_i),\mathcal{M}_{I, -\hat{\phi}_n} (X_j)\big)
	,
\end{equation}
where~$\hat{\phi}_n$ is an equivariant estimator of~$\phi$. Since the test statistics in~(\ref{Edu}) are invariant under rotations, Corollary~\ref{CorDF} implies that~$\Dstat_n$ is then exactly distribution-free under the null hypothesis. To implement the corresponding tests, critical values can thus be approximated arbitrarily
accurately by Monte Carlo simulation; see Section~\ref{sec:critvalues} for details.

To obtain a well-defined test, we need to have a suitable equivariant estimator of~$\phi$. We construct such an estimator in the next section.

\section{Population and sample M\"{o}bius means}
\label{sec:estimation}

Let ${\rm P}$ be a probability measure on~$\mathbb{S}^d$. We say that $\phi \in \mathbb{B}^{d+1}$ is a \emph{M\"{o}bius mean} of ${\rm P}$ if
\begin{equation}
	\phi
	\in
	\argmax_{\varphi\in\mathbb{B}^{d+1}} 
	\E \bigg[
	\log \bigg( \frac{1 - \| \varphi \|^2}{\| X - \varphi \|^2 }\bigg)
	\bigg]
	\quad
	\bigg(
	=
	\argmax_{\varphi\in\mathbb{B}^{d+1}} 
	\E [
	\log f(X;\varphi)
	]
	\bigg)
	,
\end{equation}
where~$X$ has distribution~${\rm P}$; see~\eqref{eq:densCauchy}.
Letting~$G (\varphi) = \E[ \log ( 1 - \| \varphi \|^2 ) - \log  \| X - \varphi \|^2]$, a routine application of the Lebesgue Dominated Convergence Theorem entails that the map~$G$ is differentiable on~$\mathbb{B}^{d+1}$ with gradient
\begin{equation*}
	\nabla G (\varphi) = \frac{2}{1 - \| \varphi \|^2} \E[ \mathcal{M}_{I, -\varphi} (X)].
\end{equation*}
Therefore, a M\"{o}bius mean must satisfy the gradient condition
\begin{equation}
	\label{eq:Mobiusmean}
	\E \big[ \mathcal{M}_{I, -\phi} (X) \big] = 0
\end{equation}
(this condition, which shows that~$\phi$ is the parameter value that makes the M\"{o}bius transformations of $X$ have zero mean, explains the terminology \emph{M\"{o}bius mean}). In principle, there is no guarantee that this gradient condition is a sufficient condition for~$\phi$ to be a M\"{o}bius mean. The following result takes care of this issue under a very mild assumption on~$\rm P$, and also settles the question of existence and uniqueness of the M\"{o}bius mean.

\begin{Prop}
	\label{prop:Mobunique}
	Let ${\rm P}$ be a probability measure on~$\mathbb{S}^d$ and~$X$ be a random vector with distribution~${\rm P}$. Assume that~$\mathrm{P}[X=x]<\frac12$ for all~$x\in \mathbb{S}^{d}$. Then, ${\rm P}$ has one and only one M\"{o}bius mean, which is the unique solution of~\eqref{eq:Mobiusmean}.
\end{Prop}

Obviously, this result in particular applies in the spherical Cauchy model. Note that in this model, the M\"{o}bius mean is Fisher consistent in the sense that if~${\rm P}=\mathrm{SC} (\phi)$, then the unique M\"{o}bius mean of~${\rm P}$ is~$\phi$ (this directly follows from~\eqref{eq:CauchyUnif}).

We turn to the sample case. Given an i.i.d.~sample $X_1, \ldots, X_n$ from~${\rm P}$, a natural estimator of the M\"{o}bius mean of~${\rm P}$ is the M\"{o}bius mean~$\hat{\phi}_n$ of~${\rm P}_n$, where~${\rm P}_n$ is the empirical probability measure associated with the sample at hand. More specifically,
\begin{equation}
	\label{eq:sampleMobius}
	\hat{\phi}_n
	=
	\hat{\phi}_n (X_1, \ldots, X_n)
	=
	\argmax_{\varphi \in\mathbb{B}^{d+1}} 
	\frac{1}{n} \sum_{i = 1}^n 
	\log \bigg( \frac{1 - \| \varphi \|^2}{\| X_i - \varphi \|^2 }\bigg)
\end{equation}
(uniqueness is discussed in Proposition~\ref{Prop:ML} below). Note that $\hat{\phi}_n$ coincides with the maximum likelihood estimator of $\phi$ when $X_1,\ldots,X_n$ are \mbox{i.i.d.} with distribution~$\mathrm{SC} (\phi)$. The corresponding gradient condition now writes
\begin{equation}
	\label{eq:meanzero}
	\sum_{i = 1}^{n} \mathcal{M}_{I, - \hat{\phi}_n} (X_i)
	=
	0.
\end{equation}

The following result then easily follows from Proposition~\ref{prop:Mobunique}.

\begin{Prop}
	\label{Prop:ML}
	(i) Let $x_1,\ldots,x_n\in\mathbb{S}^d$, with~$n \geq 3$, be pairwise different. Then, the empirical probability measure associated with~$x_1,\ldots,x_n$ has one and only one M\"{o}bius mean, which is the unique solution of~\eqref{eq:meanzero} (with~$x_i$ rather than~$X_i$).
	In particular, (ii) if~$X_1,\ldots,X_n$, with~$n \geq 3$, form a random sample from a probability measure~${\rm P}$ that is non-atomic, then, with probability one, the corresponding empirical probability measure~${\rm P}_n$ has one and only one M\"{o}bius mean~$\hat{\phi}_n$, which is the unique solution of~\eqref{eq:meanzero}.
\end{Prop}

For any~$n\geq 3$, the M\"{o}bius mean~$\hat{\phi}_{n}$ is therefore well-defined when the sample belongs to the collection---${\cal X}^n$, say---of samples~$(x_1,\ldots,x_n)$ with pairwise different entries in~$\mathbb{S}^d$. For such samples, $\hat{\phi}_{n}$ is also equivariant with respect to M\"{o}bius transformations.

\begin{Theor}
	\label{TheorIvarianceMobMean}
	For any~$n \geq 3$, the M\"{o}bius mean~$\hat{\phi}_{n}$ is equivariant with respect to M\"{o}bius transformations on~${\cal X}^n$. In other words, for any~$n\geq 3$, we have
	\begin{equation*}
		\hat{\phi}_n \big( \mathcal{M}_{R,\varphi} (x_1), \ldots, \mathcal{M}_{R, \varphi} (x_n) \big)
		=
		\mathcal{M}_{R, \varphi} \big( \hat{\phi}_n (x_1, \ldots, x_n) \big)
	\end{equation*}
	for all $(x_1, \ldots x_n) \in {\cal X}^n$, $\varphi \in \mathbb{B}^{d+1}$ and $R \in \mathrm{SO} (d + 1)$.
\end{Theor}

The next result shows that $\hat{\phi}_n$ is a $\sqrt{n}$-consistent and asymptotically normal estimator of the population M\"{o}bius mean and that it admits a Bahadur representation.

\begin{Prop}
	\label{PropBahadur}
	Let ${\rm P}$ be a non-atomic probability measure on~$\mathbb{S}^d$ and denote as~$\phi$ its M\"{o}bius mean. For~$n\geq 3$, let~$X_1,\ldots,X_n$ form a random sample from~${\rm P}$ and denote as~$\hat{\phi}_n$ the corresponding (almost surely unique) sample M\"{o}bius mean. Then, denoting as~$\stackrel{\mathcal{L}}{\to}$ convergence in distribution,
	\begin{eqnarray}
		\sqrt{n} ( \hat{\phi}_n - \phi)
		&	=&
		- \frac{1}{\sqrt{n}}W_{\phi}^{-1} \sum_{i = 1}^n \mathcal{M}_{I, - \phi} (X_i) + o_{\rm P} (1)
		\\[2mm]
		& \stackrel{\mathcal{L}}{\to} &
		\mathcal{N}(0,
		W_{\phi}^{-1} \E \big[ \mathcal{M}_{I, - \phi} (X_1) \mathcal{M}_{I, - \phi} (X_1)^T \big] W_{\phi}^{-1}
		)
		,
		\nonumber
	\end{eqnarray}
	with $W_{\phi} := 2 (1 - \| \phi \|^2)^{-1}  ( \E[ \mathcal{M}_{I, - \phi} (X_1) \mathcal{M}_{I, - \phi} (X_1)^T ] - I)$.
\end{Prop}

Note that if~$X_1\sim {\rm SC}(\phi)$, then~$W_{\phi}=-2d/\{(d+1)(1 - \| \phi \|^2)\}I$.

\section{Asymptotic null behaviour of the proposed tests}
\label{secNull}

The proposed tests reject the spherical Cauchy null hypothesis for large values of the test statistic~$\Dstat_n$ in~(\ref{TestStatistDn}), where~$\hat{\phi}_n$ is the M\"{o}bius mean of~${\rm P}_n$. As explained in Section~\ref{secinvariance}, distribution-freeness of~$\Dstat_n$ under the null hypothesis allows one to approximate arbitrarily well the required critical value by simulations. Since this may be computationally intensive for large sample sizes~$n$, we now provide the asymptotic null distribution of~$\Dstat_n$.

To state the result, denote as~$\nabla_\gamma\mu_\phi(u,s;\phi)$ the gradient of the map
\begin{equation}
	\label{Defmu}    
	\gamma
	\mapsto
	\mu_\phi(u,s;\gamma)
	:=
	{\rm P}_\phi[ u^T \mathcal{M}_{I,-\gamma}(X) \leq s ]
	-
	F_d(s)
\end{equation}
at~$\phi$, where~${\rm P}_\phi$ means that~$X\sim \mathrm{SC} (\phi)$ (we will show in the proof of Theorem~\ref{TheorAsymptDistr} below that the gradient is well-defined for all~$\phi\in\mathbb{B}^{d+1}$). We then have the following result.

\begin{Theor}
	\label{TheorAsymptDistr}
	Let~$W$ be a nonnegative bounded Borel measure on~$[0,1]$. Fix~$\phi\in\mathbb{B}^{d+1}$ and let $X_1,\ldots,X_n$ be a random sample from~$\mathrm{SC} (\phi)$. Let
	$$
	D_{n\ast}^{\phi}
	:=
	D_{n\ast}^{\phi}(X_1,\ldots,X_n)
	:=
	\frac{1}{n}
	\sum_{i,j=1}^n
	h_*^\phi(X_i,X_j)
	,
	$$
	with
	\begin{eqnarray*}
		\lefteqn{
			h_*^\phi(x,y)
			:=
			\int_{\mathbb{S}^d}
			\int_{-1}^1
			\bigg\{
			\mathbbm{1}[ u^T \mathcal{M}_{I, - \phi} (x) \leq s ]-F_d(s)
			- 
			(\nabla_\gamma\mu_\phi(u,s;\phi))^T
			W_{\phi}^{-1} \mathcal{M}_{I, - \phi} (x)
			\bigg\}
		}
		\\[2mm]
		& & 
		\hspace{-1mm}
		\times 
		\bigg\{
		\mathbbm{1}[ u^T \mathcal{M}_{I, - \phi} (y) \leq s ]-F_d(s)
		- 
		(\nabla_\gamma\mu_\phi(u,s;\phi))^T
		W_{\phi}^{-1} \mathcal{M}_{I, - \phi} (y)
		\bigg\}
		\,
		dW(F_d(s))
		d\nu_d(u)
		.
	\end{eqnarray*}
	Then, 
	\begin{equation}
		\label{AsymptDistrDn}
		\Dstat_n
		=
		D_{n\ast}^{\phi} + o_{\rm P}(1)
		\stackrel{\mathcal{L}}{\to}
		\sum_{k=1}^\infty \lambda_k Q_k
		,
	\end{equation}
	where~$\lambda_1,\lambda_2,\ldots$ are the eigenvalues of the operator~$A$ defined by
	\begin{equation}
		\label{OperatorA}
		(Aq)(x)
		=
		\int_{\mathbb{S}^d}
		h_*^0(x,y)
		q(y)
		\,
		d\nu_d(y)
	\end{equation}
	and
	where~$Q_1,Q_2,\ldots$ are \mbox{i.i.d.} $\chi^2_1$. 
\end{Theor}

As expected since exact distribution-freeness implies asymptotic distribution-freeness, the asymptotic null distribution of~$\Dstat_n$ does not depend on~$\phi$ either: this distribution is fully determined by the self-adjoint and positive semi-definite operator~$A$,\footnote{In particular, $\lambda_k\geq 0$ for all~$k$.}
whose kernel~$h_*^0$ rewrites 
\begin{eqnarray}
	\lefteqn{
		\hspace{2mm}
		h_*^0(x,y)
		=
		\int_{\mathbb{S}^d}
		\int_{-1}^1
		\bigg\{
		\mathbbm{1}[ u^T x \leq s ]-F_d(s)
		+ 
		\frac{(d+1) \Gamma(\frac{d+1}{2})}{\sqrt{\pi}d\Gamma(\frac{d}{2})}
		(1-s^2)^{d/2}
		u^T x
		\bigg\}
	}
	\label{Exprhzerostar}
	\\[2mm]
	& & 
	\hspace{1mm}
	\times 
	\bigg\{
	\mathbbm{1}[ u^T y \leq s ]-F_d(s)
	+ 
	\frac{(d+1) \Gamma(\frac{d+1}{2})}{\sqrt{\pi}d\Gamma(\frac{d}{2})}
	(1-s^2)^{d/2}
	u^T y
	\bigg\}
	\,
	dW(F_d(s))
	d\nu_d(u)
	\nonumber
\end{eqnarray}
(this expression will be obtained in the proof of Corollary~\ref{CorolNull} below).
More importantly, Theorem~\ref{TheorAsymptDistr} allows us to relate the null asymptotic distribution of~$\Dstat_n$ to that of its antecedent statistic~$P_{n}^W$ based on the kernel~$\psi$ in~(\ref{DefPsi}). Theorem~6.4.1B in \citealp{Serfling1980} entails that, under the null hypothesis of uniformity,
$$
P_{n}^W
\stackrel{\mathcal{L}}{\to}
\sum_{k=1}^\infty \lambda^\psi_k Q_k
,
$$
where~$\lambda^\psi_1,\lambda^\psi_2,\ldots$ are the eigenvalues of the operator~$A_\psi$ defined by 
$$
(A_\psi q)(x)
=
\int_{\mathbb{S}^d}
\psi(x,y)
q(y)
\,
d\nu_d(y)
$$
and where~$Q_1,Q_2,\ldots$ are still \mbox{i.i.d.} $\chi^2_1$. Since $\psi(Rx,Ry)=\psi(x,y)$ for all rotations~$R$, the operator~$A_\psi$ commutes with the $SO(d+1)$-action on~$L^2(\mathbb S^d,\nu_d)$, in the sense that~$A_\psi(\rho_Rq)=\rho_R(A_\psi q)$ with~$\rho_R q:=q\circ R^{-1}$.
Thus, $A_\psi$ acts as a scalar, $\eta_\ell^\psi$ say, on each spherical-harmonic subspace~$\mathcal H_\ell$ of degree~$\ell$. In other words, $\eta_\ell^\psi$ is an eigenvalue of~$A_\psi$ with multiplicity~$\dim\mathcal H_\ell$.\footnote{More precisely, the multiplicity is at least~$\dim\mathcal H_\ell$ since the $\eta_\ell$'s need not be pairwise different.} It follows that, under the null hypothesis of uniformity, 
\begin{equation}
	\label{AsymptDistrPnW}
	P_n^W
	\stackrel{\mathcal{L}}{\to}
	\sum_{\ell=1}^\infty 
	\eta^\psi_\ell Z_\ell
	,
\end{equation}
where~$Z_1,Z_2,\ldots$ are mutually independent with
$
Z_\ell\sim\chi^2_{\dim(\mathcal H_\ell)}
$. 

Since~(\ref{Exprhzerostar}) entails that~$h_*^0(Rx,Ry)=h_*^0(x,y)$ for all rotations~$R$, the asymptotic null distribution of~$D_n$ admits an expression similar to the one in~(\ref{AsymptDistrPnW}).
We have the following result.

\begin{Corol}
	\label{CorolNull}
	Let~$W$ be a nonnegative bounded Borel measure on~$[0,1]$. Then, under the null hypothesis,
	\begin{equation}
		\label{AsymptDistrDnCorol}
		\Dstat_n
		\stackrel{\mathcal{L}}{\to}
		\sum_{\ell=2}^\infty 
		\eta^\psi_\ell Z_\ell
		,
	\end{equation}
	where~$\eta^\psi_\ell$ denotes the eigenvalue of~$A_\psi$ associated with the spherical-harmonic subspace~$\mathcal H_\ell$ of degree~$\ell$, and where~$Z_2,Z_3,\ldots$ are mutually independent with
	$
	Z_\ell\sim\chi^2_{\dim(\mathcal H_\ell)}
	$.
\end{Corol}

This result shows that the null asymptotic distribution of~$\Dstat_n$ is obtained from that of~$P_n^W$ by dropping the term associated with the spherical-harmonic subspace~$\mathcal H_1$.
Note that, as soon as~$W((0,1))>0$, this changes the null asymptotic distribution since we then have
$$
\eta_1^\psi
=
\bigg\{
\frac{\Gamma(\frac{d+1}{2})}{\sqrt{\pi}d\Gamma(\frac{d}{2})}
\bigg\}^2
\int_{-1}^1
(1-s^2)^d
\,
dW(F_d(s))
>
0
$$
(this expression of~$\eta_1^\psi$ can be obtained by proceeding as in the proof of Corollary~\ref{CorolNull}). In particular, for the Cram\'{e}r--von Mises weight function~$W(x)=x$, this yields
$\eta_1^\psi=1/(2\pi^2)$, $1/30$, $35/(144\pi^2)$ for~$d=1,2,3$, respectively, which agrees with \cite{GarciaPortugues2023} (see their
Equation~(33) and Proposition~3.1). 

We conducted a Monte Carlo study to explore the finite-sample relevance of these theoretical results. For each dimension~$d\in\{1,2,3\}$, we generated $M=10,\!000$ mutually independent random samples of size~$n=200$ from the uniform distribution on~$\mathbb{S}^d$. In each sample, we evaluated the test statistics~$\Dstat_n$ and~$P_{n}^W$. For each dimension~$d$ and each test statistic, Figure~\ref{Fig1} reports a histogram of the resulting~$M$ values. The figure also shows estimated densities of the corresponding null asymptotic distributions (these are kernel density estimates using the default kernel and bandwidth choices of the \textsf{R} function \texttt{density}, obtained from $10,\!000$ random variables generated from the corresponding asymptotic distribution). Obviously, the excellent agreement between the histograms and the estimated densities strongly supports our theoretical null results.

\begin{figure}[h!]
	\begin{center}
		\includegraphics[width=\textwidth]{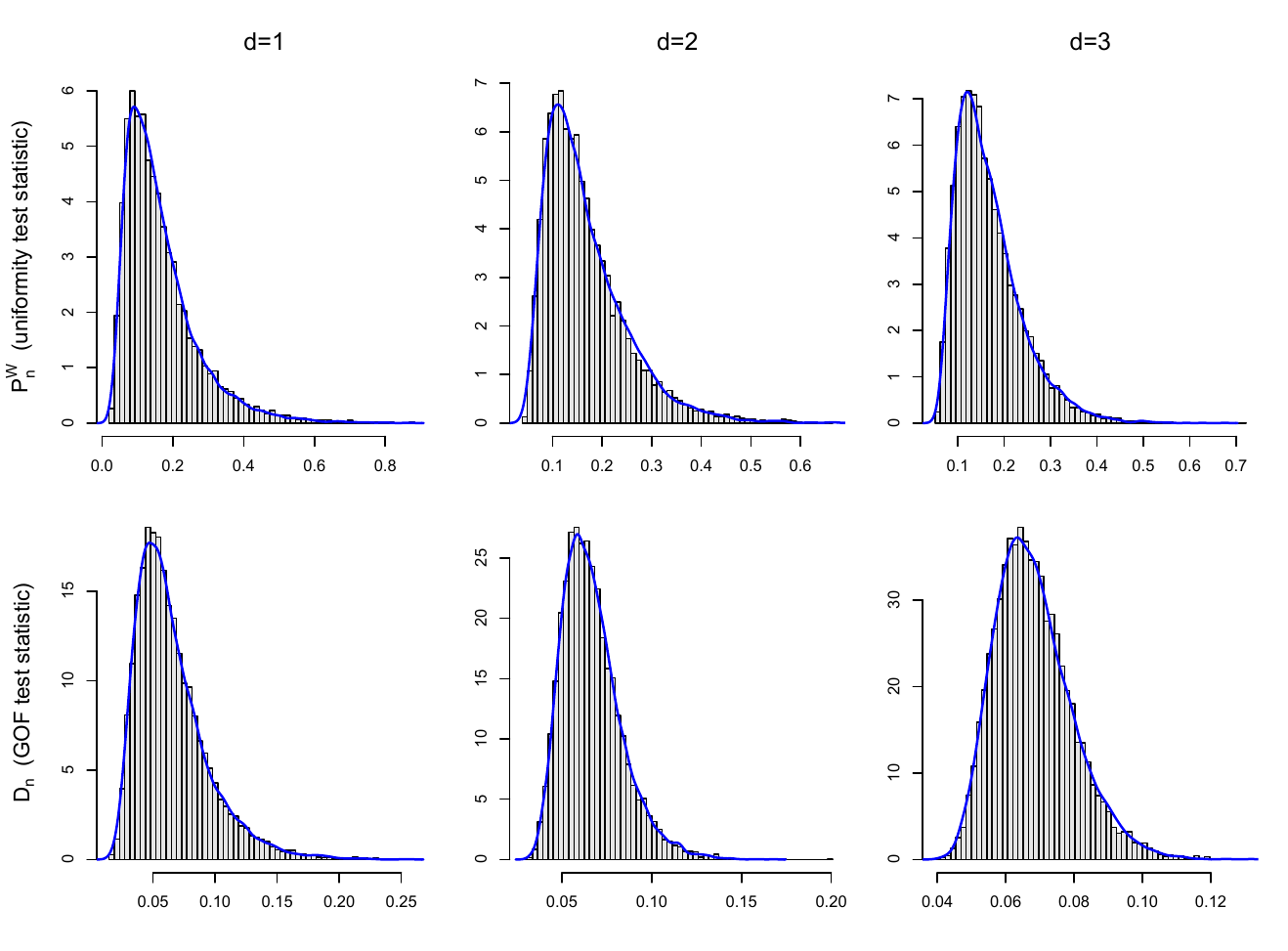}
		\caption{Monte Carlo null distributions of~$P_{n}^W$ (1st row) and~$\Dstat_n$ (2nd row) for $d=1,2,3$. 
			The histograms are based on $M=10,\!000$ independent samples of size $n=200$ from the uniform distribution on~$\mathbb{S}^d$. 
			The solid curves are kernel density estimates based on $10,\!000$ independent draws from the corresponding asymptotic null distributions.}
		\label{Fig1}
	\end{center}
\end{figure}

\section{Asymptotic non-null behaviour of the proposed tests}
\label{secNonNull}

We start with consistency against general fixed alternatives. We have the following result.

\begin{Theor}
	\label{TheorConsistency}
	Let~$\rm P$ be a non-atomic probability measure on~$\mathbb{S}^d$ with M\"{o}bius mean $\phi \in \mathbb{B}^{d+1}$,
	and let~$X_1, \ldots, X_n$ be a random sample from~$\rm P$. Assume that~$W$ is a nonnegative bounded Borel measure on~$[0,1]$ such that
	\begin{equation*}
		e_{\rm P}
		:=
		\int_{\mathbb{S}^d}
		\int_{-1}^1
		\{{\rm P}[ u^T \mathcal{M}_{I, -\phi} (X_1) \leq s ]-F_d(s)\}^2
		\,
		dW(F_d(s))
		d\nu_d(u) > 0.
	\end{equation*}
	Then, 
	${\rm P}[\Dstat_{n}>c]\to 1$ for any~$c>0$, so that the corresponding goodness-of-fit test is consistent under~$\rm P$.  
\end{Theor}
Theorem~\ref{TheorConsistency} guarantees that the proposed test is omnibus if the measure $W$ is absolutely continuous.
\begin{Corol}
	\label{CorConsistency}
	Assume that~$W$ is a nonnegative, bounded, and absolutely continuous Borel measure on~$[0,1]$ such that~$W([0, 1]) > 0$. Then, 
	the corresponding goodness-of-fit test is omnibus under alternatives with density in $L^2 (\nu_d)$. That is to say, given any absolutely continuous probability measure~$\rm P$ over~$\mathbb{S}^d$ with density in $L^2 (\nu_d)$ that is not spherical Cauchy and an i.i.d.~random sample $X_1, \ldots, X_n$ from~$\rm P$, we have that
	${\rm P}[\Dstat_{n}>c]\to 1$ for any~$c>0$.
\end{Corol}

Corollary~\ref{CorConsistency} ensures that any positive measure $W$ that admits a density function provides an omnibus test. In particular, the goodness-of-fit test based on the Cramér--von Mises weight function $W (x) = x$ is omnibus. Note that the simple assumption on~$W$ in this result is sufficient but not necessary for omnibusness. For a necessary and sufficient condition, see Corollary~3.7 in \cite{GarciaPortugues2023}.

We now turn to local alternatives. We start with local alternatives to uniformity associated to densities (with respect to the surface area measure~$\sigma_d$) of the form
\begin{equation}
	\label{eq:fgn}
	f_{\tau_n}(x) 
	:=
	\frac{\Gamma\big( \frac{d + 1}{2} \big)}{2 \pi^{(d + 1)/2}}
	\bigg(
	1 + \frac{\tau_n (x)}{\sqrt{n}} 
	\bigg)
	,
\end{equation}
where~$(\tau_n)$ is a sequence of continuous functions on~$\mathbb{S}^d$ such that
\begin{equation}
	\label{eq:taun1}
	\int_{\mathbb{S}^d}
	\tau_n(u)
	\,
	d\nu_d(u)
	= 
	0
\end{equation}
for all $n \in \N$ and such that 
\begin{equation}
	\label{eq:taun2}
	\lim_{n \to \infty}
	\sup_{u \in \mathbb{S}^d} \vert \tau_n (u) - \tau (u) \vert 
	= 
	0
\end{equation}
for some function~$\tau$ on~$\mathbb{S}^d$ (clearly, $\tau$ is then continuous). Note that~$f_{\tau_n}$ might not be a density for~$n$ small because it might take negative values; however, since $\tau_n/\sqrt{n}$ converges uniformly to zero, $f_{\tau_n}$ will be a density for~$n$ large enough. We denote by~${\rm P}^{(n)}_{\tau_n}$ the hypothesis under which~$X_1,\ldots,X_n$ form a random sample from the density~$f_{\tau_n}$. Consider then the local log-likelihood ratio
\begin{equation}
	\label{eq:logratio}
	\Lambda_n 
	:=
	\log \bigg( \frac{d{\rm P}^{(n)}_{\tau_n}}{d{\rm P}^{(n)}_0} \bigg)
	=
	\sum_{i = 1}^n
	\,
	\log
	\bigg(
	1 + \frac{\tau_n (X_i)}{\sqrt{n}} 
	\bigg)
	,
\end{equation}
where~${\rm P}^{(n)}_0$ denotes the (null) hypothesis under which~$X_1,\ldots,X_n$ form a random sample from the uniform distribution on~$\mathbb{S}^d$. The following result clarifies the null asymptotic behavior of $\Lambda_n$.

\begin{Prop}
	\label{prop:contiguity}
	Under~${\rm P}^{(n)}_0$,
	\begin{equation*}
		\Lambda_n 
		=
		\frac{1}{\sqrt{n}} \sum_{i = 1}^n \tau (X_i)
		-
		\frac{1}{2} \E_0[\tau^2 (X_1)]
		+
		o_{\rm P} (1)
	\end{equation*}
	as~$n\to\infty$, where~$\E_0$ denotes expectation under~${\rm P}^{(n)}_0$. 
\end{Prop}

This likelihood expansion allows us to study the asymptotic behavior of~$\Dstat_n$ under the local alternatives~${\rm P}^{(n)}_{\tau_n}$. The following result shows that the asymptotic power of the proposed tests under these alternatives is determined by the spherical-harmonic components of the limiting direction~$\tau$.

\begin{Theor}
	\label{th:localpower}
	Let~$(\tau_n)$ satisfy~\eqref{eq:taun1}--\eqref{eq:taun2}, and let~$W$ be a nonnegative bounded Borel measure on~$[0,1]$. Let~$\Pi_\ell$ be the orthogonal projection in~$L^2(\nu_d)$ onto the spherical-harmonic subspace~$\mathcal H_\ell$ of degree~$\ell$. Then, under~${\rm P}^{(n)}_{\tau_n}$,
	\begin{equation}
		\label{AsymptDistrDnLocal}
		\Dstat_n
		\stackrel{\mathcal{L}}{\to}
		\sum_{\ell=2}^{\infty}
		\eta_\ell^\psi
		Z_\ell
		,
	\end{equation}
	where~$\eta^\psi_\ell$ denotes the eigenvalue of~$A_\psi$ associated with~$\mathcal H_\ell$, and where~$Z_2,Z_3,\ldots$ are mutually independent with
	$
	Z_\ell\sim\chi^2_{\dim(\mathcal H_\ell)}(\|\Pi_\ell\tau\|_{L^2(\nu_d)}^2)
	$.
\end{Theor}

Theorem~\ref{th:localpower} shows in particular that, under local alternatives whose limiting direction belongs to~$\mathcal H_1$, the asymptotic distribution of~$\Dstat_n$ is the same as under the null hypothesis, since only the projections~$\Pi_\ell\tau$ with $\ell\geq2$ enter the limiting distribution in~\eqref{AsymptDistrDnLocal}. Equivalently, if the $\tau_n$'s are such that $\tau(x)=v^Tx$ for some~$v\in\R^{d+1}$, then the proposed tests have asymptotic power equal to their nominal level under the corresponding local alternatives. This fact may seem surprising at first glance, but it actually makes perfect sense, as explained in the following result.

\begin{Prop}
	\label{prop:blind}
	Let~$(\tau_n)$ satisfy~\eqref{eq:taun1}--\eqref{eq:taun2} and assume that there exists~$v \in \R^{d+1}$ such that
	$\tau(x)=v^Tx$ for any~$x\in\mathbb{S}^d$.
	Let~${\rm Q}_{v}^{(n)}$ denote the sequence of hypotheses under which~$X_1,\ldots,X_n$ form a random sample from the distribution
	$$
	{\rm SC}\bigg(\frac{v}{2d\sqrt n}\bigg),
	$$
	for all sufficiently large~$n$. Then, (i) the local alternatives~${\rm P}^{(n)}_{\tau_n}$ are asymptotically equivalent to the spherical Cauchy hypotheses~${\rm Q}_{v}^{(n)}$, in the sense that, under both~${\rm P}^{(n)}_{\tau_n}$ and~${\rm Q}_{v}^{(n)}$,
	\begin{equation}
		\label{eq:PgnQv0equiv}
		\log \bigg( \frac{d{\rm P}^{(n)}_{\tau_n}}{d{\rm Q}^{(n)}_{v}} \bigg)
		=
		o_{\rm P}(1)
	\end{equation}
	as~$n\to\infty$.
	(ii) If~$T_n: \mathbb{S}^d \times \ldots \times \mathbb{S}^d\ (n \textrm{ times}) \to \mathbb{R}$ is a sequence of statistics that are invariant under M\"{o}bius transformations $($that is,
	$
	T_n (x_1, \ldots, x_n)
	=
	T_n \big( \mathcal{M}_{R, \varphi} (x_1), \ldots, \mathcal{M}_{R, \varphi} (x_n) \big)
	$
	for all~$n \in \N_0$, $x_1, \ldots x_n \in \mathbb{S}^d$, $\varphi \in \mathbb{B}^{d+1}$, and $R \in \mathrm{SO} (d + 1))$, then~$T_n$ has a limiting distribution under~${\rm P}^{(n)}_{\tau_n}$ if and only if it has one under~${\rm P}^{(n)}_{0}$, in which case both limiting distributions are the same.
\end{Prop}

Proposition~\ref{prop:blind} also clarifies which local directions are relevant for a M\"{o}bius-invariant goodness-of-fit test. The local tangent space to the spherical Cauchy family at the uniform distribution is
$
\mathcal H_1=\{x\mapsto v^Tx:v\in\R^{d+1}\}
$.
Therefore, M\"{o}bius-invariant statistics have the same limiting behaviour under local alternatives whose direction belongs to~$\mathcal H_1$ as under the null hypothesis. The relevant local alternatives are thus those with a non-zero component in the orthogonal complement to~$\mathcal H_1$. Against such alternatives, Theorem~\ref{th:localpower} shows that the proposed test has non-trivial local power whenever
$$
\sum_{\ell=2}^{\infty}
\eta_\ell^\psi
\|\Pi_\ell\tau\|_{L^2(\nu_d)}^2
>
0
.
$$
In this sense, the test is locally sensitive to precisely those spherical-harmonic components of the local alternative that are not tangent to the spherical Cauchy null model.

We illustrate the local power results above on two classical rotationally symmetric families. (a) Consider first Fisher--von Mises local alternatives with location~$\theta\in\mathbb{S}^d$ and concentration~$\kappa_n=\rho/\sqrt n$, whose density with respect to~$\nu_d$ is proportional to
$$
x\mapsto \exp\bigg(\frac{\rho}{\sqrt n}\theta^Tx\bigg).
$$
The corresponding sequence of functions~$\tau_n$ satisfies 
$
\tau_n(x)\to \tau(x)=\rho\,\theta^Tx
$
uniformly in~$x\in\mathbb{S}^d$. Since this limiting direction belongs to~$\mathcal H_1$, all projections~$\Pi_\ell\tau$, $\ell\geq2$, vanish. Therefore, Theorem~\ref{th:localpower} entails that the asymptotic distribution of~$D_n$ is the same under these Fisher--von Mises local alternatives as under the null hypothesis. This is compatible with Proposition~\ref{prop:blind}. 
(b) Consider then Watson local alternatives with location~$\theta\in\mathbb{S}^d$ and concentration~$\kappa_n=\rho/\sqrt n$, whose density with respect to~$\nu_d$ is proportional to
$$
x\mapsto \exp\bigg\{\frac{\rho}{\sqrt n}(\theta^Tx)^2\bigg\}
.
$$
One can easily check that the corresponding sequence of functions~$\tau_n$ is such that
$$
\tau_n(x)\to \tau(x)
=
\rho\bigg\{(\theta^Tx)^2-\frac{1}{d+1}\bigg\}
$$
uniformly in~$x\in\mathbb{S}^d$. The limiting direction~$\tau$ belongs to~$\mathcal H_2$, so that~$\Pi_\ell\tau=\delta_{\ell 2}\tau$, where~$\delta_{\ell\ell'}$ is the usual Kronecker symbol. Since~$\dim(\mathcal H_2)=d(d+3)/2$ and since
$$
\|\Pi_2\tau\|_{L^2(\nu_d)}^2
=
\|\tau\|_{L^2(\nu_d)}^2
=
\rho^2\E_0\bigg[
\bigg\{(\theta^TX_1)^2-\frac{1}{d+1}\bigg\}^2
\bigg]
=
\frac{2d\rho^2}{(d+1)^2(d+3)}
$$
(the last equality results from, \emph{e.g.}, Lemma~A.1(iii) in \citealp{PaindaveineVerdebout2016}), Theorem~\ref{th:localpower} yields, under these Watson local alternatives,
$$
\Dstat_n
\stackrel{\mathcal L}{\to}
\eta_2^\psi
\chi^2_{\frac{d(d+3)}{2}}
\bigg(
\frac{2d\rho^2}{(d+1)^2(d+3)}
\bigg)
+
\sum_{\ell=3}^\infty
\eta_\ell^\psi
\chi^2_{\dim(\mathcal H_\ell)}
,
$$
where the chi-square variables are mutually independent.

We also conducted a Monte Carlo study to compare the distributions predicted by our local-alternative asymptotic theory with simulated distributions at a large sample size. For each dimension~$d\in\{1,2,3\}$, we generated $M=10,\!000$ mutually independent random samples of size~$n=5,\!000$ from the two sequences of local alternatives above, namely Fisher--von Mises alternatives and Watson alternatives, both with location~$\theta=(1,0,\ldots,0)^T$ and concentration~$\kappa=\rho/\sqrt n$, where $\rho=20$. In each sample, we evaluated the Cram\'{e}r--von Mises version ($W(x)=x$) of the proposed test statistic~$\Dstat_n$. For each dimension~$d$ and for each of the two types of alternatives, Figure~\ref{Fig2} reports a histogram of the resulting~$M$ values. The figure also shows estimated densities of the corresponding non-null asymptotic distributions (these are still kernel density estimates using the default kernel and bandwidth choices of the \textsf{R} function \texttt{density}, obtained from $10,\!000$ random variables generated from the corresponding asymptotic non-null distribution). The close agreement between the histograms and the estimated densities provides strong numerical support for our local-alternative asymptotic results, although convergence appears to be slower for Fisher--von Mises alternatives in dimension~$d=1$.

\begin{figure}[h!]
	\begin{center}
		\includegraphics[width=\textwidth]{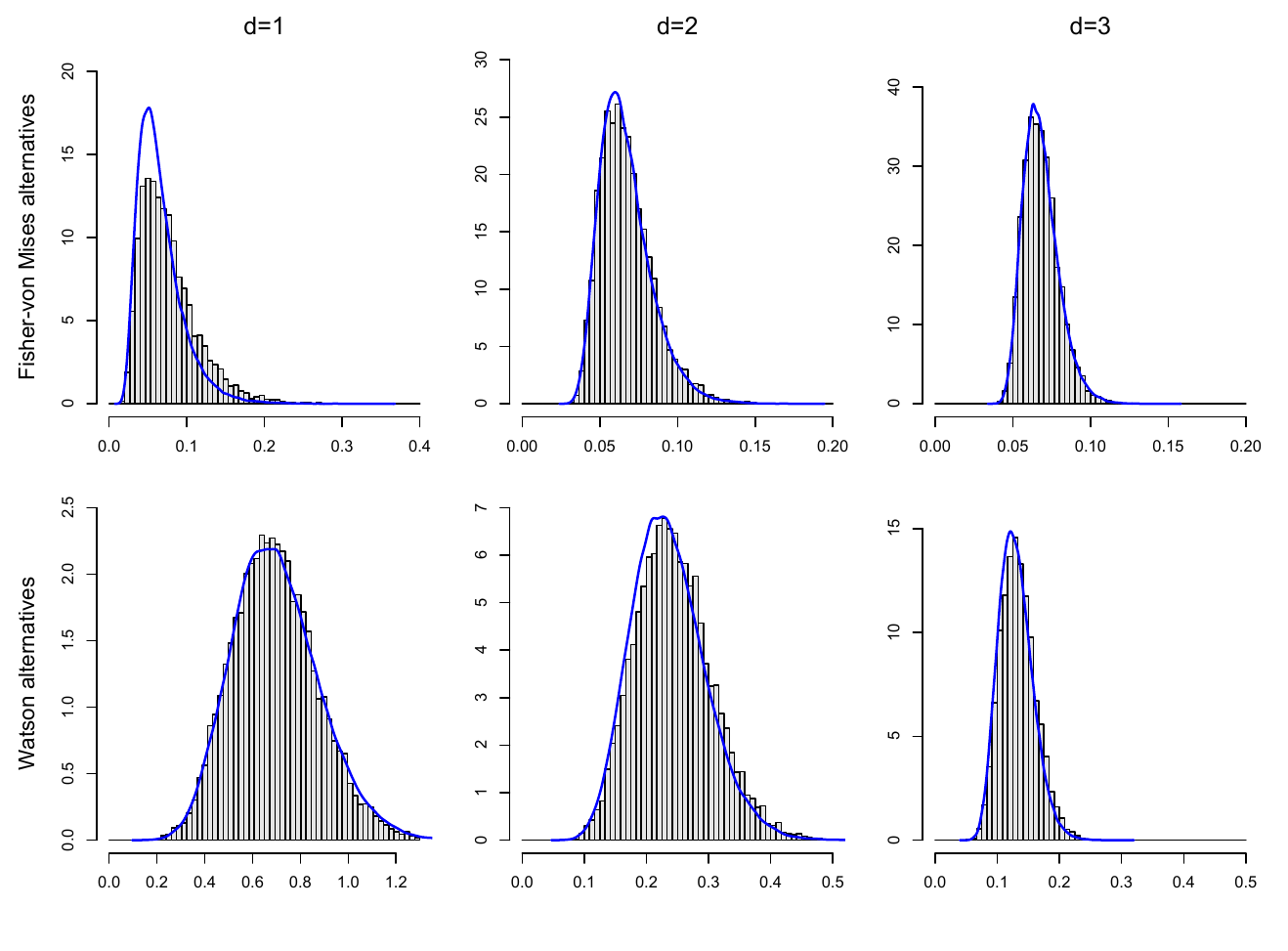}
		\caption{Monte Carlo local-alternative distributions of the Cram\'{e}r--von Mises version of~$\Dstat_n$ for $d=1,2,3$ under Fisher--von Mises alternatives (1st row) and Watson alternatives (2nd row). 
			The histograms are based on $M=10,\!000$ independent samples of size $n=5,\!000$ generated under local alternatives with location~$\theta=(1,0,\ldots,0)^T$ and concentration~$\kappa=\rho/\sqrt n$, where $\rho=20$. 
			The solid curves are kernel density estimates based on $10,\!000$ independent draws from the corresponding asymptotic non-null distributions.}
		\label{Fig2}
	\end{center}
\end{figure}

While the local non-null results of this section focused on local alternatives to uniformity, they immediately extend to local alternatives to any fixed~${\rm SC}(\phi)$. Fix~$\phi\in\mathbb B^{d+1}$ and consider alternatives having density $x\mapsto 1+r_n(x)/\sqrt n$ with respect to~${\rm P}_\phi={\rm SC}(\phi)$, where
$\int_{\mathbb S^d} r_n(x)\,d{\rm P}_\phi(x)=0$. Equivalently, with respect to the surface area measure~$\sigma_d$, the density is
$x\mapsto f(x;\phi)\{1+r_n(x)/\sqrt n\}$. Let $\tau_n(u):=r_n(\mathcal M_{I,\phi}(u))$, $u\in\mathbb S^d$, and assume that~$(\tau_n)$ satisfies~\eqref{eq:taun1}--\eqref{eq:taun2}, with limit~$\tau$. If $U_i=\mathcal M_{I,-\phi}(X_i)$, then the~$U_i$'s have density $1+\tau_n/\sqrt n$ with respect to~$\nu_d$. By M\"{o}bius invariance of~$\Dstat_n$, Theorem~\ref{th:localpower} therefore yields that, under these local alternatives,
$$
\Dstat_n
\stackrel{\mathcal L}{\to}
\sum_{\ell=2}^{\infty}
\eta_\ell^\psi Z_\ell,
$$
where~$Z_2,Z_3,\ldots$ are mutually independent with $Z_\ell\sim\chi^2_{\dim(\mathcal H_\ell)}(\|\Pi_\ell\tau\|_{L^2(\nu_d)}^2)
$. Thus, local alternatives to~${\rm SC}(\phi)$ are simply handled by pulling them back to uniformity through~$\mathcal M_{I,-\phi}$; the relevant local direction is the limit~$\tau$ of the pulled-back perturbations~$\tau_n=r_n\circ\mathcal M_{I,\phi}$.

\section{Simulation study}
\label{secsimus}

In this section, we perform several Monte Carlo exercises in order to investigate the finite-sample behavior of our proposal. For the sake of clarity, we focus entirely on the goodness-of-fit test based on the Cramér--von Mises measure, that is, on the test based on the statistic
\begin{equation*}
	\Dstat_n
	:=
	\Dstat_n(X_1,\ldots,X_n)
	:=
	\frac{1}{n}
	\sum_{i,j=1}^n
	\psi\big(\mathcal{M}_{I, -\hat{\phi}_n} (X_i),\mathcal{M}_{I, -\hat{\phi}_n} (X_j)\big)
	,
\end{equation*}
with $\hat{\phi}_n$ given by \eqref{eq:sampleMobius} and
\begin{equation*}
	\psi(x,y)
	=
	\int_{\mathbb{S}^d}
	\int_{-1}^1
	\{\mathbbm{1}[ u^T x \leq s ]-F_d(s)\}
	\{\mathbbm{1}[ u^T y \leq s ]-F_d(s)\}
	\,
	d F_d(s)
	d\nu_d(u)
	.
\end{equation*}
Note that Proposition~2.4 in \cite{GarciaPortugues2023} provides a closed-form expression for $\psi$ when $d \leq 3$. 

Section~\ref{sec:critvalues} explains the procedure that we followed to estimate the critical values of the test. Section~\ref{sec:emppower} then investigates the empirical power of our test under different alternatives, while Section~\ref{sec:comparison} compares its performance with a competitor. Throughout, we made extensive use of the R package \texttt{sphunif} \citep{GarciaPortugues2025}.

\subsection{Estimates of the critical values}
\label{sec:critvalues}

One of the main advantages of our proposal over its competitors is exact distribution-freeness: since the null distribution of $D_n$ does not depend on the true value of the parameter $\phi$, we do not need to rely on a resampling scheme for calibration. Instead, given a data sample $X_1, \ldots, X_n$ and a significance level $\alpha \in (0,1)$, we simply reject the null hypothesis $H_0$ in \eqref{eq:GoFCauchy} if $D_n > c_{n, \alpha}$, where the critical value $c_{n, \alpha} > 0$ is such that
$$
c_{n,\alpha}
:=
\inf\big\{c\in\mathbb R:
{\rm P}^{(n)}_0[D_n\leq c]\geq1-\alpha\big\};
$$
recall that~${\rm P}^{(n)}_0$ denotes the hypothesis under which~$X_1,\ldots,X_n$ form a random sample from the uniform distribution on~$\mathbb{S}^d$. 

In practice, the distribution of $D_n$ under $H_0$ is not tractable, so we estimated the critical values $c_{n, \alpha}$ via Monte Carlo. For a given sample size $n$, we simulated $M = 100{,}000$ mutually independent samples of size $n$ from the uniform distribution on~$\mathbb{S}^d$, and we computed $D_n$ for each sample. This provided $M$ replicates of the test statistic, say $D_n^1, \ldots, D_n^M$. From these replicates, we approximated the critical values $c_{n, \alpha}$ by setting
$$ 
\hat{c}^{(M)}_{n, \alpha} 
:=
\inf\left\{
c\in\mathbb R:
\frac1M\sum_{m=1}^M\mathbbm{1}[D_n^m\leq c]
\geq1-\alpha
\right\}.
$$
We did this for sample sizes $n = 25, 50, 100$ and $200$ and dimensions $d = 1, 2$ and $3$. For reproducibility purposes, the resulting estimated critical values $\hat{c}_{n, \alpha}^{(M)}$ are given in Table~\ref{TabCritValDiego}.

\begin{table}
	\caption{Estimates of the critical values of the goodness-of-fit test based on the Cramér--von Mises measure. Each estimate is based on $M=100{,}000$ Monte Carlo replicates of
		the corresponding test statistic obtained under the uniform distribution on~$\mathbb{S}^d$.}
	\begin{center}
		\begin{tabular}{cccccc}
			\toprule
			& & $n = 25$ & $n = 50$ & $n = 100$ & $n = 200$ \\
			\midrule
			& 10\% & 0.102600 & 0.102674 & 0.102562 & 0.103050\\
			$ d = 1$ & 5\% & 0.119389 & 0.120133 & 0.120195 & 0.120273\\
			& 1\% & 0.158149 & 0.159509 & 0.161810 & 0.160832\\
			\midrule
			& 10\% & 0.088381 & 0.088525 & 0.088776 & 0.088756\\
			$ d = 2 $ & 5\% & 0.097109 & 0.097455 & 0.097885 & 0.097647\\
			& 1\% & 0.115928 & 0.116518 & 0.117123 & 0.117395\\
			\midrule
			& 10\% & 0.083190 & 0.083322 & 0.083347 & 0.083261\\
			$d = 3$ & 5\% & 0.088758 & 0.089129 & 0.088999 & 0.088922\\
			& 1\% & 0.101000 & 0.101376 & 0.100733 & 0.100741\\
			\bottomrule
		\end{tabular}
		\label{TabCritValDiego}
	\end{center}
\end{table}

\subsection{Power against fixed alternatives}
\label{sec:emppower}

We studied the finite-sample power of our test under two types of alternatives:
\begin{itemize}
	\item Fisher--von Mises distributions with mean direction $\theta = (1, 0, \ldots, 0)^T$ and concentration~$\kappa$, with density proportional to $x \mapsto \exp(\kappa \theta^T x)$;
	\vspace{2mm}
	\item Watson distributions with directional parameter $\theta = (1, 0, \ldots, 0)^T$ and concentration~$\kappa$, whose density  is proportional to $x \mapsto \exp(\kappa (\theta^T x)^2)$.
\end{itemize}
These alternatives are rotationally symmetric about $\theta$. In both cases, $\kappa = 0$ yields the uniform distribution on $\mathbb{S}^d$, which belongs to the null hypothesis, whereas larger values of~$\kappa$  provide increasingly severe alternatives. However, the Fisher--von Mises distribution is unimodal with mode $\theta$, while the Watson distribution is bimodal with modes at $\pm\theta$. Therefore, the Fisher--von Mises alternatives are intuitively closer to the spherical Cauchy family than the Watson alternatives. 

For each model, we generated $M = 1{,}000$ mutually independent samples for any combination of $d \in \{1, 2, 3\}$, $n \in \{25, 50, 100, 200\}$, and $\kappa \in \{k r \colon k = 0, 1, \ldots, 6 \}$, where~$r=1$ for Fisher--von Mises alternatives and $r=0.25$ for Watson alternatives. Then, we applied our test with significance level $\alpha$ to each sample.

The resulting rejection frequencies for $\alpha = 5\%$ are plotted in Figure~\ref{FigPower}.
Our proposal works as expected. Power increases both with the concentration parameter and with the sample size. Also, increasing the dimension results in a loss of power, since the dimension of the parameter space of the spherical Cauchy model increases with $d$ and, therefore, the testing problem becomes harder. Note that the empirical powers of our test are higher for the Watson alternatives than for the Fisher--von Mises ones. For instance, with $d = 2$, the test achieves an empirical power of approximately 1 with $\kappa = 1.5$ for the Watson alternatives, whereas for the Fisher--von Mises alternatives one has to increase $\kappa$ to $5$ to reach similar power levels. This is in line with the fact that, as has been discussed above, Fisher--von Mises alternatives may be considered closer to the spherical Cauchy family than Watson alternatives.

\begin{figure}[t!]
	\begin{center}
		\includegraphics[width=\textwidth]{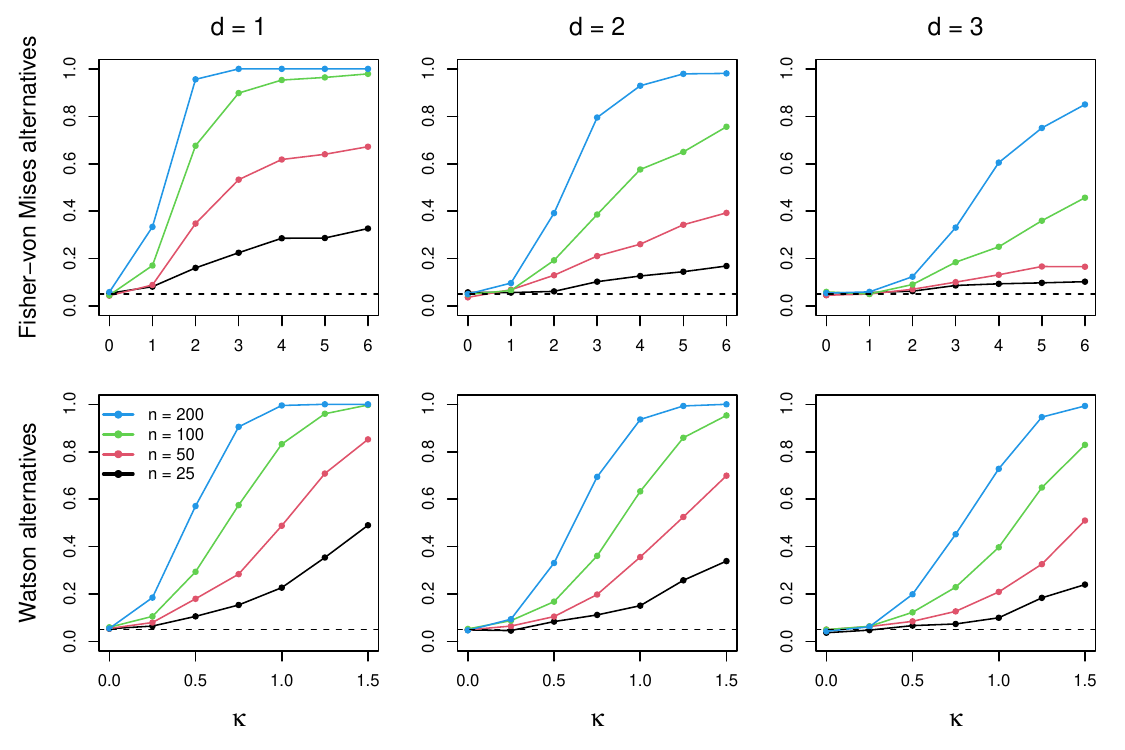}
		\caption{Rejection frequencies of our test at level~$\alpha=5\%$ under Fisher--von Mises alternatives (first row) and Watson alternatives (second row) for different values of the concentration parameter $\kappa$. In each case, the test is based on the critical values in Table~\ref{TabCritValDiego}. Rejection frequencies are obtained from~$M = 1{,}000$ Monte Carlo replications. The significance level~$\alpha$ is plotted for reference (horizontal dashed line).}
		\label{FigPower}
	\end{center}
\end{figure}

\subsection{Comparison with other goodness-of-fit tests}
\label{sec:comparison}

We now compare our test with the test proposed by \cite{Ebner2024}. In a nutshell, their method measures a weighted $L^2$-distance between the characteristic function of the spherical Cauchy distribution~$\mathrm{SC}(\phi)$ and the empirical characteristic function based on the sample~$X_1,\ldots,X_n$. Nevertheless, since the characteristic function of the spherical Cauchy
distribution has no closed form, as is also the case for many other parametric families on the hypersphere, they propose to estimate it with
$$
\hat{\Phi}_m (t) = \frac{1}{m} \sum_{j = 1}^m \exp (\mathrm{i} t^T Y_j), \quad t \in \mathbb{R}^{d+1},
$$
where $Y_1, \ldots, Y_m$ is an i.i.d.~sample from $\mathrm{SC} (\hat{\phi}_n)$ and $\hat{\phi}_n$ is the maximum likelihood estimator of $\phi$ obtained from the sample~$X_1,\ldots,X_n$ at hand. The test statistic proposed by \cite{Ebner2024} is then of the form
\begin{multline*}
	T_{n, m, \gamma}^{\{2\}}
	=
	\frac{1}{n + m}
	\bigg(
	\frac{m}{n} \sum_{j, k = 1}^n \exp \big( - \gamma \| X_j - X_k \|^2 \big) \\
	-
	2 \sum_{j = 1}^n \sum_{k = 1}^m \exp \big( - \gamma \| X_j - Y_k \|^2 \big) +
	\frac{n}{m} \sum_{j, k = 1}^m \exp \big( - \gamma \| Y_j - Y_k \|^2 \big) \bigg)
	,
\end{multline*}
where $\gamma > 0$ is a tuning parameter that has to be set by the practitioner (see equation~(22) of that paper). Since the null distribution of $T_{n, m, \gamma}^{\{2\}}$ depends on the true value of~$\phi$, the test is calibrated using a resampling algorithm that is described in Section~5 of \cite{Ebner2024}. Hereafter, we refer to the resulting test as the EHM test.

We applied the EHM test under the models  considered in Section~\ref{sec:emppower}. We did so for $\gamma=0.5, 1$ and $2$, with~$m=200$ in each case. For each sample, the p-value was estimated from $B=500$ resamples generated using the algorithm mentioned above.
Figure~\ref{FigComparison} plots the resulting rejection frequencies for $n = 200$ and $\alpha = 5\%$, 
along with the corresponding rejection frequencies of our test. Clearly, our test is consistently more powerful than the EHM test, regardless of the choice of~$\gamma$. This difference is even more remarkable for Fisher--von Mises alternatives: in that case, for $d = 2$, the rejection frequencies of the EHM tests are less than half those achieved by our test, and, for $d = 3$, the EHM tests exhibit almost no power for the considered values of~$\kappa$. The EHM tests perform better for Watson alternatives, yet they are still less powerful than our test in all cases.

\begin{figure}[t!]
	\begin{center}
		\includegraphics[width=\textwidth]{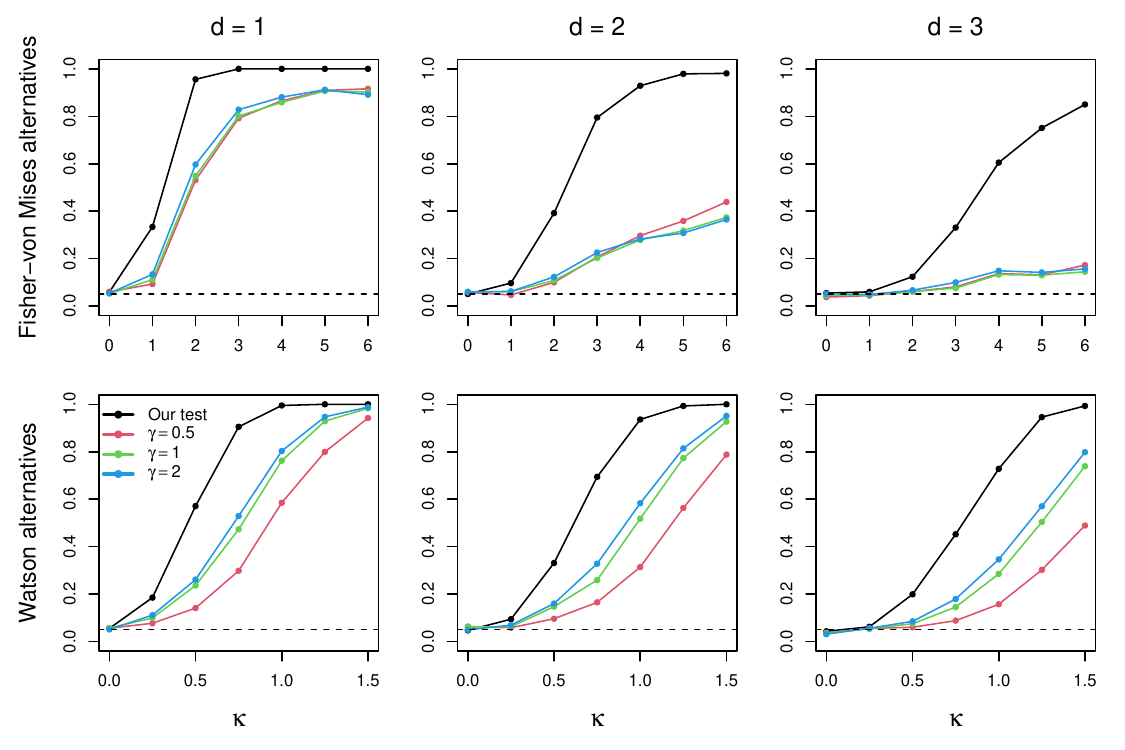}
		\caption{Rejection frequencies of several goodness-of-fit tests at level~$\alpha=5\%$ under Fisher--von Mises alternatives (first row) and Watson alternatives (second row) for different values of the concentration parameter~$\kappa$. The considered tests are our Cram\'{e}r--von Mises test and the EHM tests associated with $\gamma = 0.5, 1$ and $2$, and $m = 200$. The sample size is $n = 200$. Rejection frequencies are obtained from $M = 1{,}000$ Monte Carlo replications. Our test is based on the critical values in Table~\ref{TabCritValDiego}, whereas the EHM tests are implemented from a resampling algorithm using $B = 500$ resamples. The significance level $\alpha$ is plotted for reference (horizontal dashed line).}
		\label{FigComparison}
	\end{center}
\end{figure}

Finally, we would like to comment on another point: since our test does not require resampling for calibration, it is much faster to compute than its competitors. To put things in numbers, computing all rejection frequencies for our tests reported in Figure~\ref{FigPower} took about $7.5$ minutes, whereas the corresponding computations for the EHM tests required approximately $53$ hours. This means that our test is more than $400$ times faster to compute than the EHM test. Although a more extensive comparison of computation times across sample sizes would be useful, we think that this illustrates clearly the advantage of distribution-free tests in terms of computation time.

\section{Real data analysis}
\label{sec:realdata}

In this section, we present a simple but illustrative real data example. We analyze the data set of paleomagnetic pole positions provided by \cite{Schmidt1976}.
This data set contains 33 estimates of the Earth's former magnetic-pole position, each corresponding to a different sampling site on the island of Tasmania (Australia). Figure~\ref{FigTasmania} shows the 33 observations, which are concentrated in the Southern hemisphere.

\begin{figure}[t!]
	\begin{center}
		\includegraphics[width=0.3\textwidth]{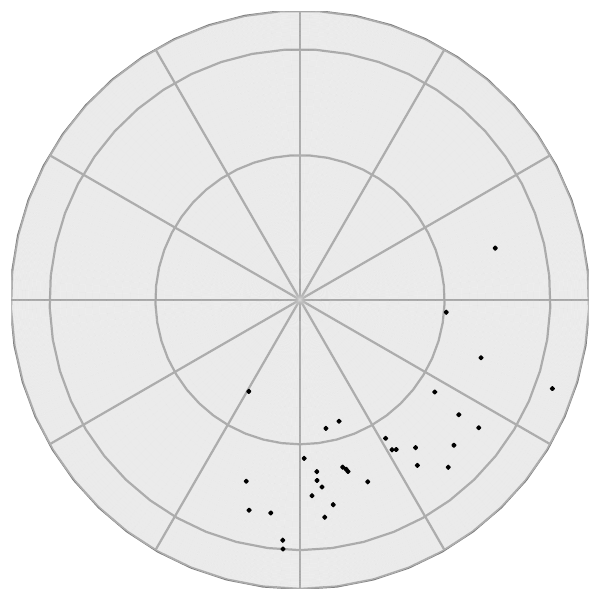}
		\caption{Paleomagnetic pole positions reported by \cite{Schmidt1976}. The 33 observations are represented on the sphere using an orthogonal projection centered at the South pole.}
		\label{FigTasmania}
	\end{center}
\end{figure}

This data set is a classical example in directional statistics and has been widely used to illustrate several data analysis techniques. For instance, \cite{Tsagris2025b} introduce two new parametric families for spherical data and they use these data to demonstrate the advantages of their models over the spherical Cauchy family. They claim that their models better capture the structure of this data set, since they allow for a non-rotationally symmetric structure, whereas the spherical Cauchy distributions are always rotationally symmetric around the mode. Our test can provide some theoretical support for this claim: rejection of $H_0$ would indicate that the spherical Cauchy model is not suitable for these data and that a more flexible model is needed.

For this data set, our Cramér--von Mises goodness-of-fit statistic~$D_{33}$ took the value~$0.108$ (rounded to three decimal places). We generated $M = 10{,}000$ Monte Carlo values $D_{33}^1,\ldots,D_{33}^M$ under $H_0$, which provided
$$ 
\frac{1}{M} \sum_{m = 1}^M \mathbbm{1}[ D_{33}^m > 0.108 ] 
\approx
0.021
$$
as an estimated p-value for our test. Therefore, our test rejects the null hypothesis at level~$5\%$, which supports the claims by \cite{Tsagris2025b}.

We also applied the EHM tests to the data for comparison. We simulated a sample $Y_1, \ldots, Y_{200}$ from the distribution $\mathrm{SC} (\hat{\phi}_n)$ and evaluated the resulting EHM test statistics for~$\gamma=0.5, 1$ and~$2$. This provided
$$
T_{33, 200, 0.5}^{\{2\}} \approx 0.249, 
\quad
T_{33, 200, 1}^{\{2\}} \approx 0.519,
\quad
\textrm{ and }
\quad
T_{33, 200, 2}^{\{2\}} \approx  0.831
,
$$
again rounded to three decimal places. Based on $M = 10{,}000$ bootstrap resamples, the corresponding estimated $p$-values are~$0.305$, $0.205$, and $0.139$, respectively. Consequently, none of them reject the null hypothesis at level~$5\%$. This illustrates that the extra empirical power of our test can be consequential at small sample sizes.

We also use this data set to explore another drawback of the EHM proposal: its dependence on the artificial sample~$Y_1,\ldots,Y_m$. Due to this artificial sample, the EHM test statistic is random \textit{given the data sample}, so that the outcome of the test is random, too: for the same data, the test may or may not reject~$H_0$ depending on the artificial sample.
To illustrate this, we performed the following Monte Carlo experiment. We simulated $A = 1{,}000$ independent artificial samples, still of size $m = 200$. For each artificial sample, we evaluated the EHM test statistics and approximated the corresponding $p$-values using $M = 10{,}000$ bootstrap samples. This provided~$A = 1{,}000$ different $p$-values for each~$\gamma$. Figure~\ref{FigDependenceOnSample} shows a boxplot of these $p$-values for each~$\gamma$. Clearly, the $p$-values vary dramatically, meaning that the artificial sample has a substantial impact on the outcome of the tests. For $\gamma = 0.5, 1$ and~$2$, the EHM test rejects the null hypothesis at level $\alpha = 5\%$ for about $28\%$, $46\%$, and $60\%$ of the artificial samples, respectively.
Of course, one could mitigate the effect of the artificial sample by increasing~$m$, but \cite{Ebner2024} does not provide any heuristic to select~$m$ in practice. Moreover, increasing~$m$ would also drastically increase the computation time due to the involved resampling procedure.
Against this background, our proposal provides a fast and reliable way of testing goodness-of-fit to the spherical Cauchy family.

\begin{figure}[t!]
	\begin{center}
		\includegraphics[width=0.5\textwidth]{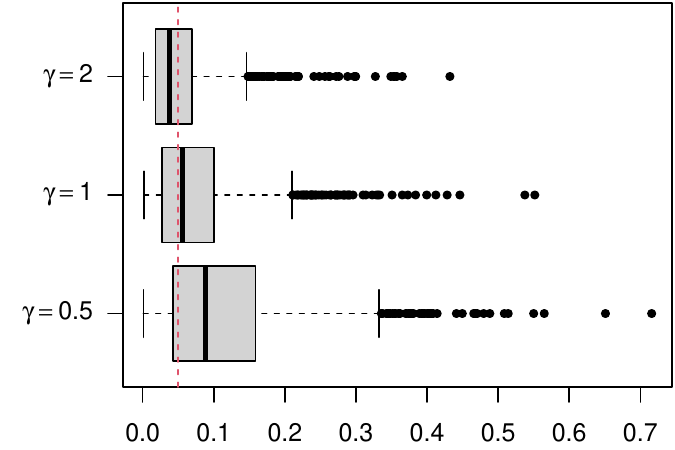}
		\caption{Boxplots of the $p$-values of the EHM tests applied to the paleomagnetic pole positions reported by \cite{Schmidt1976}. For each value of~$\gamma$, the $p$-values were computed using $A=1{,}000$ independently generated artificial samples. The significance level~$\alpha=5\%$ is shown for reference (vertical dashed line).}
		\label{FigDependenceOnSample}
	\end{center}
\end{figure}

\section{Wrap up and perspectives for future research}
\label{secwrapup}

We introduced a class of goodness-of-fit tests for the spherical Cauchy model by combining the sample M\"{o}bius mean with a rotation-invariant projection-based test of spherical uniformity. The resulting statistics are exactly distribution-free under the composite null hypothesis, so critical values can be generated from uniform samples without parameter-dependent resampling. Along the way, we developed the population and sample M\"{o}bius means, established their main structural and asymptotic properties, derived the asymptotic null distribution of the test statistics, and studied their behavior under fixed and contiguous alternatives. The analysis shows that estimating the spherical Cauchy parameter removes precisely the degree-one spherical-harmonic component of the underlying uniformity statistic, thereby linking the effect of parameter estimation to the tangent space of the null model. The Monte Carlo results support the finite-sample relevance of the asymptotic theory and indicate favorable power and computational performance relative to a general-purpose competitor. Lastly, we illustrated the practical relevance of our proposal with a real data example.

Several directions for future research emerge from this work. The same construction could be combined with other rotation-invariant tests of spherical uniformity, possibly leading to procedures with improved sensitivity against specific classes of alternatives. It would also be useful to investigate data-driven choices of the weighting measure~$W$ and questions of local optimality. The simulation experiments performed in Section~\ref{secsimus} show that our proposal suffers from the classical curse of dimensionality and loses power as $d$ increases. It would be interesting to theoretically explore this phenomenon and study the limiting behavior of the test statistic when~$d$ diverges to infinity following the path recently paved by \cite{Ebner2025} for uniformity tests. %

For the particular case of $d = 1$, the spherical Cauchy family and the Möbius transformations serve as building blocks for more complex techniques like circular-circular regression \citep{Kato2008} and time-series circular models \citep{Kato2010}. Therefore, the transformation technique presented here could be studied for testing goodness-of-fit to these extensions of the spherical Cauchy model in which the parameter varies with covariates or dependence is present. More broadly, a similar approach could be adopted to build goodness-of-fit procedures for other transformation models on directional or non-Euclidean sample spaces.

\appendix

\section{Proofs for Section~\ref{secinvariance}}
\label{Appsecinvariance}

\begin{proof}[Proof of Theorem~\ref{th:invariance}]
	Fix $x_1, \ldots, x_n \in \mathbb{S}^d$, $\varphi \in \mathbb{B}^{d+1}$, and $R \in \mathrm{SO} (d + 1)$. Take $z_i = \mathcal{M}_{R, \varphi} (x_i)$ for $i=1, \ldots, n$, and let $\hat{\phi}_n^{\ast} = \hat{\phi}_n (z_1, \ldots, z_n)$. First note that 
	equivariance of~$\hat{\phi}_n$, namely
	\begin{equation*}
		\hat{\phi}_n^{\ast}
		=
		\hat{\phi}_n \big( \mathcal{M}_{R, \varphi} (x_1), \ldots, \mathcal{M}_{R, \varphi} (x_n) \big)
		=
		\mathcal{M}_{R, \varphi} \big( \hat{\phi}_n (x_1, \ldots, x_n) \big),
	\end{equation*}
	yields
	$$
	\mathcal{M}_{I, -\varphi} 
	(        R^T \hat{\phi}_n^{\ast})
	=
	\mathcal{M}_{I, -\varphi} 
	\big\{
	\mathcal{M}_{I, \varphi}( \hat{\phi}_n (x_1, \ldots, x_n))
	\}
	=
	\hat{\phi}_n (x_1, \ldots, x_n),
	$$
	hence also
	\begin{equation}
		\label{aa}
		\mathcal{M}_{I, \varphi} 
		(       -R^T \hat{\phi}_n^{\ast})
		=
		- \hat{\phi}_n (x_1, \ldots, x_n).
	\end{equation}
	We then consider three cases.
	\vspace{2mm}
	
	(i) Assume first that~$\varphi \neq 0$ and $\hat{\phi}_n^{\ast} \neq 0$. Then, using~\eqref{eq:phi>1} above, Lemma~2.1 from \cite{Kato2020}, then \eqref{eq:phi>1} again,  yields
	\begin{align}
		\mathcal{M}_{I, - \hat{\phi}_n^{\ast} } (z_i)
		&= 
		T_{\hat{\phi}_n^{\ast}} 
		\big\{
		\mathcal{M}_{I, -\hat{\phi}_n^{\ast}/ \| \hat{\phi}_n^{\ast} \|^2} (z_i) \big\}
		\nonumber
		\\
		&=
		T_{\hat{\phi}_n^{\ast}} \big\{ \mathcal{M}_{I, -\hat{\phi}_n^{\ast}/ \| \hat{\phi}_n^{\ast} \|^2} \big(  \mathcal{M}_{R, \varphi} (x_i) \big) \big\}
		\nonumber
		\\
		&=
		T_{\hat{\phi}_n^{\ast}} 
		\big\{
		\mathcal{M}_{R_n, \mathcal{M}_{I, \varphi} (- R^T \hat{\phi}_n^{\ast}) / \| \mathcal{M}_{I, \varphi} (- R^T \hat{\phi}_n^{\ast}) \|^2 }   (x_i) \big\}
		\nonumber
		\\
		&=
		T_{\hat{\phi}_n^{\ast}} R_n
		\big\{
		\mathcal{M}_{I, \mathcal{M}_{I, \varphi} (- R^T \hat{\phi}_n^{\ast}) / \| \mathcal{M}_{I, \varphi} (- R^T \hat{\phi}_n^{\ast}) \|^2 }   (x_i) \big\}
		\nonumber
		\\
		&=
		T_{\hat{\phi}_n^{\ast}} R_n T_{\mathcal{M}_{I, \varphi} (- R^T \hat{\phi}_n^{\ast})} 
		\big\{\mathcal{M}_{I, \mathcal{M}_{I, \varphi} (- R^T \hat{\phi}_n^{\ast})} (x_i) \big\}
		,
		\label{asi}
	\end{align}
	where $R_n \in \mathrm{SO} (d + 1)$ depends\footnote{The explicit value of $R_n$ is provided by \cite{Kato2020} in Lemma~2.1, but it is not relevant for this proof.} on $R$, $\varphi$ and $\hat{\phi}_n^{\ast}$. Using~(\ref{aa}), we obtain
	$$
	\mathcal{M}_{I, - \hat{\phi}_n^{\ast} } (z_i)
	=
	O_n 
	\big\{\mathcal{M}_{I, - \hat{\phi}_n (x_1, \ldots, x_n)} (x_i) \big\}
	,
	$$
	where $O_n \in \mathrm{SO} (d + 1)$ depends  on $R$, $\varphi$ and $\hat{\phi}_n^{\ast}$. Invariance of~$\mathcal{U}_n$ under rotations thus entails that
	$
	\Dstat_n (x_1, \ldots, x_n)
	=
	\Dstat_n \big( z_1, \ldots, z_n \big)$.
	\vspace{2mm}
	
	(ii) Assume now that $\varphi = 0$. Then, the comment below Lemma~2.1 in \cite{Kato2020} implies that
	\begin{equation*}
		\mathcal{M}_{I, -\hat{\phi}_n^{\ast}} (z_i)
		=
		\mathcal{M}_{I, -\hat{\phi}_n^{\ast}} \big(  \mathcal{M}_{R, 0} (x_i) \big)
		=
		\mathcal{M}_{R, -R^T \hat{\phi}_n^{\ast}} (x_i)
		=
		R
		\mathcal{M}_{I, -R^T \hat{\phi}_n^{\ast}} (x_i)
		,
	\end{equation*}
	so that~(\ref{aa}) yields (recall that~$\mathcal{M}_{I,0}$ is the identity map)
	\begin{equation*}
		\mathcal{M}_{I, -\hat{\phi}_n^{\ast}} (z_i)
		=
		R
		\big\{		\mathcal{M}_{I, - \hat\phi_n(x_1,\ldots,x_n)} (x_i)
		\big\}
		.
	\end{equation*}
	Invariance of $\mathcal{U}_n$ thus again ensures that~$\Dstat_n (x_1, \ldots, x_n)
	=
	\Dstat_n \big( z_1, \ldots, z_n \big)$.
	\vspace{2mm}
	
	(iii) Finally, assume that $\hat{\phi}_n^{\ast} = 0$. In this case,
	\begin{equation*}
		\mathcal{M}_{I, -\hat{\phi}_n^{\ast}} (z_i)
		=
		\mathcal{M}_{I, \varphi} (x_i)
		=
		\mathcal{M}_{I, \mathcal{M}_{I, \varphi} (- R^T \hat{\phi}_n^{\ast})} (x_i)
		=
		\mathcal{M}_{I, - \hat\phi_n(x_1,\ldots,x_n)} (x_i)
		,
	\end{equation*}
	so that we again obtain~$
	\Dstat_n (x_1, \ldots, x_n)
	=
	\Dstat_n \big( z_1, \ldots, z_n \big)$.
	\vspace{2mm}
	
	Since we obtained that~$
	\Dstat_n (x_1, \ldots, x_n)
	=
	\Dstat_n \big( z_1, \ldots, z_n \big)$ in all three cases, the proof is complete.
\end{proof}

\begin{proof}[Proof of Corollary~\ref{CorDF}]
	Letting $U_i := \mathcal{M}_{I , - \phi} (X_i)$, $i=1, \ldots, n$, it follows from~\eqref{eq:CauchyUnif} that \mbox{$U_1, \ldots, U_n$} form a random sample from the uniform distribution over~$\mathbb{S}^d$. Since Theorem~\ref{th:invariance} implies that
	$
	\Dstat_n (X_1, \ldots, X_n)
	=
	\Dstat_n (U_1, \ldots, U_n)
	$, the result is proved.
\end{proof}

\section{Proofs for Section~\ref{sec:estimation}}
\label{Appsec:estimation}

The proof of Proposition~\ref{prop:Mobunique} requires the following preliminary result.

\begin{Lem}
	\label{Lem:Mobunique}
	Let ${\rm P}$ be a probability measure on~$\mathbb{S}^d$ and~$X$ be a random vector with distribution~${\rm P}$. Assume that~$\mathrm{P}[X=x]<\frac12$ for all~$x\in \mathbb{S}^{d}$. Then, there exists~$a\in(0,1)$ such that
	\begin{equation}
		\label{caP}
		c_a(\mathrm{P})
		:=
		\sup_{v \in \mathbb{S}^d}
		\mathrm{P}[v^T X  > a]
		<
		\frac{1}{2}
		.
	\end{equation}
\end{Lem}

\begin{proof}
	Assume ad absurdum that for all~$a\in(0,1)$, we have~$c_a(\mathrm{P})\geq \tfrac12$. Then, for any positive integer~$k$, there exists~$v_k\in \mathbb{S}^{d}$ with
	\begin{equation}
		\label{preDCT}
		\mathrm{P}\bigg[v_k^T X  > 1-\frac{1}{k} \bigg]
		>
		\frac{1}{2} - \frac{1}{k}
		.
	\end{equation}
	Since~$\mathbb{S}^{d}$ is compact, the sequence~$(v_k)$ admits a subsequence~$(v_{k_\ell})$ converging in~$\mathbb{S}^d$. Denote the limit as~$v_\star$. For any~$x\in\mathbb{S}^d$, we have
	$$
	\lim_{\ell\to\infty}
	\mathbbm{1}\bigg[v_{k_\ell}^T x  > v_{k\ell}^T v_{\star} -\frac{1}{k_\ell}
	\bigg]
	=
	\mathbbm{1}\big[x = v_{\star}
	\big]
	,
	$$
	so that
	the Lebesgue Dominated Convergence Theorem and~\eqref{preDCT} provide
	$$
	\mathrm{P}\big[X = v_{\star}
	\big]
	=
	\lim_{\ell\to\infty}
	\mathrm{P}\bigg[ v_{k_\ell}^T X  > v_{k\ell}^T v_{\star} -\frac{1}{k_\ell}
	\bigg]
	\geq
	\frac12
	,
	$$
	which is a contradiction.
\end{proof}

\begin{proof}[Proof of Proposition~\ref{prop:Mobunique}]
	First, we show that every solution of~\eqref{eq:Mobiusmean} is a strict local maximum of~$G:\mathbb{B}^{d+1}\to\R$. A routine application of the Lebesgue Dominated Convergence Theorem shows that~$G$ is twice (continuously) differentiable, with Hessian matrix
	\begin{equation*}
		H  G (\varphi)
		=
		4 
		\E \bigg[ \frac{(X - \varphi) (X - \varphi)^T}{ \| X - \varphi \|^4} \bigg]
		- 2 
		\E \bigg[ \frac{1}{ \| X - \varphi \|^2} \bigg]
		I_{d+1}
		- 4 \frac{\varphi \varphi^T}{(1 - \|  \varphi \| ^2)^2}
		- \frac{2}{1 - \|  \varphi \| ^2} I_{d+1}
		.
	\end{equation*}
	Let $\phi^{\ast}$ be an arbitrary solution of~\eqref{eq:Mobiusmean}. By construction, $Z = \mathcal{M}_{I, - \phi^{\ast}} (X)$ satisfies
	$$
	\E [Z] = 0
	\quad
	\text{and}
	\quad
	\frac{ Z + \phi^{\ast}}{1 - \| \phi^{\ast} \|^2}
	=
	\frac{ X - \phi^{\ast}}{\| X - \phi^{\ast} \|^2}
	,
	$$
	so that standard algebra shows that
	\begin{equation*}
		H  G (\phi^{\ast})
		=
		\frac{4}{(1 - \|  \phi^{\ast} \| ^2)^2} \big( \E[ Z Z^T] - I_{d+1} \big).
	\end{equation*}
	
	For any~$v \in \mathbb{S}^d$, the Cauchy--Schwarz inequality ensures that $\E [(Z^T v)^2] \leq \E [ \| Z \|^2] = 1$. If the equality occurs, then $Z\in\{v,-v\}$ with~$\rm P$-probability one, hence~$\max({\rm P}[Z=v],{\rm P}[Z=-v])\geq \frac12$, which entails that
	$$
	\max({\rm P}[X=\mathcal{M}_{I, - \phi^{\ast}}^{-1}(v)],{\rm P}[X=\mathcal{M}_{I, - \phi^{\ast}}^{-1}(-v)])\geq \frac12
	.
	$$
	Since this contradicts the assumptions of the proposition, we have $\E [(Z^T v)^2] < 1$ for any~$v \in \mathbb{S}^d$. Consequently, $HG (\phi^{\ast})$ is a negative definite matrix, so that~$\phi^{\ast}$ is a strict local maximum of~$G$.
	
	Now, we are going to show that~\eqref{eq:Mobiusmean} has a unique solution applying Theorem~2.1 by \cite{Gabrielsen1986} to $f(\varphi):=\exp( G (\varphi))$  with $S =\mathbb{B}^{d+1}$. It is clear that~$f$ satisfies Assumptions~(a)--(b) of that theorem, and we have shown above that it also satisfies Assumption~(d). Therefore, it only remains to prove that Assumption~(c) holds too. To do so, we apply Lemma~\ref{Lem:Mobunique} to pick~$a\in(0,1)$ such that the quantity~$c_a({\rm P})$ in Equation~\eqref{caP} is strictly smaller than~$\frac12$. Since
	$$
	\log \| X - \varphi \|^{2} 
	\geq
	\log ( 1+\| \varphi \|^{2} -2\| \varphi \|)
	=
	2\log (1 - \| \varphi \|)
	,
	$$
	it then follows that, for any $\varphi \neq 0$ and~$a\in(0,1)$,
	\begin{align}
		G (\varphi)
		=&
		\log (1 + \| \varphi \|) + \log (1 - \| \varphi \|) - \E[
		\log \| X - \varphi \|^{2}
		]
		\nonumber
		\\
		= &
		\log (1 + \| \varphi \|) + \log (1 - \| \varphi \|)
		- \E \bigg[
		\mathbbm{1} \bigg[ \frac{\varphi^T X}{\| \varphi \|}  > a \bigg]		\log \| X - \varphi \|^{2} 
		\bigg]
		\nonumber
		\\
		&- \E \bigg[
		\mathbbm{1} \bigg[ \frac{\varphi^T X}{\| \varphi \|}  \leq a \bigg]
		\log \| X - \varphi \|^{2} 
		\bigg]
		\nonumber
		\\
		\leq&
		\log (1 + \| \varphi \|) 
		+
		\log (1 - \| \varphi \|) 
		\bigg(1 - 2 \mathrm{P} \bigg[ \frac{\varphi^T X}{\| \varphi \|}  > a \bigg]
		\bigg)
		\nonumber
		\\
		&-
		\log (1 + \| \varphi \|^2 - 2 a \| \varphi \|)
		\mathrm{P} \bigg[ \frac{\varphi^T X}{\| \varphi \|}  \leq a \bigg]
		\nonumber
		\\
		\leq&
		\log (1 + \| \varphi \|) +  \log (1 - \| \varphi \|) 
		(1-2c_a({\rm P}))
		+
		|
		\log (1 + \| \varphi \|^2 - 2 a \| \varphi \|)
		|
		.
		\label{tobeusedinGC}
	\end{align}
	Since~$c_a({\rm P})<\frac12$, we have that~ $\lim_{\| \varphi \| \to 1^{-}} G(\varphi) = - \infty$, so that~$f(\varphi)=\exp( G(\varphi))$ satisfies Assumption~(c) in Theorem~2.1 from \cite{Gabrielsen1986}. Applying this theorem establishes the result.
\end{proof}

\begin{proof}[Proof of Proposition~\ref{Prop:ML}]
	(i) Fix pairwise different $x_1,\ldots,x_n\in\mathbb{S}^d$ with $n\geq 3$. Then, the empirical probability measure assigns at most probability $\frac{1}{n}(<\frac12)$ to any singleton, so that the result follows from Proposition~\ref{prop:Mobunique}.
	(ii). If~${\rm P}$ is non-atomic, then 
	$$
	{\rm P}[ X_1=X_2 ]
	=
	\int_{\mathbb{S}^d}
	{\rm P}[ X_1=x|X_2=x ] \, d{\rm P}(x)
	=
	\int_{\mathbb{S}^d}
	{\rm P}[ X_1=x ] \, d{\rm P}(x)
	=
	0
	.
	$$
	Thus, the~$X_i$'s are pairwise different almost surely, so that the result follows from Part~(i).
\end{proof}

\begin{proof}[Proof of Theorem~\ref{TheorIvarianceMobMean}]
	Fix~$n\geq 3$, $(x_1,\ldots,x_n)\in{\cal X}^n$, $R \in \mathrm{SO} (d + 1)$, and $\varphi \in \mathbb{B}^{d+1}$. For any~$i=1,\ldots,n$, let~$z_i := \mathcal{M}_{R,  \varphi} ( x_i )$. Since~$\mathcal{M}_{R, \varphi}$ is a one-to-one from~$\mathbb{S}^d$ to itself, we have that $(z_1,\ldots,z_n)\in{\cal X}^n$, so that Proposition~\ref{Prop:ML}(i) entails that~$\hat{\phi}_{n}=\hat{\phi}_{n}(x_1,\ldots,x_n)$ and~$\hat{\phi}_{n}^*=\hat{\phi}_{n}(z_1,\ldots,z_n)$ are both well-defined and are the unique solutions of
	\begin{equation}
		\label{eq:meanzetas}
		\sum_{i = 1}^{n} \mathcal{M}_{I, - \hat{\phi}_n} (x_i)
		=
		0
		\quad
		\textrm{ and }
		\quad
		\sum_{i = 1}^{n} \mathcal{M}_{I, - \hat{\phi}_n^{\ast} } (z_i)
		=
		0
		,
	\end{equation}
	respectively.     We then consider three cases.
	\vspace{2mm}
	
	(i) Assume first that~$\varphi \neq 0$ and~$\hat{\phi}_n^{\ast} \neq 0$. Then, the same computation as in~(\ref{asi}) provides
	$$
	\mathcal{M}_{I, - \hat{\phi}_n^{\ast} } (z_i)
	=
	T_{\hat{\phi}_n^{\ast}} R_n T_{\mathcal{M}_{I, \varphi} (- R^T \hat{\phi}_n^{\ast})} 
	\big\{\mathcal{M}_{I, \mathcal{M}_{I, \varphi} (- R^T \hat{\phi}_n^{\ast})} (x_i) \big\}
	,
	\quad
	i=1, \ldots, n
	,
	$$
	where $R_n \in \mathrm{SO} (d + 1)$ is a rotation matrix that depends on $R$, $\varphi$ and $\hat{\phi}_n^{\ast}$. Thus, \eqref{eq:meanzetas} yields
	\begin{equation*}
		\sum_{i = 1}^{n} \mathcal{M}_{I, - \hat{\phi}_n} (x_i)
		=
		0
		=
		\sum_{i = 1}^{n} \mathcal{M}_{I, \mathcal{M}_{I, \varphi} (- R^T \hat{\phi}_n^{\ast}) } (x_i)
		.
	\end{equation*}
	Consequently, 
	$
	- \hat{\phi}_n 
	=
	\mathcal{M}_{I, \varphi} (- R^T \hat{\phi}_n^{\ast}) 
	$,
	or equivalently
	\begin{equation}
		\label{aswanted}
		\hat{\phi}_n^{\ast}
		=
		-R\mathcal{M}_{I, -\varphi} (- \hat{\phi}_n)
		=
		\mathcal{M}_{R,\varphi} (\hat{\phi}_n)
		,
	\end{equation}
	as we wanted to show. 
	\vspace{2mm}
	
	(ii) Assume now that $\varphi = 0$. Then, the comment below Lemma~2.1 in \cite{Kato2020} implies that
	\begin{equation*}
		\mathcal{M}_{I, -\hat{\phi}_n^{\ast}} (z_i)
		=
		\mathcal{M}_{I, -\hat{\phi}_n^{\ast}} \big(  \mathcal{M}_{R, 0} (x_i) \big)
		=
		\mathcal{M}_{R, -R^T \hat{\phi}_n^{\ast}} (x_i)
		=
		R
		\mathcal{M}_{I, -R^T \hat{\phi}_n^{\ast}} (x_i)
		,
	\end{equation*}
	so that~\eqref{eq:meanzetas} yields
	\begin{equation*}
		\sum_{i = 1}^{n} \mathcal{M}_{I, - \hat{\phi}_n} (x_i)
		=
		0
		=
		\sum_{i = 1}^{n} 
		\mathcal{M}_{I, -R^T \hat{\phi}_n^{\ast}} (x_i)
		.
	\end{equation*}
	This entails that 
	$
	\hat{\phi}_n
	=
	R^T \hat{\phi}_n^{\ast}
	$, which, since~$\mathcal{M}_{I,0}$ is the identity map, rewrites again as~(\ref{aswanted}).
	\vspace{2mm}
	
	(iii) Finally, assume that $\hat{\phi}_n^{\ast} = 0$. In this case,
	$		\mathcal{M}_{I, -\hat{\phi}_n^{\ast}} (z_i)
	=z_i=
	\mathcal{M}_{I,\varphi} (x_i)
	$,
	so that
	\begin{equation*}
		\sum_{i = 1}^{n} \mathcal{M}_{I, - \hat{\phi}_n} (x_i)
		=
		0
		=
		\sum_{i = 1}^{n} 
		\mathcal{M}_{I,\varphi} (x_i)
		,
	\end{equation*}
	hence~$\hat{\phi}_n=-\varphi$. Therefore,
	$		\mathcal{M}_{R, \varphi}(\hat{\phi}_n)
	=
	\mathcal{M}_{R, \varphi}( - \varphi)
	=
	0
	=
	\hat{\phi}_n^*
	$, which establishes~(\ref{aswanted}).
	\vspace{2mm}
	
	Since we proved~(\ref{aswanted}) in all three cases, the proof is complete.
\end{proof}

\begin{proof}[Proof of Proposition~\ref{PropBahadur}]
	First, we are going to prove that $\hat{\phi}_n$ is a consistent estimator of~$\phi$ by relying on Corollary 3.2.3(ii) from \cite{Vaart2023} (hereafter, \citetalias{Vaart2023}). For any~$\varphi \in \mathbb{B}^{d+1}$, define
	\begin{equation}
		m_{\varphi} (x)
		=
		\log \bigg(
		\frac{1 - \| \varphi \|^2}{\| x - \varphi \|^2}
		\bigg),
		\qquad
		x \in \mathbb{S}^d
		.
	\end{equation}	
	By construction,
	$$
	\phi =
	\argmax_{ \varphi \in \mathbb{B}^{d+1} } 
	\E [m_{\varphi} (X)],
	\quad
	\text{and}
	\quad
	\hat{\phi}_n
	=
	\argmax_{ \varphi \in \mathbb{B}^{d+1} } 
	\frac{1}{n} \sum_{i = 1}^n m_{\varphi} (X_i).
	$$
	By the Lebesgue Dominated Convergence Theorem, the mapping $\varphi \mapsto \E [m_{\varphi} (X)]$ is continuous. Since~$\rm P$ is non-atomic, Proposition~\ref{prop:Mobunique} implies that this mapping achieves its unique maximum at~$\phi$.
	
	Now, we are going to show that, for any compact set $K \subset \mathbb{B}^{d+1}$, the class $\{ m_{\varphi} \colon \varphi \in K \}$ is Glivenko-Cantelli. Fix such a compact set~$K \subset \mathbb{B}^{d+1}$ and take $M = \max_{\varphi \in K} \| \varphi \|(<1)$. For any $x \in \mathbb{S}^d$ and~$\phi, \theta \in K$, one has
	\begin{equation}
		\label{diffm}
		\vert m_{\varphi} (x) - m_{\theta} (x) \vert
		\leq
		\bigg\vert
		\log \bigg(
		\frac{1 - \| \varphi \|^2}{1 - \| \theta \|^2}
		\bigg)
		\bigg\vert
		+
		2
		\bigg\vert
		\log \bigg(
		\frac{ \| x - \theta \|}{ \| x - \varphi \| }
		\bigg)
		\bigg\vert.
	\end{equation}
	On one side,
	\begin{align*}
		\log \bigg(
		\frac{1 - \| \varphi \|^2}{1 - \| \theta \|^2}
		\bigg)
		=&
		\log \bigg(
		1 + \frac{ \| \theta \|^2 - \| \varphi \|^2}{1 - \| \theta \|^2}
		\bigg)
		\\
		\leq&
		\frac{( \| \theta \| + \| \varphi \|) (\| \theta \| - \| \varphi \|)}{1 - \| \theta \|^2} 
		\\
		\leq&
		\frac{2}{1 - M} \| \theta - \varphi \|.
	\end{align*}
	By symmetry,
	\begin{equation*}
		-
		\log \bigg(
		\frac{1 - \| \varphi \|^2}{1 - \| \theta \|^2}
		\bigg)
		=
		\log \bigg(
		\frac{1 - \| \theta \|^2}{1 - \| \varphi \|^2}
		\bigg)
		\leq
		\frac{2}{1 - M} \| \theta - \varphi \|.
	\end{equation*}
	On the other side,
	\begin{equation*}
		\log \bigg(
		\frac{ \| x - \theta \|}{ \| x - \varphi \| }
		\bigg)
		\leq
		\log \bigg(
		1 + \frac{ \| \theta - \varphi \|}{ \| x - \varphi \| }
		\bigg)
		\leq
		\frac{ \| \theta - \varphi \|}{ \| x - \varphi \| }
		\leq
		\frac{ \| \theta - \varphi \|}{ 1 - M }
	\end{equation*}
	and, by symmetry,
	\begin{equation*}
		-
		\log \bigg(
		\frac{ \| x - \theta \|}{ \| x - \varphi \| }
		\bigg)
		=
		\log \bigg(
		\frac{ \| x - \varphi \|}{ \| x - \theta \| }
		\bigg)
		\leq
		\frac{ \| \theta - \varphi \|}{ 1 - M }.
	\end{equation*}
	Thus, (\ref{diffm}) yields
	\begin{equation}
		\label{eq:Lips}
		\vert m_{\varphi} (x) - m_{\theta} (x) \vert
		\leq
		\frac{4}{1 - M} \| \theta - \varphi \|
		\quad
		\forall x \in \mathbb{S}^d
		\quad
		\forall \varphi, \theta \in K
		.
	\end{equation}
	Let $\varepsilon>0$ and take $\eta = (1 - M) \varepsilon/8$. Since $K$ is  compact, there exists a finite subset $C \subset K$ such that
	$$
	K \subset \bigcup_{c \in C} B_{\eta}(c),
	$$
	where~$B_{\eta}(c)$ denotes the open ball with center~$c$ and radius~$\eta$. Therefore, for any $\varphi \in K$, there exists a point $c_{\varphi} \in C$ such that $\| \varphi - c_{\varphi} \| < \eta$, hence such that
	$$
	\int_{\mathbb{S}^d}
	\vert m_{\varphi} (x) - m_{c_{\varphi}} (x) \vert
	\, d{\rm P}(x)
	\leq
	\sup_{x\in \mathbb{S}^d}
	\vert m_{\varphi} (x) - m_{c_{\varphi}} (x) \vert
	<
	\frac{\varepsilon}{2}
	.
	$$
	Thus, 
	$
	\lbrace
	[ m_{c} - \frac{\varepsilon}{2}, m_{c} + \frac{\varepsilon}{2}] \colon c \in C
	\rbrace
	$
	is a finite collection of $\varepsilon$-brackets for the $L_1({\rm P})$-norm that covers $\{ m_{\varphi} \colon \varphi \in K \}$. Consequently, the Glivenko--Cantelli Theorem \citep[Theorem 2.4.1]{Vaart2023} ensures that the class $\{ m_{\varphi} \colon \varphi \in K \}$ is Glivenko--Cantelli.
	
	Finally, we will show that the sequence~$\hat{\phi}_n$ is uniformly tight. Recall the notation 
	$$
	c_a(\mathrm{P})
	:=
	\sup_{v \in \mathbb{S}^d}
	\mathrm{P}[v^T X  > a]
	$$
	from Lemma~\ref{Lem:Mobunique} and use this lemma to pick~$a \in (0, 1)$ such that $c_a(\mathrm{P})<1/2$.    Applying~\eqref{tobeusedinGC} at~${\rm P}={\rm P}_n$ yields that
	$$
	\frac{1}{n} \sum_{i = 1}^n m_{\varphi} (X_i)
	\leq
	\log (1 + \| \varphi \|) +  \log (1 - \| \varphi \|) 
	(1-2c_a({\rm P}_n))
	+
	|
	\log (1 + \| \varphi \|^2 - 2 a \| \varphi \|)
	|
	.
	$$
	If
	\begin{equation}
		\label{Cn}        
		c_a({\rm P}_n)
		\leq
		\frac12 \Big( 
		c_a({\rm P}) +\frac12 \Big)
		,
	\end{equation}
	then
	$$
	1-2c_a({\rm P}_n)
	\geq
	\frac12
	-
	c_a({\rm P})
	(>0)
	,
	$$
	so that
	$$
	\frac{1}{n} \sum_{i = 1}^n m_{\varphi} (X_i)
	\leq
	\log (1 + \| \varphi \|) +  \log (1 - \| \varphi \|) 
	\Big(
	\frac12
	-
	c_a({\rm P})
	\Big)
	+
	|
	\log (1 + \| \varphi \|^2 - 2 a \| \varphi \|)
	|
	.
	$$
	
	For any~$t \in (0, 1)$, let then
	$$
	H (t) = \log (1 +t) + \log (1 - t)
	\Big(
	\frac12
	-
	c_a({\rm P})
	\Big)
	+ \big\vert \log (1 + t^2 - 2 a t ) \big\vert
	.
	$$
	Since $\lim_{t \to 1^{-}} H(t) = - \infty$, there exists~$\tau \in (0, 1)$ such that $H (t) < -1$ for all~$t > \tau$. Hence, if~(\ref{Cn}) is satisfied, then all $\varphi \in \R^{d+1}$ with~$\tau <\| \varphi \| < 1$ satisfy
	\begin{equation*}
		\frac{1}{n} \sum_{i = 1}^n m_{\varphi} (X_i)
		\leq
		H (\| \varphi \|) < -1
		.
	\end{equation*}
	Since~$m_{0} (x) = 0$ for all $x \in \mathbb{S}^d$, we conclude that if~(\ref{Cn}) is satisfied, then~$\| \hat{\phi}_n \| \leq \tau$. Since~$c_a({\rm P}_n)\to c_a({\rm P})$ in probability (the collection of spherical caps on~$\mathbb{S}^d$ is indeed a Vapnik–Chervonenkis class), it follows that
	\begin{eqnarray*}
		\mathrm{P}[ \| \hat{\phi}_n \| > \tau]
		&\leq & 
		\mathrm{P} \bigg[
		c_a({\rm P}_n)
		>
		\frac12 \Big( 
		c_a({\rm P}) +\frac12 \Big)
		\bigg]
		\\[2mm]
		& = &
		\mathrm{P} \bigg[
		c_a({\rm P}_n)-c_a({\rm P})
		>
		\frac12 \Big( 
		\frac12-c_a({\rm P}) \Big)
		\bigg]
		\\[2mm]
		& \to &
		0
		.
	\end{eqnarray*}
	Thus, the sequence~$\hat{\phi}_n$ is uniformly tight. Hence, Corollary 3.2.3(ii) from \citetalias{Vaart2023} establishes that~$\hat{\phi}_n\to \phi$ in probability.

	Now, note that the function $(x,\varphi)\mapsto m_{\varphi} (x) = \log ( 1 - \| \varphi \|^2 ) - \log  \| x - \varphi \|^2$ is a $\mathcal{C}^{\infty}$ at any~$(x,\varphi)$ with~$\varphi\in\mathbb{B}^{d+1}$ and $x \neq \varphi$. Hence, the function
	\begin{equation*}
		(x,h)
		\mapsto 
		L(x, h) :=
		\left\lbrace
		\begin{array}{lr}
			\dfrac{m_{\phi + h} (x) - m_{\phi} (x) - h^T \nabla_{\varphi} m_{\varphi} (x)|_{\varphi=\phi} }{\| h \|} & \quad \text{if } h \neq 0 \\
			0 & \quad \text{if } h = 0
		\end{array}
		\right.
	\end{equation*}
	is continuously differentiable on 
	$$
	\{x \in \R^{d+1} \colon \| x \| > (1 + \| \phi \|)/2 \}
	\times 
	\{h \in \R^{d+1} \colon \| h \| < (1 - \| \phi \|)/2 \}
	.
	$$
	Fix $x \in \mathbb{S}^d$ and $h \in \bar{B}_{r(\phi)}(0)$, the closed ball with center~$0$ and radius~$r(\phi):=(1 - \| \phi \|)/4$. The mean value theorem ensures that there exists a point $h^{\ast} \in \bar{B}_{r(\phi)}(0)$ such that
	\begin{equation*}
		L(x, h)
		=
		h^T \nabla_{h} L(x, h^{\ast}).
	\end{equation*}
	Since the function $\nabla_{h} L$ is continuous in the compact set $\mathbb{S}^{d} \times \bar{B}_{r(\phi)}(0)$, there exists a positive constant~$M$ such that~$		\vert L(x, h) \vert \leq M \| h \|
	$.
	Consequently, for all $\| h \| \leq r(\phi)$,
	\begin{equation*}
		\E \big[ \{m_{\phi + h} (X_1) - m_{\phi} (X_1) - h^T \nabla_{\varphi} m_{\varphi} (X_1)|_{\varphi=\phi} \}^2 \big]
		=
		\| h \|^2  \E[ \{ L(X_1, h) \}^2 ]
		\leq M^2 \| h \|^4
		.
	\end{equation*}
	Moreover, it was established in the proof of Proposition~\ref{prop:Mobunique} that, when~$P$ is non-atomic, the map $G(\varphi) = \E [m_{\varphi} (X)]$ is twice continuously differentiable and the Hessian matrix at~$\phi$ is
	\begin{equation*}
		H G(\phi)
		=
		\frac{4}{(1 - \|  \phi \| ^2)^2} ( \E[ \mathcal{M}_{I, - \phi} (X_1) \mathcal{M}_{I, - \phi} (X_1)^T ] - I ) = \frac{2}{(1 - \|  \phi \| ^2)} W_{\phi},
	\end{equation*}
	which is a negative definite matrix.
	The result then follows from Example~3.2.25 in \citetalias{Vaart2023} taking into account~\eqref{eq:Lips} and the fact that $\nabla_{\varphi} m_{\varphi} (x)|_{\varphi=\phi} = 2(1 - \|  \phi \| ^2)^{-1}\mathcal{M}_{I, - \phi} (x)$.
\end{proof}

\section{Proofs for Section~\ref{secNull}}
\label{AppsecNull}  

The proof of Theorem~\ref{TheorAsymptDistr} ultimately relies on the classical theory of V-statistics with estimated parameters derived by \cite{Wet1987}. However, directly applying their results would lead to stronger assumptions on the measure~$W$ in Theorem~\ref{TheorAsymptDistr}. This is caused by the particular shape of the cumulative distribution function $F_d$ defined in \eqref{FdExpr}. Note that $F_d$ is not a Lipschitz function for $d = 1$ (its derivative goes to $+\infty$ at $-1$ and $1$). Due to this, the Lipschitz assumption~(2.12) in \cite{Wet1987} does not hold unless we impose some condition on the behavior of $W$ around~$0$ and~$1$.

The following result provides a modern formulation of Theorem~2.16 from \cite{Wet1987} that replaces the problematic assumption with an empirical process version of it, thus avoiding that issue.

\begin{Theor}
	\label{TheorRandles}
	Let~$X_1,\ldots,X_n$ be a random sample from a probability measure~${\rm P}_{\phi}$ on~$(\mathcal X,\mathcal A)$, where~$\phi$ belongs to an open subset~$\Gamma$ of~$\mathbb R^p$. Let~$M$ be a finite measure on~$(\mathcal T,\mathcal{B})$. 
	Let $g:\mathcal X\times\mathcal{T}\times\Gamma\to\mathbb R$ be a measurable function satisfying 
	\begin{equation}
		\label{gAssum}
		\sup_{(x,t,\gamma)\in \mathcal X\times\mathcal T\times\Gamma}
		|g(x,t;\gamma)|
		<
		\infty
		\quad
		\textrm{ and }
		\quad
		\int_{\mathcal T} {\rm E}[\{g(X_1,t;\gamma)\}^2]
		\,
		dM(t)
		<
		\infty
		\quad
		\forall
		\gamma\in\Gamma.
	\end{equation}
	Let
	$$
	\hat{V}_n
	=
	\frac{1}{n}
	\sum_{i,j=1}^n
	h(X_i,X_j;\hat{\phi})
	,
	\quad
	\textrm{ with }
	\,
	h(x,y;\gamma)
	=
	\int_{\mathcal{T}} 
	g(x,t;\gamma)
	g(y,t;\gamma)
	\,
	dM(t)
	,
	$$
	where the estimator~$\hat{\phi}=\hat{\phi}_n$ of~$\phi$ satisfies
	\begin{equation}
		\label{BahadurAssum}
		\sqrt{n}
		(\hat{\phi}-\phi)
		=
		\frac{1}{\sqrt{n}}
		\sum_{i=1}^n
		\alpha(X_i)
		+
		o_{\rm P}(1)
		,
	\end{equation}
	with a function~$\alpha(\cdot)$ such that~${\rm E}[\alpha(X_1)]=0$ and~${\rm E}[\|\alpha(X_1)\|^2]<\infty$. Letting~$\mu_\phi(t;\gamma):={\rm E}[g(X_1,t;\gamma)]$, assume that~$\mu_\phi(t;\phi)=0$ and that there exists a vector~$\nabla_\gamma\mu_\phi(t;\phi)$ such that
	\begin{equation}
		\label{GradFin}
		\int_{\mathcal{T}} 
		\|
		\nabla_\gamma\mu_\phi(t;\phi)
		\|^2
		\,
		dM(t)
		<
		\infty
	\end{equation}
	and
	\begin{equation}
		\label{intt1}
		\frac{1}{\|\gamma-\phi\|^2}
		\int_{\mathcal{T}} 
		\big\{
		\mu_\phi(t;\gamma)
		-
		(\gamma-\phi)^T
		\nabla_\gamma\mu_\phi(t;\phi)
		\big\}^2
		dM(t)
		\to 0
	\end{equation} 
	as~$\gamma$ converges to~$\phi$. 
	Assume further that
	\begin{equation}
		\label{stochequicont}
		\sup_{t \in \mathcal{T}}
		\bigg\vert
		\frac{1}{\sqrt{n}}
		\sum_{i=1}^n
		\big\{
		g(X_i,t;\hat{\phi}) 
		-
		g(X_i,t;\phi)
		- 
		\mu_\phi(t,\hat{\phi})
		\big\}
		\bigg\vert
		=
		o_{\rm P}(1)
		.
	\end{equation}
	Let
	\begin{eqnarray*}
		\lefteqn{
			\hspace{0mm}
			V_n
			:=
			\frac{1}{n}
			\sum_{i,j=1}^n
			h_*(X_i,X_j;\phi)
			,
			\quad \textrm{with }
			\
			h_*(x,y;\phi)
			:=
			\int_{\mathcal{T}} 
			\big\{
			g(x,t;\phi) + (\nabla_\gamma\mu_\phi(t;\phi))^T \alpha(x)
			\big\}
			}
			\\[2mm]
			& & 
			\hspace{70mm}
			\times 
			\big\{
			g(y,t;\phi) + (\nabla_\gamma\mu_\phi(t;\phi))^T \alpha(y)
			\big\}
			\,
			dM(t)
			.
		\end{eqnarray*}
		Then, 
		$$
		\hat{V}_n
		=
		V_n
		+o_{\rm P}(1)
		\stackrel{\mathcal{L}}{\to}
		\sum_{k=1}^\infty \lambda_k Q_k
		,
		$$
		where~$\lambda_1,\lambda_2,\ldots$ are the eigenvalues of the operator~$A$ defined by
		$$
		(Aq)(x)
		=
		\int_{\mathcal{X}}
		h_*(x,y;\phi)
		q(y)
		\,
		d{\rm P}(y)
		$$
		and
		where~$Q_1,Q_2,\ldots$ are \mbox{i.i.d.} $\chi^2_1$. 
	\end{Theor}

	\begin{proof}[Proof of Theorem~\ref{TheorRandles}]
		Let
		$$
		Y_n
		:=
		\int_{\mathcal{T}} 
		\bigg(
		\frac{1}{\sqrt{n}}
		\sum_{i=1}^n
		\big\{
		g(X_i,t;\phi) + 
		\mu_\phi(t,\hat{\phi})
		\big\}
		\bigg)^2
		\,
		dM(t)
		.
		$$
		Decomposing~$g(x,t;\phi) + (\nabla_\gamma\mu_\phi(t;\phi))^T \alpha(x)$ into~$\{g(x,t;\phi) + \mu_\phi(t;\hat{\phi})\}+\{(\nabla_\gamma\mu_\phi(t;\phi))^T \alpha(x)-\mu_\phi(t;\hat{\phi})\}$, we have
		\begin{eqnarray*}
			\lefteqn{
				V_n
				=
				Y_n
				+
				\int_{\mathcal{T}} 
				\bigg(
				\frac{1}{\sqrt{n}}
				\sum_{i=1}^n
				\big\{
				(\nabla_\gamma\mu_\phi(t;\phi))^T \alpha(X_i)-\mu_\phi(t;\hat{\phi})
				\big\}
				\bigg)^2
				\,
				dM(t)
			}
			\\[2mm]
			& & 
			\hspace{1mm}
			+
			\frac{2}{n}
			\int_{\mathcal{T}} 
			\bigg(
			\sum_{i=1}^n
			\big\{
			g(X_i,t;\phi) + 
			\mu_\phi(t,\hat{\phi})
			\big\}
			\bigg)
			\bigg(
			\sum_{j=1}^n
			\big\{
			(\nabla_\gamma\mu_\phi(t;\phi))^T \alpha(X_j)-\mu_\phi(t;\hat{\phi})
			\big\}
			\bigg)
			\,
			dM(t)
			\\[2mm]
			& & 
			\hspace{1mm}
			=: Y_n+T_{n1}+2T_{n2}
			,
		\end{eqnarray*}
		say. Letting
		$$
		U_n
		:=
		\sqrt{n}
		(\hat{\phi}-\phi)
		-
		\frac{1}{\sqrt{n}}
		\sum_{i=1}^n
		\alpha(X_i)
		,
		$$
		(\ref{BahadurAssum}) and \eqref{GradFin} yield
		\begin{eqnarray*}
			T_{n1}
			&=&
			\int_{\mathcal{T}} 
			\big(
			\sqrt{n}
			\big\{
			(\nabla_\gamma\mu_\phi(t;\phi))^T (\hat{\phi}-\phi)-\mu_\phi(t;\hat{\phi})
			\big\}
			-
			(\nabla_\gamma\mu_\phi(t;\phi))^T
			U_n
			\big)^2
			dM(t)
			\\[2mm]
			&\leq &
			2n
			\int_{\mathcal{T}} 
			\big\{
			(\nabla_\gamma\mu_\phi(t;\phi))^T (\hat{\phi}-\phi)-\mu_\phi(t;\hat{\phi})
			\big\}^2
			dM(t)
			+
			2
			\|U_n\|^2
			\int_{\mathcal{T}} 
			\|\nabla_\gamma\mu_\phi(t;\phi)\|^2
			dM(t)
			\\[2mm]
			&= &
			2n\|\hat{\phi}-\phi\|^2
			\mathbbm{1}[\hat{\phi}\neq \phi]
			\int_{\mathcal{T}} 
			\frac{
				\big\{
				\mu_\phi(t;\hat{\phi})
				-
				(\hat{\phi}-\phi)^T
				\nabla_\gamma\mu_\phi(t;\phi)
				\big\}^2}{\|\hat{\phi}-\phi\|^2}
			\,
			dM(t)
			+
			o_{\rm P}(1)
			.
		\end{eqnarray*}
		Since~(\ref{intt1}) shows that 
		$$
		\ell(h)
		:=
		\int_{\mathcal{T}} 
		\frac{
			\big\{
			\mu_\phi(t;\phi+h)
			-
			h^T
			\nabla_\gamma\mu_\phi(t;\phi)
			\big\}^2}{\|h\|^2}
		\,
		dM(t)
		=
		o(1)
		$$
		as~$h$ converges to zero and~$\hat{\phi}-\phi=o_{\rm P}(1)$, we have that~$\ell(\hat{\phi}-\phi)=o_{\rm P}(1)$. Since~$\sqrt{n}(\hat{\phi}-\phi)=O_{\rm P}(1)$, we conclude that~$T_{n1}=o_{\rm P}(1)$. 
		
		Let us turn to~$T_{n2}$. The Cauchy--Schwarz inequality provides
		$$
		|T_{n2}|
		\leq
		\sqrt{Y_{n}}\sqrt{T_{n1}}
		=
		\sqrt{Y_{n}} o_{\rm P}(1)
		.
		$$
		Now,
		\begin{eqnarray*}
			Y_n
			&\leq &
			2
			\int_{\mathcal{T}} 
			\bigg(
			\frac{1}{\sqrt{n}}
			\sum_{i=1}^n
			\big\{
			g(X_i,t;\phi) 
			+ 
			(\nabla_\gamma\mu_\phi(t;\phi))^T \alpha(X_i)
			\big\}
			\bigg)^2
			dM(t)
			+
			2
			T_{n1}
			\\[2mm]
			&\leq &
			4
			\int_{\mathcal{T}} 
			\bigg\{
			\frac{1}{\sqrt{n}}
			\sum_{i=1}^n
			g(X_i,t;\phi) 
			\bigg\}^2
			dM(t)
			+
			4
			\bigg\|
			\frac{1}{\sqrt{n}}
			\sum_{i=1}^n
			\alpha(X_i)
			\bigg\|^2
			\int_{\mathcal{T}} 
			\|\nabla_\gamma\mu_\phi(t;\phi)\|^2
			\,
			dM(t)
			+
			o_{\rm P}(1)
			\\[2mm]
			&\leq &
			\frac{4}{n}
			\sum_{i,j=1}^n
			h(X_i,X_j;\phi)
			+
			O_{\rm P}(1)
			\\[2mm]
			&= &
			O_{\rm P}(1)
			,
		\end{eqnarray*}
		where we have applied (\ref{BahadurAssum}) and \eqref{GradFin} and where we used the classical consistency result for $V$-statistics (that follows from, e.g., Lemmas~5.2.1A and~5.7.3 in \citealp{Serfling1980}). It follows that~$T_{n2}=o_{\rm P}(1)$, hence that
		\begin{equation}
			\label{P1}
			V_n-Y_n=o_{\rm P}(1)
			.
		\end{equation}
		We now prove that, similarly, 
		\begin{equation}
			\label{P2}
			\hat{V}_n-Y_n=o_{\rm P}(1)
			.
		\end{equation}
		Proceeding as above, write
		\begin{eqnarray*}
			\lefteqn{
				\hat{V}_n
				=
				Y_n
				+
				\int_{\mathcal{T}} 
				\bigg(
				\frac{1}{\sqrt{n}}
				\sum_{i=1}^n
				\big\{
				g(X_i,t;\hat{\phi}) 
				-
				g(X_i,t;\phi)
				- 
				\mu_\phi(t,\hat{\phi})
				\big\}
				\bigg)^2
				\,
				dM(t)
			}
			\\[2mm]
			& & 
			\hspace{1mm}
			+
			\frac{2}{n}
			\int_{\mathcal{T}} 
			\bigg(
			\sum_{i=1}^n
			\big\{
			g(X_i,t;\phi) + 
			\mu_\phi(t,\hat{\phi})
			\big\}
			\bigg)
			\bigg(
			\sum_{j=1}^n
			\big\{
			g(X_i,t;\hat{\phi}) 
			-
			g(X_i,t;\phi)
			- 
			\mu_\phi(t,\hat{\phi})
			\big\}
			\bigg)
			\,
			dM(t)
			\\[2mm]
			& & 
			\hspace{1mm}
			=: Y_n+S_{n1}+2S_{n2}
			,
		\end{eqnarray*}
		say. Since~$|S_{n2}|\leq \sqrt{Y_{n}}\sqrt{S_{n1}}$, proving~(\ref{P2}) only requires to show that~$S_{n1}=o_{\rm P}(1)$. The fact that $M$ is a finite measure and Equation~\eqref{stochequicont} yield
		\begin{align*}
			S_{n1}
			= &
			\int_{\mathcal{T}} 
			\bigg(
			\frac{1}{\sqrt{n}}
			\sum_{i=1}^n
			\big\{
			g(X_i,t;\hat{\phi}) 
			-
			g(X_i,t;\phi)
			- 
			\mu_\phi(t,\hat{\phi})
			\big\}
			\bigg)^2 dM(t)
			\\
			\leq &
			\sup_{t \in \mathcal{T}}
			\bigg\vert
			\frac{1}{\sqrt{n}}
			\sum_{i=1}^n
			\big\{
			g(X_i,t;\hat{\phi}) 
			-
			g(X_i,t;\phi)
			- 
			\mu_\phi(t,\hat{\phi})
			\big\}
			\bigg\vert^2
			M (\mathcal{T})
			=
			o_{\rm P}(1)
		\end{align*}
		
		Obviously, (\ref{P1})--(\ref{P2}) imply that~$\hat{V}_n=V_n+o_{\rm P}(1)$. Since the asymptotic distribution of~$V_n$ follows from Theorem~6.4.1B in \cite{Serfling1980}, the result is proved.
	\end{proof}

	We can now prove Theorem~\ref{TheorAsymptDistr}.

	\begin{proof}[Proof of Theorem~\ref{TheorAsymptDistr}]
		Fix~$\phi\in\mathbb{B}^{d+1}$ and let
		$$
		\Dstat_n^\phi
		:=
		\Dstat_n^\phi(X_1,\ldots,X_n)
		:=
		\frac{1}{n} 
		\sum_{i,j = 1}^n  
		h^\phi(X_i,X_j)
		,
		$$
		based on the kernel
		$
		h^\gamma(x,y)
		:=
		\psi(\mathcal{M}_{I, -\gamma} (x),\mathcal{M}_{I, -\gamma} (y))
		$.
		Note that~$D_n=D_n^{\hat{\phi}_n}$, where~$\hat{\phi}_n$ is the sample M\"{o}bius mean (see~(\ref{TestStatistDn})). We prove the result by applying Theorem~\ref{TheorRandles}. 
		Using~(\ref{DefPsi}), the kernel~$h^\gamma$ writes
		$$
		h^\gamma(x,y)
		=
		\int_{\mathbb{S}^d\times [-1,1]} 
		g(x,u,s;\gamma)
		g(y,u,s;\gamma)
		\,
		dM(u,s)
		,
		$$
		with
		\begin{equation}
			\label{gdeff}
			g(x,u,s;\gamma)
			:=
			\mathbbm{1}[ u^T \mathcal{M}_{I,-\gamma}(x) \leq s ]-F_d(s)
		\end{equation}
		and with~$M$ the product measure defined by
		$$
		M(A\times B)
		=
		\int_{A\times B}
		dM(u,s)
		:=
		\int_{A}
		\int_{B}
		dW(F_d(s))
		d\nu_d(u)
		$$
		for all measurable sets~$A\subset \mathbb{S}^d$ and~$B\subset [-1,1]$. 
		Since~$g$ and~$M$ are bounded, (\ref{gAssum}) holds trivially.
		Moreover, Proposition~\ref{PropBahadur} directly entails that~$\hat{\phi}_n$ satisfies~(\ref{BahadurAssum}) with
		$$
		\alpha(x)
		:=
		- W_{\phi}^{-1} 
		\mathcal{M}_{I, - \phi}(x)
		.
		$$
		Note that the quantity~$\mu_\phi(u,s;\gamma)$ in~(\ref{Defmu}) writes
		$$
		\mu_\phi(u,s;\gamma)
		=
		{\rm E}_\phi[g(X,u,s;\gamma)]
		=
		{\rm E}[g(\mathcal{M}_{I,\phi}(Z),u,s;\gamma)]
		,
		$$
		where~$Z\sim \mathrm{Unif} (\mathbb{S}^d)$ and~${\rm E}_\phi$ is the expectation under~$X\sim \mathrm{SC} (\phi)$. Thus,
		\begin{equation}
			\label{exprpremu}
			\mu_\phi(u,s;\gamma)
			=
			\int_{\mathbb{S}^d}
			\big\{
			\mathbbm{1}[ u^T \mathcal{M}_{I,-\gamma}(\mathcal{M}_{I,\phi}(v)) \leq s ]-F_d(s)
			\big\}
			\,
			d\nu_d(v)
			,
		\end{equation}
		and~$\mu_\phi(u,s;\phi)=0$. It thus remains to check Conditions~(\ref{GradFin})--(\ref{stochequicont}).
		\vspace{3mm}
		
		{\sc Conditions~(\ref{GradFin})--(\ref{intt1}).}
		For any~$\gamma\in \mathbb{B}^{d+1}$, the map~$v\mapsto T_\gamma(v):=\mathcal{M}_{I,-\gamma}(\mathcal{M}_{I,\phi}(v))$ is a diffeomorphism of~$\mathbb{S}^d$, so that 
		\begin{equation}
			\label{exprmu}
			\mu_\phi(u,s;\gamma)
			=
			-F_d(s)
			+
			\int_{C_{u,s}}
			J_\gamma(w)
			\,
			d\nu_d(w)
			,
		\end{equation}
		where~$C_{u,s}:=\{w\in\mathbb{S}^d:u^Tw\leq s\}$ is a fixed spherical cap and~$J_\gamma(w):={\rm Jac}_{\mathbb S^d}T_\gamma^{-1}(w)$ is \label{pageJomega} the surface area Jacobian of~$T_\gamma^{-1}$ on~$\mathbb{S}^d$. Note that the collection of diffeomorphisms
		$$
		\{T_\gamma:\gamma\in \mathbb{B}^{d+1}\}
		$$
		is smooth. With a closed ball~$B:=\{\gamma\in \mathbb{B}^{d+1}:\|\gamma-\phi\|\leq \delta\}$ small enough so that~$B\subset \mathbb{B}^{d+1}$, the map~$(w,\gamma)\mapsto J_\gamma(w)$ is then continuously differentiable in~$\mathbb{S}^d\times B$. From a routine domination argument (using compactness and smoothness), it follows from~(\ref{exprmu}) that~(i) $\gamma\mapsto \mu_\phi(u,s;\gamma)$ is continuously differentiable on~$B$ for any~$u\in \mathbb{S}^d$ and~$s\in [-1,1]$, and that~(ii) 
		$$
		\sup_{u\in \mathbb{S}^d}
		\sup_{s\in [-1,1]}
		\sup_{\gamma\in B}
		\| \nabla_\gamma \mu_\phi(u,s;\gamma) \|
		<
		\infty
		.
		$$
		Since the measure~$M$ is finite, Condition~(\ref{GradFin}) holds. Using dominated convergence again, we also obtain that the map
		$$
		\theta\mapsto
		L(\theta)
		:=
		\int_{\mathbb{S}^d}
		\int_{-1}^1
		\nabla_\gamma \mu_\phi(u,s;\theta)(\nabla_\gamma \mu_\phi(u,s;\theta))^T
		\,
		dW(F_d(s))
		d\nu_d(u)
		,
		$$
		is thus continuous at~$\phi$.
		
		We now check Condition~\eqref{intt1}. Since~$\mu_\phi(u,s;\phi)=0$, we need to show that 
		\begin{equation}
			\label{intt}
			\int_{\mathbb{S}^d}
			\int_{-1}^1
			\bigg\{
			\frac{\mu_\phi(u,s;\gamma)-\mu_\phi(u,s;\phi)}{\|\gamma-\phi\|}
			-
			\frac{(\gamma-\phi)^T}{\|\gamma-\phi\|}
			\nabla_\gamma\mu_\phi(u,s;\phi)
			\bigg\}^2
			dW(F_d(s))
			d\nu_d(u)
			\to 0
		\end{equation}
		as~$\gamma$ converges to~$\phi$. We prove this by establishing that the convergence in~(\ref{intt}) holds along any sequence~$(\gamma_n)$ converging to~$\phi$ with~$(\gamma_n-\phi)/\|\gamma_n-\phi\|\to \xi$ for some~$\xi\in\mathbb{S}^d$ (it is easy to check that the compactness of~$\mathbb{S}^d$ implies that it is enough to consider such sequences). From~(i) above, we have that for any~$u,s$,
		$$
		\frac{\mu_\phi(u,s;\gamma_n)-\mu_\phi(u,s;\phi)}{\|\gamma_n-\phi\|}
		-
		\frac{(\gamma_n-\phi)^T}{\|\gamma_n-\phi\|}
		\nabla_\gamma\mu_\phi(u,s;\phi)
		\to 0
		,
		$$
		hence
		\begin{equation}
			\label{intn}
			\frac{\mu_\phi(u,s;\gamma_n)-\mu_\phi(u,s;\phi)}{\|\gamma_n-\phi\|}
			-
			\xi^T \nabla_\gamma\mu_\phi(u,s;\phi)
			\to 0
			.
		\end{equation}
		Now, (i) also entails that
		$$
		\mu_\phi(u,s;\gamma_n)-\mu_\phi(u,s;\phi)
		=
		\int_0^1
		(\gamma_n-\phi)^T
		\nabla_\gamma\mu_\phi(u,s;\phi+\lambda(\gamma_n-\phi))
		\,
		d\lambda
		.
		$$
		Thus,
		\begin{eqnarray*}
			\lefteqn{
				\hspace{-2mm}
				\int_{\mathbb{S}^d}
				\int_{-1}^1
				\bigg\{
				\frac{\mu_\phi(u,s;\gamma_n)-\mu_\phi(u,s;\phi)}{\|\gamma_n-\phi\|}
				\bigg\}^2
				dW(F_d(s))
				d\nu_d(u)
			}
			\\[2mm]
			& &
			\hspace{-1mm}
			=
			\frac{1}{\|\gamma_n-\phi\|^2}
			\int_{\mathbb{S}^d}
			\int_{-1}^1
			\bigg\{
			\int_0^1
			(\gamma_n-\phi)^T
			\nabla_\gamma\mu(t;\phi+\lambda(\gamma_n-\phi))
			\,
			d\lambda
			\bigg\}^2
			dW(F_d(s))
			d\nu_d(u)
			\\[2mm]
			& &
			\hspace{-1mm}
			\leq
			\frac{1}{\|\gamma_n-\phi\|^2}
			\int_{\mathbb{S}^d}
			\int_{-1}^1
			\int_0^1
			\big\{
			(\gamma_n-\phi)^T
			\nabla_\gamma\mu(t;\phi+\lambda(\gamma_n-\phi))
			\big\}^2
			\,
			d\lambda
			\,
			dW(F_d(s))
			d\nu_d(u)
			,
		\end{eqnarray*}
		so that Fubini's theorem and the Lebesgue Dominated Convergence Theorem provide
		\begin{eqnarray*}
			\lefteqn{
				\hspace{-6mm}
				\int_{\mathbb{S}^d}
				\int_{-1}^1
				\bigg\{
				\frac{\mu_\phi(u,s;\gamma_n)-\mu_\phi(u,s;\phi)}{\|\gamma_n-\phi\|}
				\bigg\}^2
				dW(F_d(s))
				d\nu_d(u)
			}
			\\[2mm]
			& &
			\hspace{3mm}
			\leq
			\frac{1}{\|\gamma_n-\phi\|^2}
			(\gamma_n-\phi)^T
			\bigg\{\int_0^1
			L(\phi+\lambda(\gamma_n-\phi))
			\,d\lambda\bigg\}
			(\gamma_n-\phi)
			\\[2mm]
			& &
			\hspace{3mm}
			=
			\xi^T
			L(\phi)
			\xi
			+
			o(1)
			=
			\int_{\mathbb{S}^d}
			\int_{-1}^1
			\{ \xi^T \nabla_\gamma\mu(u,s;\phi)\}^2
			\,
			dW(F_d(s))
			d\nu_d(u)
			+
			o(1)
			.
		\end{eqnarray*}
		Remembering~(\ref{intn}), it then follows from Proposition~2.29 in~\cite{vanderVaart1998AsymptoticStatistics} that
		$$
		\int_{\mathbb{S}^d}
		\int_{-1}^1
		\bigg\{
		\frac{\mu_\phi(u,s;\gamma_n)-\mu_\phi(u,s;\phi)}{\|\gamma_n-\phi\|}
		-
		\xi^T \nabla_\gamma\mu_\phi(u,s;\phi)
		\bigg\}^2
		dW(F_d(s))
		d\nu_d(u)
		\to 0
		,
		$$
		which in turn entails that
		$$
		\int_{\mathbb{S}^d}
		\int_{-1}^1
		\bigg\{
		\frac{\mu_\phi(u,s;\gamma_n)-\mu_\phi(u,s;\phi)}{\|\gamma_n-\phi\|}
		-
		\frac{(\gamma_n-\phi)^T}{\|\gamma_n-\phi\|}
		\nabla_\gamma\mu_\phi(u,s;\phi)
		\bigg\}^2
		dW(F_d(s))
		d\nu_d(u)
		\to 0
		.
		$$
		This establishes~(\ref{intt}) and, therefore, Condition~\eqref{intt1} holds.
		\vspace{3mm}

		{\sc Condition~(\ref{stochequicont}).}
		For any~$(u,s,\gamma)
		\in  \mathbb{S}^d\times[-1,1]\times \mathbb{B}^{d+1}$, the map
		$$
		x
		\mapsto \|x-\gamma\|^2 (u^T \mathcal{M}_{I,-\gamma}(x)-s)
		$$ 
		is the restriction to~$\mathbb{S}^d$ of a polynomial function on~$\R^{d+1}$ with degree at most~$m_d$ for some positive integer~$m_d$. Since the collection of such polynomials is a finite-dimensional vector space, Lemma~2.6.16 from \citetalias{Vaart2023}   implies that
		$$
		\{
		x\mapsto \|x-\gamma\|^2 (u^T \mathcal{M}_{I,-\gamma}(x)-s)
		:
		(u,s,\gamma)
		\in 
		\mathbb{S}^d\times[-1,1]\times \mathbb{B}^{d+1}
		\}
		$$
		is a VC-subgraph class. Therefore, applying Lemma~2.6.20(viii) from \citetalias{Vaart2023} with the fixed monotone function~$r\mapsto \mathbbm{1}[r\leq 0]$ directly entails that
		$$
		\mathcal{F}
		:=
		\{
		x\mapsto 
		f(\cdot,u,s;\gamma)
		:=
		\mathbbm{1}[u^T \mathcal{M}_{I,-\gamma}(x)-s\leq 0]
		:
		(u,s,\gamma)
		\in 
		\mathbb{S}^d\times[-1,1]\times \mathbb{B}^{d+1}
		\}
		$$
		is a VC-subgraph class, too (note that~$\|x-\gamma\|>0$ for~$x\in\mathbb{S}^d$ and~$\gamma\in \mathbb{B}^{d+1}$). As usual, Theorems~2.6.7 and~2.5.2 from \citetalias{Vaart2023} then entail that~$\mathcal F$ is universally Donsker. Note that, using standard notation from empirical process theory, we have
		$$
		\sqrt n({\rm P}_n-{\rm P}_\phi)
		\big\{
		f(\cdot,u,s;\hat{\phi}_n)
		-
		f(\cdot,u,s;\phi) 
		\big\}
		=
		\frac{1}{\sqrt{n}}
		\sum_{i=1}^n
		\big\{
		g(X_i, u, s;\hat{\phi}) 
		-
		g(X_i, u, s;\phi)
		- 
		\mu_\phi(u, s,\hat{\phi})
		\big\}
		,
		$$
		so we only need to show that
		\begin{equation}
			\label{eq:REP}
			\sup_{u \in \mathbb{S}^d, s \in [-1,1]}
			\big\vert
			\sqrt n({\rm P}_n-{\rm P}_\phi)
			\big\{ 
			f(\cdot,u,s;\hat{\phi}_n)
			-
			f(\cdot,u,s;\phi) 
			\big\}
			\big\vert
			=
			o_{\rm P} (1).
		\end{equation}
		We are going to show that~\eqref{eq:REP} holds relying on Theorem~3.13.4 from~\citetalias{Vaart2023}. 
		
		To do so, fix $u\in\mathbb{S}^d$, $s\in[-1,1]$, and~$\delta \in \big(0,  (1 - \| \phi \|)/2 \big)$. Note that, for any~$\gamma\in B_{\delta}(\phi)$,
		\begin{eqnarray*}
			\lefteqn{
				{\rm P}_\phi \big\lbrace
				f(\cdot,u,s;\gamma)-f(\cdot,u,s;\phi)
				\big\rbrace^2
			}
			\\[2mm]
			& & 
			\hspace{7mm}
			=
			\int_{\mathbb{S}^d}
			\big\{
			f(\mathcal{M}_{I,\phi}(v),u,s;\gamma)
			-
			f(\mathcal{M}_{I,\phi}(v),u,s;\phi)
			\big\}^2
			\,
			d\nu_d(v)
			,
		\end{eqnarray*}
		where
		\begin{eqnarray*}
			\lefteqn{
				\hspace{-3mm}
				|
				f(\mathcal{M}_{I,\phi}(v),u,s;\gamma)
				-
				f(\mathcal{M}_{I,\phi}(v),u,s;\phi)
				|
			}
			\\[2mm]
			& & 
			\hspace{3mm}
			=
			|
			\mathbbm{1}[ u^T \mathcal{M}_{I,-\gamma}(\mathcal{M}_{I,\phi}(v)) \leq s ]
			-
			\mathbbm{1}[ u^T v \leq s ]
			|
			\\[2mm]
			& & 
			\hspace{3mm}
			=
			|
			\mathbbm{1}[ u^T T_{\gamma}(v) \leq s ]
			-
			\mathbbm{1}[ u^T v \leq s ]
			|
			\\[2mm]
			& & 
			\hspace{3mm}
			=
			\mathbbm{1}[ u^T T_{\gamma}(v) \leq s , u^T v > s 
			]
			+
			\mathbbm{1}[ u^T T_{\gamma}(v) > s , u^T v \leq s 
			]
			.
		\end{eqnarray*}
		As above, $(u, v,\gamma)\mapsto h(u, v,\gamma):= u^T T_{\gamma}(v)$ is continuously differentiable on~$\mathbb{S}^d \times \mathbb{S}^d \times B_{\delta}(\phi)$, so that 
		$$
		C
		:=
		\sup_{u \in \mathbb{S}^d}
		\sup_{v\in\mathbb{S}^d}
		\sup_{\gamma \in B_{\delta}(\phi)}
		\| \nabla_\gamma h(u,v,\gamma ) \|
		<
		\infty
		.
		$$
		It follows that, for any~$u,v\in\mathbb{S}^d$, 
		$$
		\sup_{\gamma \in B_{\delta}(\phi)}
		| u^T T_{\gamma}(v) - u^T v |
		=
		\sup_{\gamma \in B_{\delta}(\phi)}
		| u^T T_{\gamma}(v) - u^T T_{\phi}(v) |
		<
		C\delta
		.
		$$
		Thus, if~$u^T v>s$ and if there exists some~$\gamma \in B_{\delta} (\phi)$ with~$u^T T_{\gamma}(v)\leq s$, then we must have~$s<u^T v\leq s+C\delta$, that is,
		$$
		\sup_{\gamma \in B_{\delta}(\phi)}
		\mathbbm{1}[ u^T T_{\gamma}(v) \leq s , u^T v > s 
		]
		\leq
		\mathbbm{1}[ s<u^T v \leq s+C\delta]
		.
		$$
		Similarly, if~$u^T v\leq s$ and if there exists~$\gamma \in B_{\delta}(\phi)$ such that~$u^T T_{\gamma}(v)> s$, then we must have~$s-C\delta<u^T v \leq s$, that is,
		$$
		\sup_{\gamma \in B_{\delta}(\phi)}
		\mathbbm{1}[ u^T T_{\gamma}(v) > s , u^T v \leq s 
		]
		\leq
		\mathbbm{1}[ s-C\delta<u^T v \leq s]
		.
		$$
		Consequently,
		\begin{equation}
			\sup_{\gamma\in B_{\delta}(\phi)}
			|
			f(\mathcal{M}_{I,\phi}(v),u,s;\gamma)
			-
			f(\mathcal{M}_{I,\phi}(v),u,s;\phi)
			|
			\leq
			\mathbbm{1}[ s-C\delta<u^T v \leq s+C\delta]
			,
			\label{keybound}
		\end{equation}
		which yields
		\begin{eqnarray*}
			\sup_{\gamma \in B_{\delta}(\phi)}
			{\rm P}_\phi \big\lbrace
			f(\cdot,u,s;\gamma)-f(\cdot,u,s;\phi)
			\big\rbrace^2
			&\leq  &
			\int_{\mathbb{S}^d}
			\mathbbm{1}[ s-C\delta<u^T v\leq s+C\delta]
			\,
			d\nu_d(v)
			\\[2mm]
			&=  &
			F_d( \min(s+C\delta,1)) -F_d( \max(s-C\delta,-1) )
			.
		\end{eqnarray*}
		Since~$F_d$ is continuous, it is uniformly continuous on~$[-1,1]$. So, for any $\varepsilon > 0$, there exists $\delta_{\varepsilon} \in \big(0,  (1 - \| \phi \|)/2 \big)$ such that
		$$
		\sup_{\gamma \in B_{\delta_{\varepsilon}}(\phi)}
		\bigg(
		\sup_{u \in \mathbb{S}^d, s \in [-1, 1]}
		{\rm P}_\phi \big\lbrace
		f(\cdot,u,s;\gamma)-f(\cdot,u,s;\phi)
		\big\rbrace^2
		\bigg)
		\leq
		\varepsilon.
		$$ 
		This implies that%
		$$
		{\rm P}
		\bigg[
		\sup_{u \in \mathbb{S}^d, s \in [-1, 1]}
		{\rm P}_\phi \big\lbrace
		f(\cdot,u,s;\hat{\phi}_n)-f(\cdot,u,s;\phi)
		\big\rbrace^2
		> \varepsilon
		\bigg]
		\leq
		{\rm P}
		[
		\| \hat{\phi}_n - \phi \| > \delta_{\varepsilon}
		]
		\to
		0.
		$$
		Since this holds for any $\varepsilon > 0$, it follows that
		$$
		\sup_{u \in \mathbb{S}^d, s \in [-1, 1]}
		{\rm P}_\phi \big\lbrace
		f(\cdot,u,s;\hat{\phi}_n)-f(\cdot,u,s;\phi)
		\big\rbrace^2
		=
		o_{\rm P} (1).
		$$
		Therefore, Theorem~3.13.4 from~\citetalias{Vaart2023} ensures that \eqref{eq:REP} holds, which in turn implies that Condition~(\ref{stochequicont}) holds. The result thus follows from Theorem~\ref{TheorRandles}.
		\vspace{3mm}
	\end{proof}

	\begin{proof}[Proof of Corollary~\ref{CorolNull}]
		We first show that the kernel~$h_*^0(x,y)$  in Theorem~\ref{TheorAsymptDistr} admits the expression in~(\ref{Exprhzerostar}). First note that, from~(\ref{exprpremu}), we have, using the notation~$f(x;\phi)$ from~(\ref{eq:densCauchy}),
		\begin{eqnarray*}
			\mu_0(u,s;\gamma)
			&=&
			\int_{\mathbb{S}^d}
			\big\{
			\mathbbm{1}[ u^T \mathcal{M}_{I,-\gamma}(v) \leq s ]-F_d(s)
			\big\}
			\,
			d\nu_d(v)
			\\[2mm]
			&=&
			\int_{\mathbb{S}^d}
			\big\{
			\mathbbm{1}[ u^T w \leq s ]-F_d(s)
			\big\}
			\,
			f(w;-\gamma) 
			\,
			d\sigma_d(w)
			,
		\end{eqnarray*}
		since~$\mathcal{M}_{I,-\gamma}(V)\sim {\rm SC}(-\gamma)$ when~$V\sim {\rm Unif}(\mathbb{S}^{d})$; see~(\ref{eq:CauchyUnif}). It follows that 
		\begin{eqnarray*}
			\lefteqn{
				\nabla_\gamma \mu_0(u,s;0)
				=
				\frac{\Gamma\big( \frac{d + 1}{2} \big)}{2 \pi^{(d + 1)/2}}
				\int_{\mathbb{S}^d}
				\mathbbm{1}[ u^T w \leq s ]
				\nabla_\gamma
				\bigg( \frac{1 - \|\gamma \|^2}{\|w + \gamma\|^2} \bigg)^d \Big|_{\gamma=0}
				\,
				d\sigma_d(w)
			}
			\\[2mm]
			& &
			\hspace{3mm}
			=
			-\frac{d\Gamma\big( \frac{d + 1}{2} \big)}{\pi^{(d + 1)/2}}
			\int_{\mathbb{S}^d}
			\mathbbm{1}[ u^T w \leq s ]
			w
			\,
			d\sigma_d(w)
			=
			-2d
			\bigg(
			\int_{\mathbb{S}^d}
			\mathbbm{1}[ u^T w \leq s ]
			u^T w
			\,
			d\nu_d(w)
			\bigg)
			u
			.
		\end{eqnarray*}
		Using~(\ref{FdExpr}), this yields
		$$
		\nabla_\gamma \mu_0(u,s;0)
		=
		-2d
		\bigg(
		\frac{\Gamma(\frac{d+1}{2})}{\sqrt{\pi}\Gamma(\frac{d}{2})}
		\int_{-1}^s
		t (1-t^2)^{(d-2)/2}
		\,
		dt
		\bigg)
		u
		=
		\frac{2 \Gamma(\frac{d+1}{2})}{\sqrt{\pi}\Gamma(\frac{d}{2})}
		(1-s^2)^{d/2}
		u
		,
		$$
		so that~$h_*^0(x,y)$ indeed rewrites as in~(\ref{Exprhzerostar}). 
		
		Since this expression entails that $h_*^0(Rx,Ry)=h_*^0(x,y)$ for all rotations~$R$, the operator~$A$ from Theorem~\ref{TheorAsymptDistr} commutes with the $SO(d+1)$-action on~$L^2(\mathbb S^d,\nu_d)$. 
		Thus, this operator acts as a scalar, $\eta_\ell$ say, on each spherical-harmonic subspace~$\mathcal H_\ell$ of degree~$\ell$ ($\eta_\ell$ is an eigenvalue of~$A$ with multiplicity~$\dim\mathcal H_\ell$).
		We now determine the corresponding eigenvalues. Since Fubini's theorem readily yields
		$$
		(A1)(x)
		=
		\int_{\mathbb{S}^d}
		h_*^0(x,y)
		\,
		d\nu_d(y)
		=
		0
		,
		$$
		constants have eigenvalue~$\eta_0=0$, with multiplicity~$\dim\mathcal H_0=1$. For any degree-one spherical harmonic~$q_v(x)=v^Tx$, Fubini provides
		\begin{eqnarray*}
			\lefteqn{
				(Aq_v)(x)
				=
				\int_{\mathbb{S}^d}
				h_*^0(x,y)
				v^Ty
				\,
				d\nu_d(y)
			}
			\\[2mm]
			& & 
			\hspace{-1mm}
			=
			\int_{\mathbb{S}^d}
			\int_{-1}^1
			\bigg\{
			\mathbbm{1}[ u^T x \leq s ]-F_d(s)
			+ 
			\frac{(d+1) \Gamma(\frac{d+1}{2})}{\sqrt{\pi}d\Gamma(\frac{d}{2})}
			(1-s^2)^{d/2}
			u^T x
			\bigg\}
			C_v(u,s)
			\,
			dW(F_d(s))
			d\nu_d(u)
			,
		\end{eqnarray*}
		where
		\begin{eqnarray*}
			C_v(u,s)
			&:=&
			\int_{\mathbb{S}^d}
			\bigg\{
			\mathbbm{1}[ u^T y \leq s ]-F_d(s)
			+ 
			\frac{(d+1) \Gamma(\frac{d+1}{2})}{\sqrt{\pi}d\Gamma(\frac{d}{2})}
			(1-s^2)^{d/2}
			u^T y
			\bigg\}
			v^T y
			\,
			d\nu_d(y)
			\\[2mm]
			&=&
			\int_{\mathbb{S}^d}
			\mathbbm{1}[ u^T y \leq s ]
			v^T y
			\,
			d\nu_d(y)
			+ 
			\frac{\Gamma(\frac{d+1}{2})}{\sqrt{\pi}d\Gamma(\frac{d}{2})}
			(1-s^2)^{d/2}
			u^T v
			\\[2mm]
			&=&
			\bigg(
			-
			\frac{\Gamma(\frac{d+1}{2})}{\sqrt{\pi}d\Gamma(\frac{d}{2})}
			(1-s^2)^{d/2}
			+ 
			\frac{\Gamma(\frac{d+1}{2})}{\sqrt{\pi}d\Gamma(\frac{d}{2})}
			(1-s^2)^{d/2}
			\bigg)
			u^T v
			=
			0
		\end{eqnarray*}
		Thus,
		$
		Aq_v
		=
		0
		$, so that degree-one spherical harmonics have eigenvalue~$\eta_1=0$,
		with multiplicity~$\dim\mathcal H_1=d+1$ (more precisely, the common eigenvalue~$\eta_0=\eta_1=0$ has multiplicity~$d+2$). Since the correction from~$\psi$ to~$h_*^0$ only affects the degree-one component, the remaining degree-wise eigenvalues are the same as for the kernel~$\psi$, that is,
		$\eta_\ell=\eta_\ell^\psi$ for~$\ell\geq2$.  This proves the result. 
	\end{proof}

	\section{Proofs for Section~\ref{secNonNull}}
	\label{AppsecNonNull}

	\begin{proof}[Proof of Theorem~\ref{TheorConsistency}]
		Since~$\phi$ is the M\"{o}bius mean of ${\rm P}$, it follows that
		\begin{equation}
			\label{pree}
			\E[ \mathcal{M}_{I, -\phi} (X_1)] = 0
			.
		\end{equation}
		Denote then as~${\rm P}^Z$ the common distribution of~$Z_i=\mathcal{M}_{I, -\phi}(X_i)$, $i=1,\ldots,n$; since~$\mathcal{M}_{I, -\phi}$ is a diffeomorphism of~$\mathbb{S}^d$, the probability measure~${\rm P}^Z$ is non-atomic too. Equation~(\ref{pree}) reads~$\E[Z_1]=\E[ \mathcal{M}_{I,0} (Z_1)]=0$,  so that Proposition~\ref{prop:Mobunique} yields that~${\rm P}^Z$ has M\"{o}bius mean zero. %
		The invariance result in Theorem~\ref{th:invariance} then entails that it is sufficient to establish the result for~$\phi=0$.
		
		Assume thus that~$X_1, \ldots, X_n$ is a random sample from~$\rm P$, where~${\rm P}$  is non-atomic and has M\"{o}bius mean zero. First note that the classical consistency result for~$V$-statistics ensures that 
		\begin{equation}
			\frac{1}{n}
			\Dstat_{n}^0
			:=
			\frac{1}{n^2}
			\sum_{i,j=1}^n
			\psi(X_i,X_j) 
			\stackrel{\rm P}{\to}
			e_{\rm P}
			> 0
			.
			\label{pree1}
		\end{equation}
		Now, assume for a moment that
		\begin{equation}
			\label{ToshowConsitency}
			\frac{1}{n}
			(\Dstat_{n} - \Dstat_{n}^0)
			=
			\frac{1}{n^2}
			\sum_{i,j=1}^n
			\big\{ 
			\psi\big(\mathcal{M}_{I, -\hat{\phi}_n} (X_i),\mathcal{M}_{I, -\hat{\phi}_n} (X_j)\big)
			-
			\psi(X_i,X_j) 
			\big\}
			= 
			o_{\rm P} (1)
			.
		\end{equation}
		Then, (\ref{pree1})--(\ref{ToshowConsitency}) entail that
		$
		\frac{1}{n}
		\mathcal{D}_{n}
		\stackrel{\rm P}{\to}
		e_{\rm P}
		$, which readily implies that~$
		{\rm P}[\mathcal{D}_{n} > c]
		\to
		1$ for any~$c>0$. Thus, it only remains to establish~(\ref{ToshowConsitency}).
		
		To do so, we use a similar decomposition as in the proof of Theorem~\ref{TheorRandles}: with the function~$g$ defined in~(\ref{gdeff}), write
		\begin{eqnarray*}
			\lefteqn{
				\frac{1}{n}
				( \Dstat_{n} - \Dstat_{n}^0 )
				=
				\int
				\bigg\lbrace
				\bigg(
				\frac{1}{n}
				\sum_{i=1}^n
				g(X_i,t; \hat{\phi}_n)
				\bigg)^2
				-
				\bigg(
				\frac{1}{n}
				\sum_{i=1}^n
				g(X_i,t; 0)
				\bigg)^2
				\bigg\rbrace
				\,
				dM (t)
			}
			\\[2mm]
			& & 
			\hspace{6mm}
			=
			\int
			\bigg(
			\frac{1}{n}
			\sum_{i=1}^n
			\big\{
			g(X_i,t; \hat{\phi}_n)
			+
			g(X_i,t; 0)
			\big\}
			\bigg)
			\bigg(
			\frac{1}{n}
			\sum_{i=1}^n
			\big\{
			g(X_i,t; \hat{\phi}_n)
			-
			g(X_i,t; 0)
			\big\}
			\bigg)
			\,
			dM (t)
			\\[2mm]
			& & 
			\hspace{6mm}
			=
			\int
			\bigg(
			\frac{1}{n}
			\sum_{i=1}^n
			\big\{
			g(X_i,t; \hat{\phi}_n)
			-
			g(X_i,t; 0)
			\big\}
			\bigg)^2
			dM (t)
			\\[2mm]
			& & 
			\hspace{12mm}
			+
			2
			\int
			\bigg(
			\frac{1}{n}
			\sum_{i=1}^n
			g(X_i,t; 0)
			\bigg)
			\bigg(
			\frac{1}{n}
			\sum_{i=1}^n
			\big\{
			g(X_i,t; \hat{\phi}_n)
			-
			g(X_i,t; 0)
			\big\}
			\bigg)
			\,
			dM (t)
			\\[2mm]
			& & 
			\hspace{6mm}
			=:
			A_n + 2B_n
			,
		\end{eqnarray*}
		say; throughout this proof, integrals with respect to~$t=(s,u)$ are over~$[-1,1]\times \mathbb{S}^d$ and the measure~$M$ is as in the proof of Theorem~\ref{TheorAsymptDistr}.
		Using again the consistency result for~$V$-statistics, we have
		$$
		\int
		\bigg(
		\frac{1}{n}
		\sum_{i=1}^n
		g(X_i,t; 0)
		\bigg)^2
		dM(t)
		=
		\frac{1}{n^2}
		\sum_{i,j=1}^n
		\psi ( X_i, X_j )
		=
		{\rm E}[\psi ( X_1, X_2 )]
		+
		o_{\rm P} (1)
		,
		$$
		so that the Cauchy–Schwarz inequality yields
		\begin{eqnarray*}
			B_n^2
			& \leq &
			\bigg(
			\int
			\bigg(
			\frac{1}{n}
			\sum_{i=1}^n
			g(X_i,t; 0)
			\bigg)^2
			dM(t)
			\bigg)
			\bigg(
			\int
			\bigg(
			\frac{1}{n}
			\sum_{i=1}^n
			\big\{
			g(X_i,t; \hat{\phi}_n)
			-
			g(X_i,t; 0)
			\big\}
			\bigg)^2
			dM (t)
			\bigg)
			\\[2mm]
			& = &
			O_{\rm P} (1) A_n
			.
		\end{eqnarray*}
		As a consequence, it only remains to prove that~$A_n = o_{\rm P} (1)$.
		
		If $\| \hat{\phi}_n \| \leq n^{-1/4}$, then Jensen's inequality provides
		\begin{eqnarray*}
			A_n
			& \leq &
			\sup_{\| \phi \| \leq n^{-1/4}}
			\int
			\bigg(
			\frac{1}{n}
			\sum_{i=1}^n 
			\vert
			g(X_i,t; \phi)
			-
			g(X_i,t; 0)
			\vert
			\bigg)^2
			dM (t)
			\\[2mm]
			& \leq &
			\int
			\bigg(
			\frac{1}{n}
			\sum_{i=1}^n 
			\sup_{\| \phi \| \leq n^{-1/4}}
			\vert
			g(X_i,t; \phi)
			-
			g(X_i,t; 0)
			\vert
			\bigg)^2
			dM (t)
			\\[2mm]
			& \leq &
			\frac{1}{n}
			\sum_{i=1}^n 
			\int
			\sup_{\| \phi \| \leq n^{-1/4}}
			\vert
			g(X_i,t; \phi)
			-
			g(X_i,t; 0)
			\vert^2
			\,
			dM (t)
			\\[2mm]
			& =: & \tilde{A}_n
			,
		\end{eqnarray*}
		say. Using~\eqref{keybound}, it follows that, for a constant~$C > 0$,
		\begin{eqnarray*}
			\tilde{A}_n
			& \leq &
			\frac{1}{n}
			\sum_{i=1}^n 
			\int_{\mathbb{S}^d}
			\int_{-1}^{1}
			\mathbbm{1}[ s-Cn^{-1/4}<u^T X_i \leq s+Cn^{-1/4}]
			\,
			dW(F_d(s))
			d\nu_d(u)
			\\
			& = &
			\int_{-1}^{1}
			\big\{
			F_{d} (s+Cn^{-1/4}) - F_{d} (s-Cn^{-1/4})
			\big\}
			\,
			dW(F_d(s)).
		\end{eqnarray*}
		Since $F_d$ is an absolutely continuous function and $W$ is a bounded measure, the dominated convergence theorem guarantees that~$\tilde{A}_n=o_{\rm P}(1)$. For any~$\varepsilon > 0$, we thus have
		\begin{eqnarray*}
			{\rm P} [A_n > \varepsilon]
			& = &
			{\rm P} [A_n > \varepsilon, \| \hat{\phi}_n \| \leq n^{-1/4}]
			+
			{\rm P} [A_n > \varepsilon, \| \hat{\phi}_n \| > n^{-1/4}]
			\\[1mm]
			& \leq &
			{\rm P} [\tilde{A}_n > \varepsilon]
			+
			{\rm P} [\| \hat{\phi}_n \| > n^{-1/4}]
			\to 0
			,
		\end{eqnarray*}
		where we used that $\| \hat{\phi}_n \| = O_{\rm P} (1/\sqrt{n})$ by Proposition~\ref{PropBahadur}. Therefore, $A_n = o_{\rm P} (1)$ and the result follows.
	\end{proof}
	
	\begin{proof}[Proof of Corollary~\ref{CorConsistency}]
		
		Let~$\rm P$ be an absolutely continuous probability measure on~$\mathbb{S}^d$ with density in $L^2 (\nu_d)$ that is not spherical Cauchy, and let $X \sim \rm P$. Denote by $\phi$ the M\"{o}bius mean of $\rm P$ (which exists and is unique by Proposition~\ref{prop:Mobunique}). By Theorem~\ref{TheorConsistency}, we just need to prove that
		\begin{equation*}
			e_{\rm P}
			=
			\int_{\mathbb{S}^d}
			\int_{-1}^1
			\{{\rm P}[ u^T \mathcal{M}_{I, -\phi} (X) \leq s ]-F_d(s)\}^2
			\,
			dW(F_d(s))
			d\nu_d(u) > 0.
		\end{equation*}
		
		Denote then as~${\rm P}^Z$ the distribution of~$Z=\mathcal{M}_{I, -\phi}(X)$. Since $\rm P$ is not spherical Cauchy, ${\rm P}^Z$ is not uniform. Moreover, as~$\mathcal{M}_{I, -\phi}$ is a diffeomorphism of~$\mathbb{S}^d$, the probability measure~${\rm P}^Z$ has a density, say $f^{Z}$, that belongs to $L^2 (\nu_d)$ too. 
		
		For each~$k\geq1$, choose an eigenfunction~$\Psi_k$ of~$A_\psi$
		associated with~$\lambda_k^\psi$ so that
		$\{\Psi_k\}_{k\geq1}$ forms an orthonormal basis of the orthogonal
		complement of the constant functions in~$L^2(\nu_d)$. The
		Hilbert--Schmidt expansion of~$\psi$ is then
		\begin{equation*}
			\psi(x,y)
			=
			\sum_{k=1}^{\infty}
			\lambda_k^\psi\Psi_k(x)\Psi_k(y)
			\qquad\text{in }L^2(\nu_d\otimes\nu_d)
		\end{equation*}
		and 
		the
		Fourier expansion of~$f^Z$ may be written as
		\begin{equation*}
			f^Z(x)
			=
			1+\sum_{k=1}^{\infty}\alpha_k\Psi_k(x)
			\qquad\text{in }L^2(\nu_d),
		\end{equation*}
		where
		$$
		\alpha_k
		=
		\int_{\mathbb S^d}
		f^Z(x)\Psi_k(x)\,d\nu_d(x).
		$$
		
		Substituting these expansions into the expression
		\begin{equation*}
			e_{\rm P} = \int_{\mathbb{S}^d \times \mathbb{S}^d} \psi (x, y) f^{Z} (x) f^{Z} (y) \,  d\nu_d (x)d\nu_d (y) 
		\end{equation*}
		and using orthonormality yields
		\begin{equation*}
			e_{\rm P}
			= 
			\sum_{k=1}^{\infty}
			\lambda_k^{\psi} \alpha_k^2.
		\end{equation*}
		On one hand, Corollary~5 from \cite{GarciaPortugues2023} ensures that $\lambda_k^\psi > 0$ for all $k \geq 1$. On the other hand, since $f^Z$ is not the density of the uniform distribution, there exists some $k \geq 1$ such that $\alpha_k \neq 0$. These two facts ensure that $e_{\rm P} > 0$, completing the proof.
	\end{proof}

	\begin{proof}[Proof of Proposition~\ref{prop:contiguity}]
		In this proof, all expectations and stochastic convergences are under~${\rm P}\n_0$. Note that 
		$$
		\max_{i=1,\ldots,n} 
		\vert \tau_n (X_i ) \vert 
		\leq 
		\sup_{x \in \mathbb{S}^d} \vert \tau_n (x) \vert 
		=o_{\rm P}(\sqrt{n})
		$$
		as~$n\to\infty$. 
		Since $\log (1 + t) = t - \tfrac12 t^2 + O(t^3)$ as $t \to 0$, Lemma~S.1.5 from \cite{BEJ2026supp} yields
		$$
		\Lambda_n 
		=
		\sum_{i = 1}^n
		\,
		\log
		\bigg(
		1 + \frac{\tau_n (X_i)}{\sqrt{n}} 
		\bigg)
		=
		\frac{1}{\sqrt{n}}  
		\sum_{i = 1}^n \tau_n (X_i)
		-
		\frac{1}{2n} \sum_{i = 1}^n \tau_n^2 (X_i)
		+
		O_{\rm P} \bigg( 
		\frac{1}{n^{3/2}}
		\sum_{i = 1}^n \big|\tau_n^3 (X_i)\big| \bigg).
		$$
		Since
		$$
		\E \Bigg[
		\frac{1}{n^{3/2}}
		\sum_{i = 1}^n \big\vert \tau_n^3 (X_i) \big\vert \Bigg]
		=
		\frac{1}{\sqrt{n}}  
		\E[\vert \tau_n^3 (X_1) \vert]
		\leq
		\frac{1}{\sqrt{n}}  
		\sup_{x \in \mathbb{S}^d} \vert \tau_n^3 (x) \vert \to 0,
		$$
		Markov's inequality implies that
		$$
		\frac{1}{n^{3/2}}
		\sum_{i = 1}^n \big\vert \tau_n^3 (X_i) \big\vert
		=
		o_{\rm P}(1)
		,
		$$
		which in turn entails that
		$$
		\Lambda_n 
		=
		\frac{1}{\sqrt{n}}  
		\sum_{i = 1}^n \tau_n (X_i)
		-
		\frac{1}{2n} \sum_{i = 1}^n \tau_n^2 (X_i)
		+
		o_{\rm P} ( 1 ).
		$$
		In order to prove the result, it is therefore sufficient to show that
		\begin{equation}
			\label{zza1}
			\frac{1}{\sqrt{n}}  
			\sum_{i = 1}^n \tau_n (X_i)
			= \frac{1}{\sqrt{n}}\sum_{i = 1}^n \tau (X_i) + o_{\rm P} (1)
		\end{equation}
		and
		\begin{equation}
			\label{zza2}
			\frac{1}{n} \sum_{i = 1}^n \tau_n^2 (X_i)
			= \frac{1}{n} \sum_{i = 1}^n \tau^2 (X_i) + o_{\rm P} (1)
			.
		\end{equation}
		
		We prove~(\ref{zza1}) and~(\ref{zza2}) by establishing the corresponding convergences in~$L_2$ and in~$L_1$, respectively.
		Since
		$$
		{\rm E}[\tau(X_1)]
		=
		\int_{\mathbb{S}^d}
		\tau(u)
		\,
		d\nu_d(u)
		= 
		\int_{\mathbb{S}^d}
		(\tau(u)-\tau_n(u))
		\,
		d\nu_d(u)
		\to 
		0
		,
		$$ 
		we have~$\E[\tau(X_1)] = 0$. Recalling that~$\E[\tau_n(X_1)] = 0$ for all~$n$, we then obtain
		\begin{eqnarray*}
			\E \Bigg[
			\bigg( \frac{1}{\sqrt{n}}\sum_{i = 1}^n \{\tau_n(X_i)-\tau (X_i)\} \bigg)^2
			\Bigg]
			&=&
			\E \big[ \big\lbrace \tau_n (X_1) -  \tau (X_1) \big\rbrace^2
			\big]
			\\[2mm]
			&        \leq &
			\sup_{x \in \mathbb{S}^d} \vert \tau_n (x) - \tau (x) \vert^2 \to 0,
		\end{eqnarray*}
		which proves~(\ref{zza1}).
		The triangle inequality and the Cauchy--Schwarz inequality yield
		\begin{eqnarray*}
			\lefteqn{
				\hspace{-5mm}
				\E \Bigg[
				\Bigg\vert 
				\frac{1}{n}\sum_{i = 1}^n \{\tau_n^2 (X_i)-\tau^2 (X_i)\} \Bigg\vert
				\Bigg]
				\leq 
				\E [
				\vert \tau_n^2 (X_1) - \tau^2 (X_1) \vert
				]
			}
			\\[2mm]
			& & 
			\hspace{6mm}
			\leq 
			\sqrt{        \E[
				\lbrace \tau_n (X_1) - \tau (X_1) \rbrace^2
				]}
			\sqrt{
				\E[
				\lbrace \tau_n (X_1) + \tau (X_1) \rbrace^2
				]
			}
			.
		\end{eqnarray*}
		Now,
		\begin{equation*}
			\E[
			\lbrace \tau_n (X_1) - \tau (X_1) \rbrace^2
			]
			\leq
			\sup_{x \in \mathbb{S}^d} \vert \tau_n (x) - \tau (x) \vert^2 \to 0,
		\end{equation*}
		and
		\begin{eqnarray*}
			\E[
			\lbrace \tau_n (X_1) + \tau (X_1) \rbrace^2
			]
			&=&
			\E[
			\lbrace 2\tau (X_1) + \tau_n (X_1) - \tau (X_1)\rbrace^2
			]
			\\[2mm]
			&\leq&
			8 \E[
			\tau^2 (X_1)
			]
			+
			2
			\E[
			\lbrace \tau_n (X_1) - \tau (X_1)\rbrace^2
			]
			\\[2mm]
			&=&
			O(1)+o(1)= O (1)
			.
		\end{eqnarray*}
		Consequently,
		\begin{equation*}
			\E \Bigg[
			\Bigg\vert 
			\frac{1}{n}\sum_{i = 1}^n \{\tau_n^2 (X_i)-\tau^2 (X_i)\} \Bigg\vert
			\Bigg]
			=
			o(1),
		\end{equation*}
		which shows~(\ref{zza2}), hence completes the proof.
	\end{proof}

	We turn to the proof of Theorem~\ref{th:localpower}, which will require to consider the Hilbert--Schmidt expansion of the square-integrable symmetric
	kernel~$h_*^0$ in~(\ref{Exprhzerostar}). For each~$\ell\geq2$, let
	$\Psi_{\ell,1},\ldots,\Psi_{\ell,m_\ell}$ be an orthonormal basis of
	$\mathcal H_\ell$ in~$L^2(\nu_d)$, where
	$m_\ell:=\dim(\mathcal H_\ell)$. By Corollary~\ref{CorolNull}, the operator~$A$
	has eigenvalue~$\eta_\ell^\psi$ on~$\mathcal H_\ell$ for every~$\ell\geq2$,
	whereas its eigenvalue is zero on~$\mathcal H_0\oplus\mathcal H_1$. Hence the
	Hilbert--Schmidt expansion of~$h_*^0$ may be written as
	\begin{equation}
		\label{HilbertSchmidt}
		h_*^0(x,y)
		=
		\sum_{\ell=2}^{\infty}
		\eta_\ell^\psi
		\sum_{j=1}^{m_\ell}
		\Psi_{\ell,j}(x)\Psi_{\ell,j}(y)
		\quad
		\textrm{ in } L^2(\nu_d\otimes\nu_d).
	\end{equation}
	We will need the next preliminary result. In the following, $\E_{0}$ and $\E_{\tau_n}$ will stand for the expectations under~${\rm P}^{(n)}_{0}$ and~${\rm P}^{(n)}_{\tau_n}$, respectively.
	
	\begin{Lem}
		\label{LemLocalPowers}
		(i) There exists~$C>0$ such that
		$\sqrt{n}|\E_{\tau_n}[\Psi_{\ell,j}(X_1)]|\leq C$
		for all~$\ell\geq2$, all~$j=1,\ldots,m_\ell$, and all positive integers~$n$.
		(ii) There exists~$C>0$ such that
		$\E_{\tau_n}[(\Psi_{\ell,j}(X_1))^2]\leq C$
		for all~$\ell\geq2$, all~$j=1,\ldots,m_\ell$, and all positive integers~$n$.
		(iii) For any positive integer~$L$, there exists~$C_L>0$ such that
		$\E_{\tau_n}[(\Psi_{\ell,j}(X_1))^4]\leq C_L$
		for all~$\ell=2,\ldots,L$ with~$\eta_\ell^\psi>0$,
		all~$j=1,\ldots,m_\ell$, and all positive integers~$n$.
		(iv) For any~$L$, the collection of random variables~
		$$
		\bigg\{
		S_{L,n}
		:=
		\sum_{\ell=2}^{L}\eta_\ell^\psi
		\sum_{j=1}^{m_\ell}
		\bigg(
		\frac{1}{\sqrt{n}}
		\sum_{i=1}^n
		\Psi_{\ell,j}(X_i)
		\bigg)^2
		,
		\
		n=1,2,\ldots
		\bigg\}
		$$
		is uniformly integrable under~${\rm P}^{(n)}_{\tau_n}$. 
	\end{Lem}

	\begin{proof}
		(i) 
		Since~$\E_{0} [\Psi_{\ell,j} (X_1)] = 0$, we have
		\begin{eqnarray*}
			\E_{\tau_n}[
			\Psi_{\ell,j} (X_1)
			]
			&        =&
			\E_{0}\!\Bigg[
			\Psi_{\ell,j} (X_1)
			\frac{d{\rm P}^{(n)}_{\tau_n}}{d{\rm P}^{(n)}_0}
			\Bigg]
			\\[2mm]
			&=&
			\E_{0}\!\bigg[
			\Psi_{\ell,j} (X_1)
			\bigg(
			1 + \frac{\tau_n (X_1)}{\sqrt{n}} 
			\bigg)
			\bigg]
			\\[2mm]
			&=&
			\frac{1}{\sqrt{n}} 
			\E_{0}[
			\Psi_{\ell,j} (X_1)
			\tau_n (X_1)        ]
			.
		\end{eqnarray*}
		Applying the Cauchy--Schwarz inequality and using $\E_{0} [\Psi_{\ell,j}^2 (X_1)] = 1$ then yields 
		$$
		\sqrt{n}
		|
		\E_{\tau_n}[
		\Psi_{\ell,j} (X_1)
		]
		|
		\leq
		\sqrt{
			\E_{0}[
			\Psi_{\ell,j}^2 (X_1)
			]
		}
		\sqrt{
			\E_{0}[
			\tau_n^2 (X_1)
			]
		}
		\leq
		C
		,
		$$
		where $C:= \sup_{n \in \N}  \E_{0} [ \tau_n^2 (X_1) ]$ is finite since $\tau_n$ converges uniformly to~$\tau$.
		\vspace{3mm}
		
		(ii) Proceeding similarly, 
		\begin{equation*}
			\E_{\tau_n}[
			\Psi_{\ell,j}^2 (X_1)
			]
			=
			\E_{0}\bigg[
			\Psi_{\ell,j}^2 (X_1)
			\bigg(
			1 + \frac{\tau_n (X_1)}{\sqrt{n}} 
			\bigg)
			\bigg]
			\leq 
			C
			\E_{0}[
			\Psi_{\ell,j}^2 (X_1)
			]
			=
			C
			,
		\end{equation*}
		where $C := \sup_{n \in \N} \sup_{u \in \mathbb{S}^d} |1+\tau_n (u)/\sqrt{n}|$ is finite since $\tau_n/\sqrt{n}$ converges uniformly to~$0$. 
		\vspace{-1mm}
		
		(iii)
		Denote as~$\mathcal{I}_L$ the collection of indices~$\ell=2,\ldots,L$ such that~$\eta_\ell^\psi>0$. We may assume that~$\mathcal I_L$ is non-empty (if it is empty, then there is nothing to prove). For any~$\ell\in \mathcal{I}_L$ and any~$j=1,\ldots,m_\ell$, applying the Cauchy--Schwarz inequality in the identity
		$$
		\int_{\mathbb{S}^d} h_*^0(x,y)\Psi_{\ell,j}(y)\,d\nu_d(y)
		=
		\eta_\ell^\psi\Psi_{\ell,j}(x)
		$$
		provides
		$$
		(\Psi_{\ell,j}(x))^2
		\leq
		\frac{1}{(\eta_\ell^\psi)^2}
		\bigg(
		\int_{\mathbb{S}^d} |h_*^0(x,y)|^2\,d\nu_d(y)
		\bigg)
		\bigg(
		\int_{\mathbb{S}^d}\Psi_{\ell,j}^2(y)\,d\nu_d(y)
		\bigg)
		$$
		$\nu_d$-almost everywhere. Since the kernel~$h_*^0$ is bounded, with sup-norm~$\|h_*^0\|_\infty$, say, and since
		$\int_{\mathbb{S}^d}\Psi_{\ell,j}^2(y)\,d\nu_d(y)=1$, this yields
		$$
		(\Psi_{\ell,j}(x))^2
		\leq
		\frac{\|h_*^0\|_\infty^2}{\ell_L^2}
		,
		\quad
		\textrm{with }
		\ell_L := \inf\{\eta_\ell^\psi:\ell\in \mathcal{I}_L\}
		.
		$$
		$\nu_d$-almost everywhere. Therefore,
		$$
		\E_{\tau_n}[(\Psi_{\ell,j}(X_1))^4]
		=
		\E_{0}\!\bigg[
		(\Psi_{\ell,j}(X_1))^4
		\bigg(
		1+\frac{\tau_n(X_1)}{\sqrt n}
		\bigg)
		\bigg]
		\leq
		\frac{\|h_*^0\|_\infty^4}{\ell_L^4}
		\sup_{n\in\N}
		\sup_{u\in\mathbb S^d}
		\bigg|1+\frac{\tau_n(u)}{\sqrt n}\bigg|
		=:
		C_L
		$$
		for all~$\ell\in\mathcal{I}_L$, all~$j=1,\ldots,m_\ell$, and all positive integers~$n$.  
		\vspace{3mm}
		
		(iv)
		Fix a positive integer~$L$.
		For any~$n$, decompose~$S_{L,n}$ into
		$$
		S_{L,n}
		=
		\sum_{\ell\in \mathcal{I}_L}
		\eta_\ell^\psi
		\sum_{j=1}^{m_\ell}
		Y_{\ell,j,n}^2 
		,
		\quad
		\textrm{with }
		Y_{\ell,j,n}
		=
		\frac{1}{\sqrt n}
		\sum_{i=1}^n \Psi_{\ell,j}(X_i)
		.
		$$
		Since~$\eta_\ell^\psi\geq 0$ for all~$\ell$, the Cauchy--Schwarz inequality yields
		$$
		S_{L,n}^2
		\leq
		\Bigg(\sum_{\ell=2}^{L}\eta_\ell^\psi m_\ell\Bigg)
		\Bigg(\sum_{\ell=2}^{L}\eta_\ell^\psi
		\sum_{j=1}^{m_\ell}Y_{\ell,j,n}^4\Bigg)
		.
		$$
		Therefore,
		\begin{equation}
			\label{DecompSupSLn}
			\sup_{n \in \N} \E_{\tau_n}[S_{L,n}^2]
			\leq
			\Bigg(\sum_{\ell=2}^{L}\eta_\ell^\psi m_\ell\Bigg)
			\Bigg(\sum_{\ell=2}^{L}\eta_\ell^\psi
			\sum_{j=1}^{m_\ell}
			\sup_{n \in \N}  \E_{\tau_n}[Y_{\ell,j,n}^4]\Bigg)
			.
		\end{equation}
		For any~$\ell\in \mathcal{I}_L$, any~$j=1,\ldots,m_\ell$, and any~$n$, decompose then
		$$
		Y_{\ell,j,n}
		=
		\frac{1}{\sqrt n}
		\sum_{i=1}^n
		\xi_{i,\ell,j,n}
		+
		\sqrt n\mu_{\ell,j,n}
		,
		$$
		with~$\mu_{\ell,j,n}:=\E_{\tau_n}[\Psi_{\ell,j}(X_1)]$ and~$\xi_{i,\ell,j,n}:=\Psi_{\ell,j}(X_i)-\mu_{\ell,j,n}$.
		The $C_r$-inequality yields
		\begin{equation}
			\label{ajspp}
			\E_{\tau_n}[Y_{\ell,j,n}^4]
			\leq
			8
			\E_{\tau_n}\!
			\bigg[
			\bigg(
			\frac{1}{\sqrt n}\sum_{i=1}^n
			\xi_{i,\ell,j,n}
			\bigg)^4
			\bigg]
			+
			8n^2\mu_{\ell,j,n}^4
			.
		\end{equation}
		Since the $\xi_{i,\ell,j,n}$'s are centered and mutually independent, we have
		\begin{eqnarray*}
			\E_{\tau_n}\!
			\bigg[
			\bigg(
			\sum_{i=1}^n
			\xi_{i,\ell,j,n}
			\bigg)^4
			\bigg]
			&=&
			\sum_{i=1}^n \E_{\tau_n}[\xi_{i,\ell,j,n}^4]
			+
			6\sum_{1\leq i<j\leq n}
			\E_{\tau_n}[\xi_{i,\ell,j,n}^2] \E_{\tau_n}[\xi_{j,\ell,j,n}^2]
			\\[2mm]
			&=&
			n
			\E_{\tau_n}[\xi_{1,\ell,j,n}^4]
			+
			6\binom n2
			(\E_{\tau_n}[\xi_{1,\ell,j,n}^2])^2
			.
		\end{eqnarray*}
		Using the $C_r$-inequality again, then Parts~(ii)--(iii) of the lemma yields that there exists~$C>0$ such that
		$$
		\E_{\tau_n}\!
		\bigg[
		\bigg(
		\sum_{i=1}^n
		\xi_{i,\ell,j,n}
		\bigg)^4
		\bigg]
		\leq
		Cn^2
		$$
		for all~$\ell\in \mathcal{I}_L$, all~$j=1,\ldots,m_\ell$, and all~$n$. Plugging this in~(\ref{ajspp}) and using Part~(i) of the lemma then shows that there exists~$C>0$ such that~$\E_{\tau_n}[Y_{\ell,j,n}^4]\leq C$ for all~$\ell\in \mathcal{I}_L$, all~$j=1,\ldots,m_\ell$, and all~$n$. From~(\ref{DecompSupSLn}), it then follows that
		$$
		\sup_{n \in \N} \E_{\tau_n}[S_{L,n}^2]
		<
		\infty
		,
		$$
		hence that~$\{S_{L,n}\}$ is uniformly integrable.
	\end{proof}

	We can now prove Theorem~\ref{th:localpower}.

	\begin{proof}[Proof of Theorem~\ref{th:localpower}]
		From Theorem~\ref{TheorAsymptDistr} and the discussion below this result, we have that
		$
		\Dstat_n 
		=
		D_{n\ast}^{0}
		+ o_{\rm P} (1)
		$
		under ${\rm P}^{(n)}_{0}$, where
		$$
		D_{n\ast}^{0}
		=
		\frac{1}{n}
		\sum_{i,j=1}^n
		h_*^0(X_i,X_j)
		$$
		involves the kernel~$h_*^0$ in~(\ref{Exprhzerostar}). 
		By the Hilbert--Schmidt expansion of the kernel~$h_*^0$ in~(\ref{HilbertSchmidt}), we have that, for each fixed~$n$,
		the random variables~$S_{L,n}$ introduced in Lemma~\ref{LemLocalPowers} are such that, as~$L\to\infty$, 
		$$
		S_{L,n}
		\to 
		D_{n\ast}^{0}
		\qquad
		\textrm{in}
		\
		L^1({\rm P}_0^{(n)}).
		$$
		Indeed, since
		$\eta_\ell^\psi\geq 0$ for all~$\ell$, $\E_0[\Psi_{\ell,j}(X_1)]=0$, and
		$\E_0[\Psi_{\ell,j}^2(X_1)]=1$,
		$$
		\E_0[|D_{n\ast}^{0}-S_{L,n}|]
		=
		\sum_{\ell>L}\eta_\ell^\psi m_\ell
		\to
		0,
		$$
		because $\sum_{\ell=2}^\infty \eta_\ell^\psi m_\ell=\E_0[h_*^0(X_1,X_1)]<\infty$. Taking $L=n$ yields $D_{n\ast}^{0}-S_{n,n}=o_{\rm P}(1)$ under ${\rm P}_0^{(n)}$, hence also
		$\Dstat_n = S_{n,n} + o_{\rm P} (1)$
		under~${\rm P}^{(n)}_{0}$.
		By applying Le Cam's first lemma, Proposition~\ref{prop:contiguity} entails that~${\rm P}^{(n)}_{\tau_n}$ and ${\rm P}^{(n)}_0$ are mutually contiguous. Therefore, the asymptotic equivalence
		$
		\Dstat_n = S_{n,n} + o_{\rm P} (1)
		$
		holds under~${\rm P}^{(n)}_{\tau_n}$, too. In order to establish the result, it is thus enough to prove that, with obvious notation,
		\begin{equation}
			\label{toprovePgn}
			S_{n,n}
			\stackrel{\mathcal{L}}{\to}
			\sum_{\ell = 2}^{\infty} \eta_\ell^\psi
			\chi^2_{m_\ell}
			\big( \|\Pi_\ell\tau\|_{L^2(\nu_d)}^2 \big)
			\quad        
			\textrm{under }{\rm P}^{(n)}_{\tau_n}
			.
		\end{equation}
		
		Using the same notation as in Lemma~\ref{LemLocalPowers}, we first show that
		\begin{equation}
			\label{eq:Snn}
			\lim_{L \to \infty}
			\sup_{n \geq L} \E_{\tau_n}[ \vert S_{n, n} - S_{L, n}\vert]
			=
			o(1)
			.
		\end{equation}
		Recalling that~$\eta_\ell^\psi \geq 0$ for all~$\ell$, we have that, for any $n > L$,
		\begin{eqnarray*}
			\E_{\tau_n}[ \vert S_{n, n} - S_{L, n}\vert]
			&=&
			\sum_{\ell = L+1}^{n}
			\eta_\ell^\psi
			\sum_{j=1}^{m_\ell}
			\E_{\tau_n}\!
			\Bigg[
			\frac{1}{n}
			\sum_{i, i' = 1}^n
			\Psi_{\ell,j} (X_i)
			\Psi_{\ell,j} (X_{i'})
			\Bigg]
			\\[1mm]
			&=&
			\sum_{\ell = L+1}^{n}
			\eta_\ell^\psi
			\sum_{j=1}^{m_\ell}
			\bigg\{
			\E_{\tau_n}[
			\Psi_{\ell,j}^2 (X_1)
			]
			+
			\frac{n - 1}{n}
			\big(
			\E_{\tau_n}[
			\sqrt{n} \Psi_{\ell,j} (X_1)]
			\big)^2
			\bigg\}
			.
		\end{eqnarray*}
		It directly follows from Lemma~\ref{LemLocalPowers}(i)--(ii) that there exists~$C>0$ such that
		\begin{equation*}
			\E_{\tau_n}[ \vert S_{n, n} - S_{L, n}\vert]
			\leq
			C        \sum_{\ell = L+1}^{n}
			\eta_\ell^\psi m_\ell
		\end{equation*}
		for any~$n>L$, hence so that
		\begin{equation*}
			\sup_{n \geq L} \E_{\tau_n} [ \vert S_{n, n} - S_{L, n}\vert]
			\leq
			C\sum_{\ell = L+1}^{\infty}
			\eta_\ell^\psi m_\ell
			.
		\end{equation*}
		Since $\sum_{\ell = 2}^{\infty} \eta_\ell^\psi m_\ell < \infty$, \eqref{eq:Snn} follows.
		
		Now, fix $\varepsilon > 0$. Denote by $\mathcal{L}_{L,n}$ the distribution of $S_{L,n}$ under~${\rm P}^{(n)}_{\tau_n}$ (in particular, $\mathcal L_{n,n}$ denotes the distribution of~$S_{n,n}$). Hence, there exists~$L_1 > 0$ such that
		\begin{equation}
			\label{eq:Wass1}
			\sup_{n \geq L} W_1 ( \mathcal{L}_{n,n}, \mathcal{L}_{L,n} )
			\leq
			\sup_{n \geq L} \E_{\tau_n} [ \vert S_{n, n} - S_{L, n}\vert]
			<
			\frac{\varepsilon}{3}
			\quad
			\forall L \geq L_1,
		\end{equation}
		where $W_1$ is the Wasserstein distance of order one.
		
		Denote by $\mathcal{L}_{L}$ and $\mathcal{L}_{\infty}$ the distributions given by
		$$
		\sum_{\ell = 2}^{L} \eta_\ell^\psi
		\chi^2_{m_\ell}
		\big( \|\Pi_\ell\tau\|_{L^2(\nu_d)}^2 \big)
		\quad
		\textrm{and}
		\quad
		\sum_{\ell = 2}^{\infty} \eta_\ell^\psi
		\chi^2_{m_\ell}
		\big( \|\Pi_\ell\tau\|_{L^2(\nu_d)}^2 \big),
		$$
		respectively. One easily sees that $\mathcal{L}_{L}$ converges to $\mathcal{L}_{\infty}$ as $L \to \infty$ in the Wasserstein distance of order one, so that there exists~$L_2 > 0$ such that
		\begin{equation}
			\label{eq:Wass2}
			W_1 ( \mathcal{L}_{L}, \mathcal{L}_{\infty} )
			<
			\frac{\varepsilon}{3}
			\quad
			\forall L \geq L_2.
		\end{equation}
		
		Given $L_{\ast}=\max\{L_1,L_2\}$, the finite-dimensional vector
		$$
		Z_{L_{\ast},n}
		=
		\bigg(
		\frac{1}{\sqrt n}\sum_{i=1}^n \Psi_{2,1}(X_i),
		\ldots,
		\frac{1}{\sqrt n}\sum_{i=1}^n \Psi_{L_{\ast},m_{L_{\ast}}}(X_i)
		\bigg)^T
		$$
		is asymptotically standard normal under~${\rm P}_0^{(n)}$. Moreover, Proposition~\ref{prop:contiguity} shows that the log-likelihood ratio of ${\rm P}_{\tau_n}^{(n)}$ with respect to ${\rm P}_0^{(n)}$ is
		$$
		\Lambda_n
		=
		\frac{1}{\sqrt n}\sum_{i=1}^n \tau(X_i)
		-
		\frac12 \E_0[\tau^2(X_1)]
		+
		o_P(1)
		$$
		under ${\rm P}_0^{(n)}$. Since
		$$
		{\rm Cov}_0
		[\Psi_{\ell,j}(X_1),\tau(X_1)]
		=
		\int_{{\mathbb S}^d} \Psi_{\ell,j}(u)\tau(u)
		\,d\nu_d(u)
		=:
		\tau_{\ell,j},
		$$
		Le Cam's third lemma yields that, under ${\rm P}_{\tau_n}^{(n)}$, 
		$Z_{L_{\ast},n}$ is asymptotically normal with identity covariance matrix and mean vector with entries~$\tau_{\ell,j}$, $\ell=2,\ldots,L_{\ast}$, $j=1,\ldots,m_\ell$. Consequently,
		$$
		S_{L_{\ast},n}
		=
		\sum_{\ell=2}^{L_{\ast}}\eta_\ell^\psi
		\sum_{j=1}^{m_\ell}
		\left(
		\frac{1}{\sqrt n}\sum_{i=1}^n \Psi_{\ell,j}(X_i)
		\right)^2
		\stackrel{\mathcal{L}}{\to}
		\sum_{\ell=2}^{L_{\ast}}\eta_\ell^\psi
		\chi^2_{m_\ell}
		\bigg(
		\sum_{j=1}^{m_\ell}\tau_{\ell,j}^2
		\bigg)
		$$
		under ${\rm P}_{\tau_n}^{(n)}$. Since Parseval's identity on~$\mathcal H_\ell$ yields
		$$
		\sum_{j=1}^{m_\ell}\tau_{\ell,j}^2
		=
		\|\Pi_\ell\tau\|_{L^2(\nu_d)}^2,
		$$
		this means that $\mathcal L_{L_{\ast},n}$ converges weakly to $\mathcal L_{L_{\ast}}$.
		Now, Lemma~\ref{LemLocalPowers}(iv) implies that~$\{S_{L_{\ast},n}\}$ is uniformly integrable under~$P_{\tau_n}^{(n)}$, so that
		we have the convergence for the corresponding first moments, too.
		Since weak convergence together with convergence of first moments is equivalent to convergence in $W_1$, we obtain that~$
		W_1(\mathcal L_{L_{\ast},n},\mathcal L_{L_{\ast}})\to0$.
		Thus, there exists $n_0>0$ such that
		\begin{equation}
			\label{eq:Wass3}
			W_1 ( \mathcal{L}_{L_{\ast},n}, \mathcal{L}_{L_{\ast}} )
			<
			\frac{\varepsilon}{3}
			\quad
			\forall n \geq n_0.
		\end{equation}
		
		Finally, taking $n_{\ast} = \max \{ n_0, L_{\ast} \}$, the triangle inequality and Equations~\eqref{eq:Wass1}--\eqref{eq:Wass3} ensure that, for any $n \geq n_{\ast}$,
		\begin{equation*}
			W_1 ( \mathcal{L}_{n ,n}, \mathcal{L}_{\infty} )
			\leq
			W_1 ( \mathcal{L}_{n ,n}, \mathcal{L}_{L_{\ast} ,n} )
			+
			W_1 ( \mathcal{L}_{L_{\ast} ,n}, \mathcal{L}_{L_{\ast}} )
			+
			W_1 ( \mathcal{L}_{L_{\ast}}, \mathcal{L}_{\infty} )
			< \varepsilon.
		\end{equation*}
		Since convergence in~$W_1$ implies convergence in distribution, the result is proved.
	\end{proof}
	
	Finally, we prove Proposition~\ref{prop:blind}.

	\begin{proof}[Proof of Proposition~\ref{prop:blind}]
		(i) For any~$x\in\mathbb{S}^d$, let
		$$
		\tau_n^\ast(x)
		:=
		\sqrt n
		\bigg\{
		\bigg(
		\frac{1-\|\phi_n\|^2}{\|x-\phi_n\|^2}
		\bigg)^d
		-
		1
		\bigg\}
		,
		\qquad
		\textrm{with }
		\phi_n
		:=
		\frac{v}{2d\sqrt n}
		.
		$$
		For all sufficiently large~$n$, $\phi_n\in\mathbb{B}^{d+1}$ and, by definition of the spherical Cauchy density,
		${\rm Q}_{v}^{(n)}={\rm P}^{(n)}_{\tau_n^\ast}$. Moreover, uniformly in~$x\in\mathbb{S}^d$,
		\begin{eqnarray*}
			\tau_n^\ast(x)
			&=&
			\sqrt n
			\bigg\{
			\bigg(
			\frac{1-\|v\|^2/(4d^2n)}
			{1-2x^T v/(2d\sqrt n)+\|v\|^2/(4d^2n)}
			\bigg)^d
			-
			1
			\bigg\}
			\\[2mm]
			&=&
			v^Tx+o(1)
			\\[2mm]
			&=&
			\tau(x)+o(1)
			,
		\end{eqnarray*}
		as~$n\to\infty$. Thus, the sequence~$(\tau_n^\ast)$ satisfies~\eqref{eq:taun1}--\eqref{eq:taun2} with the same limiting function~$\tau$ as the sequence~$(\tau_n)$.
		Therefore, applying Proposition~\ref{prop:contiguity} to~$(\tau_n)$ and to~$(\tau_n^\ast)$ yields, under~${\rm P}_0^{(n)}$,
		\begin{eqnarray*}
			\lefteqn{
				\log \bigg( \frac{d{\rm P}^{(n)}_{\tau_n}}{d{\rm Q}^{(n)}_{v}} \bigg)
				=
				\log \bigg( \frac{d{\rm P}^{(n)}_{\tau_n}}{d{\rm P}^{(n)}_0} \bigg)
				-
				\log \bigg( \frac{d{\rm Q}^{(n)}_{v}}{d{\rm P}^{(n)}_0} \bigg)
			}
			\\[1mm]
			& &
			\hspace{5mm}
			=
			\bigg\{
			\frac{1}{\sqrt n}\sum_{i=1}^n \tau(X_i)
			-\frac12\E_0[\tau^2(X_1)]
			\bigg\}
			-
			\bigg\{
			\frac{1}{\sqrt n}\sum_{i=1}^n \tau(X_i)
			-\frac12\E_0[\tau^2(X_1)]
			\bigg\}
			+
			o_{\rm P}(1)
			\\[1mm]
			& &
			\hspace{5mm}
			=
			o_{\rm P}(1)
			,
		\end{eqnarray*}
		as~$n\to\infty$. Proposition~\ref{prop:contiguity} and Le Cam's first lemma also imply that
		${\rm P}^{(n)}_{\tau_n}$ and~${\rm Q}^{(n)}_{v}$ are both mutually contiguous with respect to~${\rm P}^{(n)}_0$. Hence, the last display holds under~${\rm Q}^{(n)}_{v}$ and under~${\rm P}^{(n)}_{\tau_n}$, too. This proves~\eqref{eq:PgnQv0equiv}.
		\vspace{3mm}
		
		(ii) Fix~$n$ large enough to have~$\phi_n\in\mathbb{B}^{d+1}$  and assume that~$X_1,\ldots,X_n$ are \mbox{i.i.d.} with distribution~${\rm Q}^{(n)}_{v}$, that is, with distribution~${\rm SC}(\phi_n)$. Since
		$$
		\log \bigg( \frac{d{\rm P}^{(n)}_{\tau_n}}{d{\rm Q}^{(n)}_{v}} \bigg)
		=
		o_{\rm P}(1)
		$$
		as~$n\to\infty$ 
		under~${\rm Q}^{(n)}_{v}$, Le Cam's third lemma entails that~$T_n$ has a limiting distribution under~${\rm P}^{(n)}_{\tau_n}$ if and only if it has one under~${\rm Q}^{(n)}_{v}$, and that the two limiting distributions are then equal. Since invariance implies that the (limiting) distribution of~$T_n$ under~${\rm Q}^{(n)}_{v}$ is the same as under~${\rm P}^{(n)}_0$, the result follows.
	\end{proof}
	
	\bibliographystyle{apalike}
	\bibliography{Paper.bib}

\end{document}